\documentclass[11 pt]{amsart}

\usepackage{amssymb,amsfonts}
\usepackage{ dsfont }

\usepackage{verbatim}
\usepackage[all,arc]{xy}
\usepackage{enumerate}
\usepackage{mathrsfs}
\usepackage[hidelinks]{hyperref}
\usepackage{graphicx}
\usepackage{mathtools}
\usepackage[dvipsnames]{xcolor}
\usepackage{fixme}
\usepackage{physics}
\fxsetup{status=draft}
\usepackage{svg}
\usepackage{tikz}
\usepackage[toc]{appendix}
\usepackage{caption}
\usepackage{subcaption}
\usepackage{multicol}
\usepackage[inline]{enumitem}

\usepackage{pgfplots}
\pgfplotsset{compat=newest}

\usetikzlibrary{3d,calc}

\usepackage[a4paper, margin=1.5in]{geometry}

\def \R {\mathbb R }
\def \N {\mathbb N }

\def \Z {\mathbb Z}

\def \C{\mathbb C }
\def \P{\mathbb P }

\def \cN{\mathcal N }
\def \cF{\mathcal F}

\renewcommand{\angle}[1]{ \langle #1 \rangle}
\newcommand{\set}[1]{\left\lbrace #1 \right\rbrace}
\renewcommand{\phi}{\varphi}
\newcommand\restr[2]{{%
		\left.\kern-\nulldelimiterspace %
		#1 %
		\vphantom{\big|} %
		\right|_{#2} %
}}
\newtheorem{thm}{Theorem}[section]
\newtheorem{cor}[thm]{Corollary}
\newtheorem{prop}[thm]{Proposition}
\newtheorem{lem}[thm]{Lemma}

\newtheorem{claim}[thm]{Claim}

\theoremstyle{definition}
\newtheorem{defn}[thm]{Definition}

\newtheorem{notn}[thm]{Notation}

\newtheorem{ques}[thm]{Question}

\newtheorem{rmk}[thm]{Remark}

\theoremstyle{remark}

\DeclareMathOperator{\Amp}{Amp}
\DeclareMathOperator{\Nef}{Nef}
\DeclareMathOperator{\Psef}{Psef}
\DeclareMathOperator{\Kah}{Kah}
\DeclareMathOperator{\NS}{NS}
\DeclareMathOperator{\Aut}{Aut}
\DeclareMathOperator{\Prob}{Prob}
\DeclareMathOperator{\supp}{supp}
\DeclareMathOperator{\GL}{GL}

\DeclareMathOperator{\Orth}{O}
\DeclareMathOperator{\Vol}{Vol}
\DeclareMathOperator{\Dens}{Dens}

\makeatletter
\let\c@equation\c@thm
\makeatother
\numberwithin{equation}{section}

\usepackage[numbers]{natbib}
\usepackage{graphicx}

\renewcommand{\epsilon}{\varepsilon}

\title[Maximal entropy for random skew products]{The measure of maximal entropy for random skew products on compact complex surfaces}
\author[Ethan Cohen]{Ethan Cohen}
\address{Department of Mathematics, Yale University}
\email{ethan.cohen@yale.edu}

\makeatletter
\def\ps@headings{%
	\def\@evenhead{\normalfont\scriptsize\thepage\hfil ETHAN COHEN\hfil}%
	\def\@oddhead{\normalfont\scriptsize\hfil MAXIMAL ENTROPY FOR RANDOM SKEW PRODUCTS\hfil\thepage}%
	\def\@evenfoot{}%
	\def\@oddfoot{}%
}
\makeatother

\usepackage{microtype}
\begin{document}
	\begin{abstract}
		Let $X$ be a compact complex surface. We prove that the skew product associated to a Borel probability measure $\mu$ on $\Aut(X)$ admits a unique invariant measure of maximal fiber entropy, assuming that $\mu$ satisfies a logarithmic integrability condition and that $\supp(\mu)$ generates a non-elementary subgroup of $\Aut(X)$. We describe this measure canonically in terms of the random limit currents constructed by Cantat and Dujardin, and show that its fiber entropy is equal to the Furstenberg exponent of the associated random action on cohomology. Under an exponential moment assumption, we prove that it is mixing.
	\end{abstract}
	\maketitle
	\pagestyle{headings}
	\tableofcontents
	\section{Introduction}
	
	\subsection{Motivation and main result}

		Let $f$ be a biholomorphism of a compact complex surface. It is well known that if $f$ has positive entropy, then there exist nonzero closed positive $(1,1)$-currents $T^+_f, T^-_f$ satisfying 
	\[f^* T^+_f = \lambda T^+_f, \hspace{1 cm} f^* T^-_f = \lambda^{-1} T^-_f\] 
	where $ \lambda \in \R_{>1}$ is the spectral radius of the action of $f^*$ on cohomology, and such currents are unique up to scaling \cite{MR1864630, MR2137979, MR2629598}.
	
	  The currents $T^\pm_f$ have continuous potentials, so their wedge \[m_f\coloneqq T^+_f \wedge T^-_f\] is well-defined in the sense of Bedford--Taylor \cite{MR445006} (see Section \ref{sec:wedges}). The measure $m_f$ is a nonzero finite Radon measure on $X$ that we assume to be normalized to a probability measure, and it is the unique $f$-invariant probability measure having maximal entropy \cite{MR1864630,MR2219241}. By theorems of Gromov \cite{MR2026895} and Yomdin \cite{MR889979}, $f$ has topological entropy $h_{top}(f) = \ln \lambda$, so the variational principle \cite[Theorem~4.5.3]{MR1326374} implies $h_{m_f}(f) = \ln \lambda$. Similar results hold in other complex dynamical systems, including polynomial diffeomorphisms of $\C^2$ (see \cite{MR1207478}), more general birational maps of surfaces (see \cite{MR2219241}), and some higher-dimensional settings (see \cite{MR2137979}, \cite{MR2629598}).

		The goal of this article is to establish a random analog of the theorem of \cite{MR1864630,MR2219241}, that a positive-entropy automorphism of a compact complex surface admits a unique measure of maximal entropy. Precisely, let $X$ be a compact complex surface. Endow the group of biholomorphisms of $X$, denoted $\Aut(X)$, with the topology of uniform convergence, and consider a Borel probability measure $\mu$ on $\Aut(X)$. Define the skew product associated to $\mu$ by
 \begin{align*}	F\colon \Omega \times X \to \Omega \times X, \quad (\omega, x) \mapsto (\sigma(\omega), \omega_0(x))
 \end{align*}
	 where $\Omega \coloneqq(\supp \mu)^{\mathbb Z}$, $\sigma\colon \Omega \to \Omega$ is the left shift, $\omega = (\dots, \omega_{-1}, \omega_0, \omega_1, \dots) \in \Omega$, and $x\in X$. In Section \ref{sec:FiberEntropy}, we recall the notion of \emph{fiber entropy} $h_{\nu}^\mathcal{F}$ for an $F$-invariant Borel probability measure $\nu$ on $\Omega \times X$.

Assume that $\Gamma_\mu\coloneqq \angle{\supp \mu}$ is a non-elementary subgroup of $\Aut(X)$ (see Section \ref{sec:setup}), meaning that $\Gamma_\mu$ contains a non-abelian free group, and that $\mu$ satisfies the moment condition
	\begin{align} \label{int}
		\int (\ln \norm{f}_{C^1} + \ln \norm{f^{-1}}_{C^1})  \,d\mu(f) < \infty.
		\end{align}
	From the associated random action on cohomology, we obtain a Furstenberg exponent, $\lambda_\mu$, and almost-everywhere defined Furstenberg limit classes in the boundary of the Nef cone (see Section \ref{sec:cohomology}).
  	
	In \cite{MR4635340}, Cantat and Dujardin prove that there exist unique unit-mass closed positive $(1,1)$-currents $T_\omega^+, T_\omega^-$ whose cohomology classes are the forward and backward Furstenberg limit classes determined by $\omega$
	(see Section \ref{sec:cohomology}). Moreover, they demonstrate that the $T_\omega^{\pm}$ have continuous potentials. Therefore, we may consider their wedge $T_{\omega}^+ \wedge T_{\omega}^-$ in the Bedford--Taylor sense. %
	The classes $[T_\omega^+]$ and $[T_\omega^-]$ are distinct almost surely, so $T_{\omega}^+ \wedge T_{\omega}^-$ defines a nonzero Radon measure on $X$. Let $m_\omega$ denote the normalization of this measure, that is,
	\begin{align*}
		m_\omega\coloneqq \frac 1 {\int_X T_{\omega}^+ \wedge T_{\omega}^-} T_{\omega}^+ \wedge T_{\omega}^- .
	\end{align*}
The measurability of the assignment $\omega\mapsto m_\omega$ is demonstrated in Lemma \ref{lem:m_omega_measurability}.
We prove the following:
\begin{thm}\label{thm:main} Let $X$ be a compact complex surface and $\mu$ a Borel probability measure on $\Aut(X)$ satisfying the moment condition \eqref{int} and such that $\Gamma_\mu\coloneqq \angle{\supp \mu}$ is non-elementary.
		Then the measure 
		\begin{align*}
			m\coloneqq \int_{\Omega}( \delta_{\omega}\otimes m_\omega)  \,d\mu^\Z(\omega)
		\end{align*}
	is the unique one in the class 
	\begin{align*}
		S(\mu)  \coloneqq  \set{\nu\in \Prob(\Omega \times X)\colon F_*\nu = \nu \text{ and } p_*\nu = \mu^\Z} 
	\end{align*}
	with maximal fiber entropy. Moreover, $m$ has fiber entropy $h_{m}^\mathcal{F} = \lambda_\mu$. 
\end{thm}

\begin{rmk}
	$F$-invariance of $m$ follows immediately from the equivariance property of the currents $T^\pm_\omega$ (see Remark \ref{equivariance}).
\end{rmk}
	As we will explain in the next section, the new aspect of Theorem \ref{thm:main} is the uniqueness of the maximizing measure and its identification in terms of the Cantat--Dujardin limit currents.
  	
\subsection{Some remarks}

The following theorem is a combination of several results: a variational principle on $S(\mu)$ due to Bogensch\"utz 
\cite{MR1181382} and Kifer \cite{MR1819189}, an extension of Yomdin's theorem to the random setting due to Yomdin and Kifer \cite{MR970566}, and the observation that Gromov's theorem applies in the random setting (see Appendices \ref{app:entropy} and \ref{app:gromov}).

\begin{thm} \label{thm:var}
	We have
	\begin{align*}
			\sup_{\nu \in S(\mu)}h_\nu^\mathcal{F} = \lambda_\mu.
	\end{align*}
\end{thm}

	By compactness of $S(\mu)$ (see Remark \ref{compactness}) and upper-semicontinuity of the entropy function $\nu\mapsto h_\nu^\mathcal{F}$ \cite[Lemma~2.1]{MR1819189}, the supremum in Theorem \ref{thm:var} is achieved.  Thus, there exists a measure in $S(\mu)$ with maximal fiber entropy.

\begin{rmk}
	By Cauchy's estimate, the $C^1$-norm in \eqref{int} can be replaced by the $C^k$-norm for any $k\geq 1$. 
\end{rmk}
\begin{rmk}
	Recall that a $\mu$-stationary Borel probability measure $\tilde \nu$ on $X$ induces a unique $F$-invariant Borel probability measure $\nu$ on $\Omega \times X$ that pushes forward to $\tilde \nu \times \mu^\N$ under the projection $\Omega \times X \to \Omega_+ \times X$ \cite[Proposition~I.1.2]{MR1369243}. Here $\Omega_+\coloneqq(\supp \mu)^\N$ denotes the space of futures. The set $S(\mu)$ therefore contains those measures on $\Omega \times X$ obtained from $\mu$-stationary measures on $X$.
\end{rmk}
	\subsection{Proof outline for Theorem \ref{thm:main}}\label{sec:MainOutline}
	The proof follows the same strategy used to show that a single positive-entropy automorphism of a K\"ahler surface admits a unique measure of maximal entropy \cite{MR1207478,MR1864630,MR2219241}. 
	
		We fix a measure $\nu\in S(\mu)$ satisfying $h_{\nu}^\mathcal{F} = \lambda_\mu$; by the ergodic decomposition, we may assume that $\nu$ is ergodic. Then we consider a measurable partition $\eta$ subordinate to the unstable manifolds of $\nu$ (see Definition \ref{def:subordinate}) such that $\eta$ is a refinement of the partition into fibers of the projection $p\colon \Omega\times X \to \Omega$ and 
	\begin{align*}
		h_{\nu}^\mathcal{F} = H( F^{-1} \eta \mid \eta);
	\end{align*} such a partition was developed in the proof of the Ledrappier-Young entropy formula \cite{MR819556} for a single automorphism and later extended to the random setting \cite{MR1369243,MR2032494}. 
	
	Let $\eta_\omega(x)$ be as in Definition \ref{def:subordinate}. The conditional measures of $\nu$ with respect to $\eta$, denoted $\nu_{\eta(\omega, x)}$, satisfy
	\begin{align*}
		\nu_{\eta(\omega, x)} = \delta_\omega \otimes \nu_{\eta_\omega(x)}
	\end{align*} for some Borel probability measure $\nu_{\eta_\omega(x)}$ on $X$. 
		The key step is to prove that conditionals of $\nu$ on $\eta$ are proportional to so-called slices of $T_\omega^+$ (see Section \ref{sec:slices}) on the atoms of $\eta$, meaning
	\begin{align}\label{equivalence}
		\nu_{\eta_\omega(x)}  = \cN\big(\restr{T_\omega^+}{\eta_\omega(x)}\big)
	\end{align} $\nu$-almost everywhere, where $\cN(\cdot)$ denotes the normalization of a finite positive measure. This is encouraging because, if we assume $m$ is hyperbolic (which is not a priori true), then $m$ also has the slice property in \eqref{equivalence} by the product structure of $m$ in Pesin boxes (see Appendix \ref{app:Product}). 
		
		Using the ergodic theorem, we will prove that, for $\nu$-generic $(\omega, x)$, the sequence 
\begin{align*}
			\frac{1}{N} \sum_{n=0}^{N-1} F_*^n\nu_{\eta(\omega, x)}
\end{align*} converges weakly to $\nu$ in the sense of Section \ref{sec:setup}, with $f_\omega^n$ as in Notation \ref{notn:product}. By \eqref{equivalence}, we have 
	\begin{align}\label{subseq}
	\frac{1}{N} \sum_{n=0}^{N-1} F_*^n\nu_{\eta(\omega, x)} &= \frac{1}{N} \sum_{n=0}^{N-1} \delta_{\sigma^n(\omega)}\otimes (f_\omega^n)_*\nu_{\eta_\omega(x)}\notag \\
	&= \frac{1}{N} \sum_{n=0}^{N-1} \delta_{\sigma^n(\omega)}\otimes (f_\omega^n)_*\cN\big(\restr{T_\omega^+}{\eta_\omega(x)}\big).
	\end{align} We will prove that a subsequence of the averages in \eqref{subseq} converges to $m$ as well; hence, $\nu = m$.

		Assume for simplicity that $\eta_\omega(x)$ is a holomorphically embedded closed disk in $X$ and $\psi\colon X\to \R_{\geq 0}$ is smooth and satisfies $(\supp \psi) \cap \partial \eta_\omega(x) = \varnothing$. It will follow from Proposition \ref{lem:last2} that, for all $\rho > 0$, there exists a subsequence $(n_i)_{i\in \N}$ of $\N$ with upper-density (see Definition \ref{defn:density}) greater than $1 - \rho$ along which 
\begin{align}\label{meas1}
	(f_\omega^{n_i})_*m_\omega - (f_\omega^{n_i})_*\cN(T_\omega^+\wedge \psi [\eta_\omega(x)])\to 0
	\end{align} as $i\to \infty$ in the weak sense of measures, where the wedge $(T_\omega^+\wedge \psi [\eta_\omega(x)])$ is defined as in Section \ref{sec:wedges}. The proof of \eqref{meas1} in Proposition \ref{lem:last2} passes through approximation at the level of currents. That is, the currents
	$(f_\omega^{n_i})_*(\psi[\eta_\omega(x)])$ are well-approximated, up to normalization, by $T_{\sigma^{n_i}(\omega)}^-$ when $i$ is large.

By the definition of upper-density, \eqref{meas1} implies that there exists a sequence $N_i\to \infty$ along which, roughly speaking,
\begin{align*}
	\frac{1}{N_i}\sum_{n=0}^{N_i -1} \delta_{\sigma^n(\omega)}\otimes (f_\omega^n)_*m_\omega -  \frac{1}{N_i}\sum_{n=0}^{N_i -1} \delta_{\sigma^n(\omega)}\otimes (f_\omega^n)_*\cN(T_\omega^+\wedge \psi [\eta_\omega(x)]) \approx 0
\end{align*} for $i$ large, where the approximation depends on $\rho$. 

By the ergodic theorem,
\begin{align*}
	\frac{1}{N_i}\sum_{n=0}^{N_i -1} \delta_{\sigma^n(\omega)}\otimes (f_\omega^n)_*m_\omega \to m
\end{align*} as $i\to \infty$; thus,
\begin{align*}
	&   m - \frac{1}{N_i}\sum_{n=0}^{N_i -1} \delta_{\sigma^n(\omega)}\otimes (f_\omega^n)_*\cN(T_\omega^+\wedge \psi [\eta_\omega(x)]) \approx 0
\end{align*} 
	for $i$ large. Here again the approximation depends on $\rho$. Taking $\rho \to 0$ and replacing the wedge $T_\omega^+\wedge \psi[\eta_\omega(x)]$ by the slice $\restr{T_\omega^+}{\eta_\omega(x)}$, as justified in Step 2 of the proof of Theorem \ref{thm:main}, we obtain 
	\begin{align*}
	 \frac{1}{N_i} \sum_{n=0}^{N_i-1} \delta_{\sigma^n(\omega)}\otimes (f_\omega^n)_*\cN\big(\restr{T_\omega^+}{\eta_\omega(x)}\big)\to m,
\end{align*} so $m = \nu$. 

\subsection{Acknowledgements}
The author would like to thank his advisor, Sebastian Hurtado, for guidance throughout this project. The author is also grateful to Romain Dujardin for explaining aspects of his work on laminar currents and geometric intersections.
	
	\section{Preliminaries}
	For a much more comprehensive introduction to random dynamical systems on compact complex surfaces, we refer the reader to \cite{MR4635340}.

	\subsection{Notation and setup}\label{sec:setup}
	So far, we have stated our results for compact complex surfaces. However, any compact complex surface admitting an automorphism of positive entropy is K\"ahler \cite{MR1864630}. In fact, the existence of a non-elementary subgroup of $\Aut(X)$ forces $X$ to be projective \cite[Theorem~E]{MR4635340}. Henceforth, we specialize to the projective setting.

	Fix once and for all a K\"ahler form $\kappa_0$ with $\int_X \kappa_0 \wedge \kappa_0 = 1$; the form $\kappa_0$ induces a Riemannian metric on $X$. Given $a\in H^{1,1}(X, \R)$ define the \emph{mass} of $a$ to be $M(a)\coloneqq \int_X a\wedge \kappa_0$, and take the mass of any positive $(1,1)$-current $C$ to be $(C\mid \kappa_0)$. We also fix a Euclidean norm $\norm{\cdot}$ on $H^{1,1}(X,\R)$ to use for the remainder of the paper, which we will later assume is adapted to the intersection form. %
	
	Let $\Kah(X)$ denote the K\"ahler cone, $\NS(X, \Z)$ the N\'eron-Severi group, 
		$\NS(X) \coloneqq \NS(X, \Z) \otimes_\Z \R$, $\Amp(X)\coloneqq \NS(X)\cap \Kah(X)$ the ample cone, and $\Nef(X) \coloneqq  \NS(X)\cap \overline{\Kah(X)}$ the nef cone. Let $\Psef(X)$ denote the pseudo-effective cone, consisting of cohomology classes of closed positive currents. We note the inclusions
	\begin{align*}
		\Amp(X)\subset \Nef(X) \subset \Psef(X).
	\end{align*}

	\begin{rmk}
		Mass is positive on $\Kah(X)$, and thus positive on $\Amp(X)$ and nonnegative on $\Nef(X)$. Moreover, mass is nonnegative on $\Psef(X)$ since $\kappa_0$ pairs nonnegatively with positive currents. 
	\end{rmk}

	\begin{rmk}\label{cpt}
		The collection of classes with unit mass in $\Psef(X)$ is compact. Combined with the 
		scaling equivariance of mass and $\norm{\cdot}$,  we find that there exists $C> 0$ such that $\frac 1 CM(a)\leq \norm{a} \leq C M(a)$ for all $a\in \Psef(X)$.
	\end{rmk}
	
		Throughout, $\Prob(X)$ and $\Prob(\Omega \times X)$ denote the sets of Borel probability measures on their respective spaces. We equip $\Omega$ and $\Omega\times X$ with the product topologies. Given $\mu_n, \mu \in \Prob(X)$, we say $\mu_n \to \mu$ weakly if
		\[\int f\,d\mu_n \to \int f\,d\mu\] for all $f\in C(X)$. Given $\nu_n, \nu \in \Prob(\Omega \times X)$, we say $\nu_n\to \nu$ weakly if 
		\[\int f \,d \nu_n \to \int f \,d \nu\] for all $f\in C_b(\Omega\times X)$, the set of bounded continuous functions. 
		
		We will say that a function from $\Omega$ into a measure space is $\mu^\Z$-measurable if it is measurable with respect to the $\mu^\Z$-completion of the Borel $\sigma$-algebra on $\Omega$. Moreover, once a measure $\nu\in\Prob(\Omega\times X)$ is fixed, $\nu$-measurable means measurable with respect to the $\nu$-completion of the Borel $\sigma$-algebra on $\Omega\times X$. For maps into a topological space, the target is endowed with its Borel $\sigma$-algebra unless otherwise stated.

	  	\begin{rmk}\label{compactness}
		Let $L^1(\Omega, C(X))$ denote the collection of Borel-measurable functions $f\colon \Omega \times X \to \R $ such that $f_\omega\coloneqq \restr{f}{\set{\omega} \times X}$ is continuous for almost every $\omega$ and $\int_\Omega \norm{f_\omega}_{C^0}\,d\mu^\Z(\omega) < \infty$. Since $p_*\nu=\mu^\Z$, we have $L^1(\Omega, C(X)) \subset L^1(\nu)$ for every $\nu \in S(\mu)$. For $\nu_n, \nu \in S(\mu)$, write $\nu_n \xrightarrow{*} \nu$ if, for each $f\in L^1(\Omega, C(X))$,
		\begin{align*}
			\int f \,d \nu_n \to \int f  \,d\nu.
		\end{align*} Then $S(\mu)$ is compact in this topology \cite[Lemma~2.1]{MR1819189}.
	\end{rmk}

\begin{rmk}\label{disintegration}
	For each $\nu\in S(\mu)$, since $\Omega$ and $X$ are standard Borel spaces, there exists a Borel measurable assignment $\Omega \to \Prob(X)$, where $\Prob(X)$ is endowed with the weak topology, such that
	\begin{align*}
		\nu=\int_\Omega (\delta_\omega\otimes \nu_\omega)\,d\mu^\Z(\omega)
	\end{align*}
	\cite[Theorem~8.5]{MR1876169}. The measures $\nu_\omega$ are the conditional measures of $\nu$ along the fibers of $p:\Omega\times X\to\Omega$, and are uniquely determined for $\mu^\Z$-almost every $\omega$.

\end{rmk}

	\subsection{Action on cohomology} \label{sec:actionCohomology}
	 $\Aut(X)$ acts on the cohomology groups of $X$ via pull-back, preserving the Hodge decomposition. 
	By the Hodge Index Theorem, the intersection form 
	\begin{align*}
		\angle{\alpha, \beta}\coloneqq \int_X \alpha \wedge \beta
	\end{align*}
  has signature $(1, \dim H^{1,1} - 1)$ on $H^{1,1} \coloneqq H^{1,1}(X, \R)$. Let $\Orth(H^{1,1})$ denote the group of linear transformations of $H^{1,1}$ that preserve the intersection form, and $\Orth^+(H^{1,1})$ the subgroup of those preserving the positive cone. Since pullbacks preserve both the intersection form and the positive cone, we get an anti-homomorphism $\pi:\Aut(X) \to \Orth^+(H^{1,1})\cong \Orth^+(1, \dim H^{1,1} - 1)$ defined by $\pi(f) \coloneqq f^*$. 
 
 Let $\Gamma_\mu$ denote the subgroup of $\Aut(X)$ generated by the support of $\mu$, and $\Gamma_\mu^* $ denote the image of $\Gamma_\mu$ under $\pi$. We assume that $\Gamma_\mu$ is \emph{non-elementary}, meaning that the following equivalent conditions hold \cite[Appendix~A]{MR4635340}:
 \begin{enumerate}
 	\item $\Gamma_\mu$ contains a non-abelian free subgroup in which every non-identity element has positive entropy.
 	\item $\Gamma^*_\mu$ contains a non-abelian free subgroup.
 \end{enumerate}
	Unless otherwise stated, we assume throughout that $\mu$ satisfies \eqref{int} and that $\Gamma_\mu$ is non-elementary.

		 By projectivity of $X$, the restriction of the intersection form to $\NS(X)$ has signature $(1, \dim \NS(X) - 1)$.
	By Proposition 2.11 of \cite{MR4635340}, there exists a $\Gamma_\mu$-invariant subspace $\NS(X)_+\subset \NS(X)$ on which $\angle{\cdot,  \cdot}$ has signature $(1, \dim \NS(X)_+ - 1)$ and the action of $\Gamma_\mu^*$ on $\NS(X)_+$ is \emph{strongly irreducible}, meaning that there exists no $\Gamma_\mu^*$-invariant finite union of proper subspaces of $\NS(X)_+$. Furthermore, the action remains non-elementary on $\NS(X)_+$, meaning the image of $\Gamma_\mu$ in $\GL(\NS(X)_+)$ contains a non-abelian free subgroup. Finally, the intersection form is negative definite on $\NS(X)_+^\perp \subset H^{1,1}(X,\R)$. 
		
		Throughout, we let $\mathbb{H}_\mu$ denote the connected component of
		\begin{align*}
			\set{v\in \NS(X)_+ \colon \angle{v,v} = 1}
		\end{align*} intersecting the K\"ahler cone. Then $\mathbb{H}_\mu$ is a hyperbolic space on which $\Gamma_\mu$ acts by isometries. We view its boundary $\partial \mathbb{H}_\mu$ as a subset of $\P(\NS(X)_+)$. 
	\subsection{Adapted inner product}\label{sec:adapted}
	 We assume from now on that the norm $\norm{\cdot}$ on $H^{1,1}$ is adapted to the intersection form. That is, fix $v\in \mathbb H_\mu$ and define a positive-definite inner product by $\angle{w,u}_{\mathrm{ad}}\coloneqq 2\angle{w,v}\angle{u,v} - \angle{w,u}$. 
	 Then we let $\norm{w}^2 =  \angle{w,w}_{\mathrm{ad}}$ be the norm associated to $\angle{\cdot, \cdot}_{\mathrm{ad}}$. 
	 
	 Let $d\coloneqq \dim H^{1,1}$ and choose $e_1, \dots, e_{d-1}\in H^{1,1}$ such that $\angle{e_i, e_i} = -1$, $\angle{e_i ,v } = 0$, and $\angle{e_i, e_j} = 0$ for $i\neq j$. Notice that the basis $(v, e_1, \dots, e_{d-1})$ is orthonormal with respect to $\angle{\cdot, \cdot}_{\mathrm{ad}}$. In that basis, any $B\in \Orth^+(H^{1,1})$ admits a Cartan decomposition as $B= k_1Ak_2$ where $k_i$ are orthogonal and 
	 \begin{align*}
	 	A = \mqty(
	 	\cosh t & \sinh t & 0 \\
	 	\sinh t & \cosh t & 0 \\
	 	0 & 0 & I_{d-2}
	 	)
	 \end{align*} for $t = d_{\mathbb H^{1,1}} (v, Bv)$, where $d_{\mathbb H^{1,1}}$ represents hyperbolic distance in the component of $\set{w\in H^{1,1} \colon \angle{w, w} = 1}$ intersecting the K\"ahler cone. In particular, $\norm{B} = \norm{B^{-1}} = e^{d_{\mathbb H^{1,1}} (v, Bv)}$ where $\norm{B}$ denotes the operator norm induced by $\norm{\cdot}$.

	 Throughout, if $A$ is a linear endomorphism of $H^{1,1}$, we write $A^t$ for the adjoint of $A$ with respect to $\angle{\cdot,\cdot}_{\mathrm{ad}}$. Equivalently, $A^t$
	 is the transpose of the matrix of $A$ in the orthonormal basis
	 $(v,e_1,\dots,e_{d-1})$.
	 
	 The expanding and contracting singular vectors of $B$ have self-intersection zero. Indeed, the eigenspaces of $A$ are spanned by 
	 \begin{align*}
		u_+\coloneqq v + e_1, \hspace{1 cm } u_-\coloneqq v- e_1
	\end{align*} which are orthogonal for the adapted inner product. Then $k_2^{-1}u_+, k_2^{-1}u_-$ span the expanding and contracting singular directions for $B$, respectively. 
So \[\angle{k_2^{-1}u_\pm, k_2^{-1}u_\pm} = \angle{u_\pm, u_\pm} = 0.\] 

Given $B\in \Orth^+(H^{1,1})$ with $\norm{B}>1$, we let $e_+(B), e_-(B)$ denote the unique representatives of the expanding and contracting singular directions, respectively, with unit norm and lying in the closure of the positive cone. 

If $B$ preserves $\NS(X)_+$, then it follows from the Cartan decomposition that the expanding and contracting singular vectors of $B$ lie in $\NS(X)_+$. To see this, notice
\begin{align*}
	H^{1,1} = \NS(X)_+ \oplus \NS(X)_+^\perp
\end{align*} where  $\NS(X)_+^\perp$ denotes the perpendicular space with respect to the intersection form. Firstly, this decomposition is also orthogonal with respect to $\angle{\cdot, \cdot}_{\mathrm{ad}}$. Secondly, $\angle{u, w}_{\mathrm{ad}} = -\angle{u, w}$ whenever $u,w\in \NS(X)_+^\perp$, and hence the restriction of the adapted inner product to $\NS(X)_+^\perp$ is $\Gamma_\mu$-invariant. It follows from those two facts that singular directions corresponding to singular values different from $1$ lie in $\NS(X)_+$.

	\subsection{Potentials}\label{sec:Potentials}
	
	For the rest of the paper, we fix K\"ahler forms $\kappa_1, \dots, \kappa_n$ whose cohomology classes form a basis of $H^{1,1}$ and fix a volume form $V$ on $X$ with $\int_X V = 1$. We choose the basis so that one of the $\kappa_i$ is $\kappa_0$. Any class $\alpha\in H^{1,1}$
	can be written uniquely as a linear combination $\alpha = \sum c_i [\kappa_i]$.
	Let \[\Theta(\alpha)  \coloneqq \sum_{i=1}^n c_i \kappa_i.\]
	For a closed real current or form $T$ on $X$ with bidegree $(1,1)$, we set $\Theta(T) = \Theta([T])$. In particular, $\Theta(\kappa_0)=\kappa_0$.

For a smooth closed real $(1,1)$-form $\beta$, define $\phi(\beta)\in C^\infty(X)$ to be the unique function such that $\int_X \phi(\beta)\,dV = 0$ and
\[dd^c \phi(\beta) = \beta - \Theta(\beta).\] 
If $T$ is a closed positive $(1,1)$-current, there exists a \emph{unique} $\phi(T)\in L^1(V)$ such that $\int_X \phi(T)\,dV = 0$ and \[dd^c \phi(T) = T - \Theta(T)\] \cite[Section~8.1]{GuedjZeriahi2017}. Here we use the convention
$d^c \coloneqq \frac{i}{2}(\overline{\partial}-\partial)$
so that $dd^c = i\partial\overline{\partial}$.

The function $\phi(T)$ is called the \emph{normalized potential} of $T$ with respect to the $\kappa_i$'s and $V$. %
A different choice of $V$ changes $\phi(T)$ by an additive constant. 

The following appears as Corollary 6.2 in \cite{MR4635340}. We let $\mathcal{Z}_1$ denote the set of closed positive $(1,1)$-currents with unit mass.
\begin{lem}\label{lem:compactness}
	The set  
	\begin{align*}
		\phi(\mathcal{Z}_1) = \set{\phi(T) \mid T \in \mathcal{Z}_1}
	\end{align*} is compact in $L^1(V)$. 
\end{lem}
\noindent Indeed, Cantat and Dujardin note that compactness of $\mathcal{Z}_1$ combined with the continuity of $\Theta$ implies that there exists $A\in \R$ such that each of the potentials in $\phi(\mathcal{Z}_1)$ is $A\kappa_0$-plurisubharmonic. Then Lemma \ref{lem:compactness} follows from Proposition 8.5 and Remark 8.6 of \cite{GuedjZeriahi2017}.

\begin{cor}\label{cor:cts}
	The restriction $\phi: \mathcal{Z}_1 \to L^1(V)$ is continuous.
\end{cor}
\begin{proof}[Proof of Corollary \ref{cor:cts}]
	Let $(T_n)_{n\in \N}$ be a sequence of currents in $\mathcal{Z}_1$, and suppose $T_n \to T$ weakly. Then $\Theta(T_n) \to \Theta(T)$ since $\Theta$ is continuous. Now, if $\psi$ is an accumulation point of $\set{\phi(T_n)}_{n\in \N}$ in $L^1(V)$, then $dd^c \phi(T_n) \to dd^c \psi$ weakly. Thus $T = \Theta(T) + dd^c \psi$. But $\psi$ is normalized and $\phi(T)$ is the unique normalized potential for $T$, so $\psi = \phi(T)$. Since $\phi(\mathcal{Z}_1)\subset L^1(V)$ is compact by Lemma \ref{lem:compactness}, it follows that $\phi(T_n) \to  \phi(T)$ in $L^1(V)$. 
\end{proof}
 
 \subsection{A semi-norm on currents}\label{sec:norms}
 The following construction will be convenient in the proof of Proposition \ref{lem:last2}.
	 Let $B_1(H^{1,1})$ denote the unit ball in $H^{1,1}$ with respect to our fixed Euclidean norm. Define a semi-norm $\norm{\cdot}_\Theta$ on the vector space of $(1,1)$-currents on $X$ by
\begin{align*}
	\norm{C}_\Theta \coloneqq \sup_{\alpha \in B_1(H^{1,1})} |(C \mid \Theta(\alpha))|.
\end{align*}
	Define the distance between two currents $C_1, C_2$ by \[d_\Theta(C_1, C_2)\coloneqq \norm{C_1 - C_2}_{\Theta}.\]

\begin{rmk}\label{rmk:dtheta-mass}
	If $C_1$ and $C_2$ are positive $(1,1)$-currents, then $M(C_i)=(C_i\mid \kappa_0)$ for $i=1,2$. Since $\Theta(\kappa_0)=\kappa_0$, it follows that
	\begin{align*}
		|M(C_1)-M(C_2)|=|(C_1-C_2\mid \Theta(\kappa_0))|\leq \norm{[\kappa_0]}d_\Theta(C_1,C_2).
	\end{align*}
\end{rmk}

\begin{rmk}
	The semi-norm $\norm{\cdot}_\Theta$ satisfies the triangle inequality and is continuous with respect to the weak topology on currents. 
\end{rmk}

 		\subsection{Furstenberg theory on cohomology} \label{sec:cohomology}
 	By Lemma 5.1 in \cite{MR4635340}, there exists $C >0$ such that 
 	\begin{align*}
 		\norm{f^*}_{H^k}\leq C^k \norm{f}_{C^1}^k
 	\end{align*} for all $C^1$ maps $f:X\to X$. Thus the moment condition \eqref{int} implies
 \begin{align}\label{int:cohomology}
 	\int (\ln \norm{f^*}_{H^{1,1}} + \ln \norm{(f^*)^{-1}}_{H^{1,1}})  \,d\mu(f) < \infty
 \end{align} which allows us to apply Furstenberg theory to the action on $H^{1,1}$.  
 	\begin{notn}\label{notn:product}
 		Let $\omega = (\dots, \omega_{-1}, \omega_0, \omega_1, \dots) \in \Omega$. Throughout the paper, we use the notation
 		\begin{align*}
 			f_\omega^n \coloneqq \omega_{n-1} \circ \dots \circ \omega_0 \hspace{.25 cm} \text{ and } \hspace{.25 cm} f_\omega^{-n} \coloneqq \omega_{-n}^{-1} \circ \dots \circ \omega_{-1}^{-1}
 		\end{align*} where $n\geq 1$. For convenience, we let $f_\omega^0$ be the identity map. 
 	
 	\end{notn}
	\begin{thm}[Furstenberg, \cite{MR163345}]\label{FurstenbergLimit} Let $a\in \Psef(X)$ with $\angle{a,a} > 0$. For $\mu^\N$-almost every $\omega\in \Omega_+ = (\supp \mu)^\N$,
		\begin{align*}
			e^+(\omega)\coloneqq \lim_{n\to \infty} \frac{1}{M((f_\omega^n)^*a )}(f_\omega^n)^*a 
		\end{align*} exists and is independent of $a$. By construction, $M(e^+(\omega)) = 1$. Moreover, $e^+(\omega)\in \overline \Kah(X)$ and its projective class is contained in $\partial \mathbb{H}_\mu$. 
	\end{thm}
	\begin{rmk}
		The usual Furstenberg limiting direction is only defined up to projective class, but since $\Aut(X)$ preserves the K\"ahler cone%
		, $e^+(\omega)$ is an honest cohomology class.
	\end{rmk}
	\begin{rmk}\label{wekk}
			The projective classes $[e^+(\omega)]_{\P}\in \partial \mathbb{H}_\mu$ are distributed with respect to the Furstenberg measure on $\partial \mathbb{H}_\mu$, denoted $\mu_{\mathcal{F}}$. The measure $\mu_{\mathcal{F}}\in \Prob(\partial \mathbb H_\mu)$ is simply the push-forward of $\mu^\N$ under $\omega \mapsto[ e^+(\omega)]_{\P}$, i.e. $\mu_{\mathcal F} = \int \delta_{[e^+(\omega)]_{\P}}\, d\mu^\Z$. Moreover, $\mu_{\mathcal{F}}$ is the unique $\mu$-stationary measure on $\P H^{1,1}$ giving zero mass to $\P(\NS(X)_+^\perp)$ \cite[Chapter~3]{MR886674}. Strong irreducibility implies that $\mu_{\mathcal{F}}$ gives no mass to any proper projective subspace in $\P(\NS(X)_+)$. 
	\end{rmk}
	\begin{rmk}\label{rmk:FurstBack}
		The backwards limit classes $e^-(\omega)$ are defined analogously by instead taking $n \to -\infty$. Notice that $e^-(\omega)$ is independent from $e^+(\omega)$, so it follows from the non-atomicity of $\mu_{\mathcal{F}}$ in Remark \ref{wekk} that the two are generically distinct. 
	\end{rmk}

\begin{thm}[Furstenberg, \cite{MR163345}] \label{lyap}
	There exists $\lambda_\mu\in \R_{> 0}$ such that for $\mu^\Z$-almost every $\omega\in \Omega$ and every $a\in H^{1,1}(X, \R)$ with $\angle{a,a} > 0$, 
	\begin{align*}
		\lambda_\mu  = \lim_{n\to \infty}\frac 1 n \ln \norm{(f_\omega^n)^*a} = - \lim_{n\to -\infty}\frac 1 n \ln \norm{(f_\omega^n)^*a} . 
	\end{align*}
\end{thm}
	\begin{rmk}
		If additionally $a\in \Psef(X)$, then by Remark \ref{cpt} we also have 
		\begin{align*}
			\lambda_\mu  = \pm\lim_{n\to \pm\infty}\frac 1 n \ln M((f_\omega^n)^*a).
		\end{align*}
	\end{rmk}
	\begin{rmk}
		The positivity of the exponent $\lambda_\mu$ is guaranteed by Furstenberg's theorem \cite[Theorem~III.6.3]{MR886674} since $\Gamma_\mu$ is non-elementary.
	\end{rmk}
	\begin{rmk}
	The limits in Theorem \ref{lyap} are equal because if $A$ preserves the intersection form then $\norm{A} = \norm{A^{-1}}$ for our adapted norm. 
\end{rmk}

		The following mass-version of Furstenberg's formula is obtained by Cantat and Dujardin from the usual Furstenberg formula using the Birkhoff ergodic theorem and compactness considerations.
		\begin{thm}[Furstenberg, \cite{MR163345}; Cantat--Dujardin, {\citep[Section~5]{MR4635340}}] \label{thm:furst}
	Recalling the notation $\omega = (\omega_i)_{i\in \Z}$, we have
		\begin{align*}
			\lambda_\mu
			&=\int_{\Omega}   \ln M(\omega_0^* e^+(\sigma(\omega))) \,d\mu^\Z(\omega).
		\end{align*}
\end{thm}

		\subsection{Fiber entropy and its basic properties} \label{sec:FiberEntropy}
			In this section, we define fiber entropy and recall some basic facts used in the paper. See, for example, Appendix A of \cite{MR1369243}, or \cite{MR1819189}.
			
			\begin{rmk}\label{rmk:completion}
				When a measure $\nu\in\Prob(\Omega\times X)$ is fixed, notions such as measurable partitions, conditional measures, and entropy are understood with respect to the $\nu$-completion of the Borel $\sigma$-algebra on $\Omega\times X$. With this completed $\sigma$-algebra, $\Omega\times X$ is a Lebesgue probability space.
			\end{rmk}

			\begin{defn}[Fiber entropy] \label{defn:fiber} Let $\eta_\Omega$ denote the $\sigma$-algebra generated by the partition of $\Omega \times X$ into fibers of $p\colon \Omega \times X \to \Omega$.
			Define the \emph{fiber entropy} of $F$-invariant $\nu\in \Prob(\Omega \times X)$ by
			\begin{align*}
					h_\nu^\mathcal{F}\coloneqq h_\nu^{\eta_\Omega}(F)
			\end{align*}
				where $h_\nu^{\eta_\Omega}(F)$ is the $\eta_\Omega$-conditional entropy of $F$ (see \cite{MR1369243}, Definitions 0.4.1 and Appendix A or see \cite{MR1819189}, Section~2). 
		\end{defn} 
		\begin{rmk}
			In the case that $\nu\in S(\mu)$ and  $h_{\mu^\Z}(\sigma) < \infty$ (the second property holds, for example, when $\mu$ is finitely supported) we have  $h_\nu^\mathcal{F} = h_\nu(F) - h_{\mu^\Z}(\sigma)$. %
		\end{rmk}
	
		Given a finite measurable partition $\eta$ of $X$ and $\tau$ a probability measure on $X$, recall  
		\begin{align}
			H_\tau(\eta)\coloneqq - \int_{X}\ln \tau(\eta(x)) d\tau(x).
		\end{align} 
		Let $\nu\in S(\mu)$ be ergodic and recall the disintegration from Remark \ref{disintegration}. For any finite measurable partition $\eta$ of $X$, 
		\begin{align}\label{align:2}
			h_{\nu}^\mathcal{F} (\eta ) \coloneqq \lim_{n\to \infty}\frac 1 n H_{\nu_\omega}\left (\bigvee_{k=0}^{n-1}(f_\omega^k)^{-1}\eta\right )
		\end{align} exists and is constant $\mu^\Z$-almost surely. Moreover, 
		\begin{align}
			h_{\nu}^\mathcal{F} = \sup\set{h_{\nu}^\mathcal{F} (\eta )\mid \eta \text{ is a finite measurable partition of $X$}}
		\end{align} where $h_{\nu}^\mathcal{F}$ is the fiber entropy as defined in Definition \ref{defn:fiber} \cite[Theorem~II.1.4]{MR884892} or \cite{MR1181382}.
		
		Fiber entropy satisfies many properties analogous to entropy for a single map; see, for example, Chapters 0 and 1 of \cite{MR1369243}.

		\subsection{Recollections from Pesin theory}
	The following constructions are standard in Pesin theory, and can be found for example in \cite{MR1369243}. Note that $\nu\in S(\mu)$ might not be induced by a stationary measure on $X$, and therefore is not explicitly considered in Liu-Qian. 
	However, the theorems from \cite{MR1369243} which we cite also apply to ergodic $\nu$ with the same proofs. See also Section~5 of \cite{MR3671937}, where our situation is explicitly considered, or \cite{MR2032494}.

\begin{thm}[Oseledets' multiplicative ergodic theorem]
		Suppose $\nu \in S(\mu)$ is ergodic. There exist $\nu$-measurable assignments
	\begin{align*}
		(\omega,x) \mapsto E^u_\omega(x), \hspace{.5 cm}
		(\omega,x) \mapsto E^s_\omega(x),
	\end{align*}
	where $E^u_\omega(x)$, $E^s_\omega(x)$ are complex subspaces of $T_xX$, and real numbers $\lambda^s \leq \lambda^u$ such that, for $\nu$-almost every $(\omega, x)\in \Omega \times X$, $T_xX = E^u_\omega(x) \oplus E^s_\omega(x)$ and for every $v\in E^{s/u}_\omega(x) \setminus \set{0}$,
	\begin{align*}
		\lim_{n\to \pm \infty} \frac 1 n \ln \norm{Df_\omega^n v} &= \lambda^{s/u}.
	\end{align*}
\end{thm}
	\begin{rmk}
		It might be that $\lambda^s = \lambda^u$, in which case we allow one of $E^s, E^u$ to be trivial. In the case that $X$ admits an $\Aut(X)$-invariant volume form, for example when $X$ is a $K3$ surface or a complex torus, we have $\lambda^s + \lambda^u = 0$.
	\end{rmk}
		\begin{defn}[Hyperbolic measure]
		We call an ergodic measure $\nu\in S(\mu)$ \emph{hyperbolic} if $\lambda^s < 0 < \lambda^u$.
	\end{defn}
	\noindent Let
	\begin{align*}
		W_\omega^s(x)& \coloneqq \set{y\in X\colon\lim_{n\to \infty}  \frac 1 n \ln d(f_\omega^n y,f_\omega^n x) < 0},\\
 &\hspace{-1.3 cm}W_\omega^u(x) \coloneqq \set{y\in X\colon \lim_{n\to \infty} \frac 1 n \ln d(f_\omega^{-n} y,f_\omega^{-n} x)< 0}.
	\end{align*}
		In Section 7 of \cite{MR4635340}, the following is stated as a consequence of Pesin theory in the situation of $\nu$ obtained from a stationary measure. Again, it holds for arbitrary hyperbolic $\nu\in S(\mu)$. 
	
		\begin{thm}[\cite{MR4635340}, Proposition 7.11]
		Suppose $\nu \in S(\mu)$ is ergodic and hyperbolic. Then, $\nu$-almost everywhere, $W_\omega^{s/u}(x)$ are injectively immersed entire curves. That is, there exist  holomorphic injections $\phi^{s/u}:\C\to X$ parameterizing $W_\omega^{s/u}(x)$. 
	\end{thm}
	
		Indeed, as noted in \cite{MR2032494} and \cite{MR3671937}, the construction of Lyapunov charts in Chapter~6, Section~3 of \cite{MR1369243} is identical with the same proofs, and therefore we have the local stable manifold theorem (see Chapter~6, Section~4 of \cite{MR1369243}; or Theorem 6.4 of \cite{MR3671937}; or Proposition 7.9 of \cite{MR4635340}). By holomorphicity of the maps $f_\omega^n$, the manifolds $W^{s/u}_\omega(x)$ are $J$-invariant and thus complex submanifolds. In fact, they are holomorphic immersions of $\C$ (see \cite{MR4635340}, Proposition 7.10).
	
	\subsection{A special partition}
	Throughout this section, we assume that $\nu\in S(\mu)$ is ergodic and hyperbolic.
		\begin{defn}[Subordinate partition]\label{def:subordinate}
		A $\nu$-measurable partition $\eta$ of $\Omega \times X$ is called \emph{subordinate} to unstable manifolds of $\nu$ if, for $\nu$-almost every $(\omega, x)$, we have
		\begin{enumerate}
			\item $\eta_\omega(x)\coloneqq \set{y\in X \mid (\omega, y)\in \eta(\omega, x)}\subset W^u_\omega(x)$ and 
			\item there exists some open neighborhood $U(\omega, x)$ of $x$ in $ W^u_\omega(x)$ such that $U(\omega, x)\subset \eta_\omega(x)$.
		\end{enumerate}
		$\nu$-measurable partitions subordinate to stable manifolds are defined analogously. 
	\end{defn}

	Let $ \lambda^u_{(\omega, x)}$ denote the Lebesgue measure on $W^u_\omega(x)$ induced by our fixed Riemannian metric on $X$ associated to $\kappa_0$.
	\begin{prop}[Proposition~5.2 in Chapter~VI of \cite{MR1369243}]\label{prop:partition}
		There exists a partition $\eta$ of $\Omega \times X$ subordinate to the unstable manifolds that further satisfies
		\begin{enumerate}
			\item $F\eta \leq \eta$, 
			\item $\vee_{n\geq 0}F^{-n} \eta$ is the partition into points,
			\item the sigma-algebra generated by $\wedge_{n\geq 0}F^{n} \eta$ agrees with the one generated by unions of unstable manifolds, up to $\nu$-measure $0$, and
			\item for every $\nu$-measurable subset $B\subset \Omega \times X$, the function
			\begin{align*}
			(\omega, x)\mapsto \lambda^u_{(\omega, x)}(\eta(\omega, x)\cap B)
			\end{align*}
			is $\nu$-measurable and $\nu$-almost everywhere finite.
		\end{enumerate}
		
	\end{prop}
	\begin{rmk}
		We may further assume that $\eta$ is a refinement of the partition of $\Omega \times X$ into the fibers of $p$.
	\end{rmk}
	\begin{rmk}\label{bdd}
		It follows from the construction in \cite{MR1369243} that there exists a $\nu$-measurable function $r\colon \Omega\times X\to \R_{>0}$ such that \[W^u_{r(\omega, x)}(\omega, x) \subset \eta_\omega(x)\] almost surely, where $W^u_{r}(\omega, x)$ denotes a ball of radius $r$ about $x$ in $W^u_\omega(x)$ with respect to the induced metric. It further follows from the construction that, $\nu$-almost surely, after choosing a holomorphic parametrization $\psi:\C\to W^u_\omega(x)$ there exists an open disk $\Delta\subset \C$ with 
		\[\eta_\omega(x) \subset \psi(\Delta).\]
	\end{rmk}

		Throughout, we let $\nu_{\eta(\omega, x)}$ denote the conditional measure of $\nu$ on the atom $\eta(\omega, x)$, and recall from the introduction that we write 
	\begin{align*}
		\nu_{\eta(\omega, x)} = \delta_\omega \otimes \nu_{\eta_\omega(x)}
		\end{align*} where $\nu_{\eta_\omega(x)}\in \Prob(X)$. We let $H(F^{-1} \eta \mid \eta)$ denote the conditional entropy of $F^{-1} \eta$ given $\eta$, defined by 
	\begin{align*}
		H(F^{-1} \eta \mid \eta)\coloneqq -\int \ln \nu_{\eta(\omega, x)}(F^{-1}(\eta(F(\omega, x)))) \,d\nu.
	\end{align*}
	\begin{prop}[Corollary~7.1 in Chapter~VI of \cite{MR1369243}] \label{Prop:full}
		Let $\eta$ be a partition satisfying Proposition \ref{prop:partition}. Then $H(F^{-1} \eta \mid \eta) = h_{\nu}^\mathcal{F}$. 
	\end{prop}
	\subsection{The Ledrappier-Young formula}\label{sec:LY}

	The results cited in this section are used only in Lemma \ref{posent}, where we show that positive fiber entropy rules out a degenerate case where almost surely $\nu_\omega$ is a uniform measure on a finite union of irreducible curves.
	
	Let $B_r^u(\omega, x)$ denote the ball about $x$ in $W_\omega^u(x)$ with radius $r$ in the induced Riemannian metric on $W_\omega^u(x)$, and define $B_r^s(\omega, x)$ analogously. Given measurable partitions $\eta^u$ and $\eta^s$ subordinate to unstable and stable manifolds, respectively, let 
	\begin{align*}
		\dim^u (\nu, (\omega, x)) \coloneqq \lim_{r \to 0}\frac{\ln (\nu_{\eta^u_\omega(x)}(B_r^u(\omega, x)))}{\ln r}
	\end{align*} and 
	\begin{align*}
		\dim^s (\nu, (\omega, x)) \coloneqq \lim_{r \to 0}\frac{\ln (\nu_{\eta^s_\omega(x)}(B_r^s(\omega, x)))}{\ln r}.
	\end{align*} Notice that $\dim^{s/u} (\nu, (\omega, x))$ do not depend on the choices of subordinate partitions, since any two such partitions will have conditional measures that agree on overlaps up to normalization.

	The following theorem is stated in our setting as Proposition 6.10 of \cite{MR3671937}. It is essentially due to Ledrappier and Young \cite{MR819557}.
	
	\begin{thm}\label{thm:magic}
		We have 
		\begin{align*}
			h^\mathcal F _\nu = \lambda^u \dim^u (\nu, (\omega, x))=  - \lambda^s \dim^s (\nu, (\omega, x))
		\end{align*}
	almost surely.
	\end{thm}
	
	As a corollary, we obtain the following.
	
	\begin{cor}\label{lem:atomic}
		The following are equivalent:
		\begin{enumerate}
			\item $h_{\nu}^\mathcal{F} > 0$
				\item for every $\nu$-measurable partition $\eta^u$ subordinate to the unstable manifolds, $\nu_{\eta^u_\omega(x)}$ is non-atomic almost everywhere
				\item for every $\nu$-measurable partition $\eta^s$ subordinate to the stable manifolds, $\nu_{\eta^s_\omega(x)}$ is non-atomic almost everywhere
		\end{enumerate}
	\end{cor}
	\begin{rmk}\label{invertibility}
	Our setup is invertible, meaning that the map $F^{-1}$ is contained in the class of skew-products we consider. Indeed, consider the inversion map $\iota:\Aut(X)\to \Aut(X)$ defined by $\iota(f) = f^{-1}$ and let $\hat \mu = \iota_* \mu$. Then we get a skew-product
	$\hat F : (\supp \hat \mu)^\Z \times X \to (\supp \hat \mu)^\Z \times X$ defined as before.
	
	Let $\mathrm{Flip}$ be the involution of $\Aut(X)^\Z$ defined by $\mathrm{Flip}(\omega)_i \coloneqq \omega_{-i - 1}^{-1}$. Notice that $T_\omega^{\pm} = T_{\mathrm{Flip}(\omega)}^{\mp}$. Moreover, $$\mathrm{Flip} \times id: \Aut(X)^\Z \times X \to \Aut(X)^\Z \times X$$ conjugates $F^{-1}$ to $\hat F$, so \[\hat \nu \coloneqq (\mathrm{Flip} \times id)_* \nu\] is $\hat F$-invariant. The conjugacy sends the stable directions of $\nu$ to the unstable directions of $\hat \nu$, and the unstable directions of $\nu$ to the stable directions of $\hat \nu$.
\end{rmk}
	\begin{proof}[Proof of Corollary \ref{lem:atomic}] 
			By the invertibility discussed in Remark \ref{invertibility}, it is sufficient to prove that (1) and (2) are equivalent.
		
		We first show that the negation of (2) is equivalent to the a priori stronger assertion that $x$ is an atom of $\nu_{\eta_\omega(x)}$ almost everywhere. Let $\eta^u$ be some $\nu$-measurable partition subordinate to the unstable manifolds of $\nu$, and suppose that the set of $(\omega, x)$ such that $\nu_{\eta_\omega(x)}$ contains an atom has positive measure. For $r > 0$ let $S_r$ denote the set of $(\omega, x)$ such that $\nu_{\eta_\omega(x)}$ contains an atom of mass greater than $r$. Then for some $r$, $S_r$ has positive measure. Let
		\begin{align*}
			P_r = \set{(\omega, x) \colon \nu_{\eta_\omega(x)}(\set{x}) > r}.
			\end{align*} Now $P_r$ is $\nu$-measurable. To see this, we prove $(\omega, x) \mapsto \nu_{\eta_\omega(x)}(\set{x})$ is $\nu$-measurable. Write
	\begin{align*}
	\nu_{\eta_\omega(x)}(\set{x}) =
		\lim_{s\to0} \nu_{\eta_\omega(x)}(B_s(x)).
	\end{align*} 
	Now 
	\begin{align*}
	  (\omega, x)\mapsto \nu_{\eta_\omega(x)}(B_s(x)) = \int_X\mathds{1}_{B_s(x)}(y) \, d\nu_{\eta_\omega(x)}(y)
\end{align*} is $\nu$-measurable by a standard measurable kernel argument. Indeed, note that $(x,y)\mapsto \mathds{1}_{B_s(x)}(y)$ is measurable and $(\omega,x)\mapsto \nu_{\eta_\omega(x)}$ is a $\nu$-measurable kernel $\Omega \times X\to X$, so apply Lemma 3.2 (i) of \cite{MR1876169}. 

	For every atom $\eta_\omega(x)$ where $(\omega, x)\in S_r$, we have $\nu_{\eta_\omega(x)}(P_r) > r$. By the disintegration of $\nu$ with respect to conditional measures, it follows that
	\begin{align*}
		\nu(P_r) > r \nu(S_r)
	\end{align*} which is positive.
	The set
		\[\set{(\omega, x)\colon \text{$x$ is an atom of $\nu_{\eta_\omega(x)}$}}\]
		is $F$-invariant, and since it contains $P_r$ we just showed it has positive measure; hence, it has full measure by ergodicity. 
		
		Now if $x$ is an atom in $\nu_{\eta_\omega(x)}$, then $\dim^u (\nu, (\omega, x)) = 0$ and by Theorem \ref{thm:magic} we get $h_{\nu}^\mathcal{F} = 0$. Thus (1) implies (2). Conversely, suppose $h_{\nu}^\mathcal{F} = 0$. Letting $\eta^u$ be the special partition from Proposition \ref{prop:partition}, 
		\[ 0 = h_{\nu}^\mathcal{F}  =\frac 1 n H(F^{-n}\eta^u \mid \eta^u) \] 
		for all $n\geq 1$ by Proposition \ref{Prop:full} and $F$-invariance. So we get that, for all $n\geq 1$, $\nu_{\eta^u(\omega, x)}$ is supported on $F^{-n}\eta^u(x)$ almost surely. But $\lim_{n\to \infty}F^{-n}\eta^u = \vee_{n=1}^\infty F^{-n}\eta^u$ is the partition into points, so actually $\nu_{\eta^u(\omega, x)} = \delta_x$ almost surely, proving the negation of (2).
	\end{proof}

	\subsection{Limit currents}\label{sec:LimitCurrents}

		\begin{prop}[Cantat--Dujardin, {\citep[Section~6]{MR4635340}}]\label{prop:limit_currents}
		For $\mu^\N$-almost every $\omega\in \Omega_+$, there exists a unique closed positive current $T_\omega^+$ in the class $e^+(\omega)$. Moreover, $\mu^\Z$-almost surely we have
		\begin{enumerate}
			\item for any K\"ahler form $\kappa$,
			\begin{align*}
					\frac{1}{M((f^n_\omega)^*\kappa)}(f_\omega^n)^*\kappa \to T_\omega^+,
			\end{align*} and 
			\item  $\phi(T_\omega^+)$ is continuous.
		\end{enumerate} 
	\end{prop}
	\begin{rmk}
			By symmetry, there exists an analogous current $T_\omega^- \in e^-(\omega)$ obtained by replacing $n\to \infty$ with $n\to -\infty$. The current $T_\omega^-$ of course satisfies the same properties as $T_\omega^+$. 
	\end{rmk}
	\begin{rmk}\label{rmk:msbl}
			The assignment $\omega \mapsto T_\omega^+$ is $\mu^\Z$-measurable where the space of currents is endowed with its Borel $\sigma$-algebra coming from the weak topology. Indeed, this follows from (1) in Proposition \ref{prop:limit_currents}. Alternatively, it can also be deduced from the measurability of $\omega \mapsto e^+(\omega)$ and the uniqueness of $T_\omega^+$. 
	\end{rmk}

Recall that $V$ is a volume form fixed in Section \ref{sec:Potentials}.
\begin{lem}\label{lem:L1}
	The assignment $\omega \mapsto \phi(T_\omega^+)$ is $\mu^\Z$-measurable as a map into $L^1(V)$, endowed with its norm topology.
\end{lem}
\begin{proof}[Proof of Lemma \ref{lem:L1}]
	By Remark \ref{rmk:msbl} and Corollary \ref{cor:cts}, $\omega \mapsto \phi(T_\omega^+)$ is a composition of a $\mu^\Z$-measurable map with a continuous map. Hence it is $\mu^\Z$-measurable.
\end{proof}
	\begin{lem}\label{lem:equicontinuous}
		For all $\rho > 0$, there exists a compact subset $K\subset \Omega$ such that $\mu^\Z(K) > 1 - \rho$ and 
		\begin{align*}
			\set{\phi(T_\omega^+) \mid \omega \in K}
		\end{align*} 
	 is an equicontinuous family admitting a uniform $C^0$-bound. 
	\end{lem}
	\begin{proof}[Proof of Lemma \ref{lem:equicontinuous}]
			Consider $\phi(\omega, x)\coloneqq \phi(T_\omega^+)(x)$ as a function $\Lambda \times X \to \R$ where $\Lambda\subset \Omega$ has full measure. By Proposition \ref{prop:limit_currents}, for each $\omega\in \Lambda$ we have that $x\mapsto \phi(\omega, x)$ is continuous. 
		\begin{claim}
			For each $x\in X$, $\omega \mapsto \phi(\omega, x)$ is $\mu^\Z$-measurable. 
		\end{claim}
			\noindent To prove the claim, let $x\in X$ and $(f_n)_{n\in \N}$ be a sequence of smooth functions $X\to \R_{\geq 0}$ such that $\int f_n V = 1$ for all $n$ and $f_n V \to \delta_x$ weakly as measures. For all $n$, the assignment $\omega \mapsto \int \phi(T_\omega^+) f_n V$ is $\mu^\Z$-measurable. Indeed, this follows from Lemma \ref{lem:L1} and the boundedness of each $f_n$. 

		Therefore 
		\begin{align*}
				\phi(\omega, x) = \lim_{n\to \infty}\int \phi(T_\omega^+)f_n V
		\end{align*} is $\mu^\Z$-measurable for fixed $x$, as desired.
	
			To conclude, we apply a theorem of Scorza and Dragoni to obtain a compact subset $K\subset \Lambda$ with $\mu^{\Z}(K) > 1 - \rho$ such that the restriction $\restr{\phi}{K\times X}$ is continuous. Indeed, one takes $T$ in Theorem 1 of \cite{MR1121606} to be some compact set $K_0\subset \Omega$ with large $\mu^\Z$-measure, and $B = K_0 \times X$, so that $K$ is obtained as a subset of $K_0$. Then the collection 
			\[\set{x\mapsto \restr{\phi}{\set{\omega}\times X}\mid \omega\in K}\] is equicontinuous and admits a uniform $C^0$-bound, which is what we desire. 
	\end{proof}
\begin{rmk} \label{equivariance}
		From the convergence and uniqueness statements in Proposition \ref{prop:limit_currents}, one can deduce that the limit currents satisfy the following equivariance properties:
	\begin{align}\label{equivar}
		\omega_0^*T_{\sigma(\omega)}^+ &= M(\omega_0^*T_{\sigma(\omega)}^+) T_\omega^+ = \frac{1}{M((\omega_0)_*T_{\omega}^+)} T_\omega^+ \\
		\omega_0^*T_{\sigma(\omega)}^- &= M(\omega_0^*T_{\sigma(\omega)}^-) T_\omega^- = \frac{1}{M((\omega_0)_*T_{\omega}^-)} T_\omega^-.
	\end{align}

\end{rmk}

A crucial step in Cantat and Dujardin's proof of Proposition \ref{prop:limit_currents}, specifically in their derivation of (2) from (1), is provided by the following lemmas. We state them here and include portions of the proofs due to Cantat and Dujardin because they will be independently crucial to the rest of this article. 

\begin{lem}[\cite{MR4635340}, Corollary 6.10]\label{cor:expansionBound}
	There exists a constant $C>0$, depending only on a choice of coordinate charts on $X$ and choice of basis from Section \ref{sec:Potentials}, such that, for every $f\in \Aut(X)$ and $a\in \Psef(X)$,
	\begin{align*}
		\norm{\phi(f^*\Theta(a))}_{C^1} \leq C \norm{a} \norm{f}^2_{C^1}.
	\end{align*} 
\end{lem} 

	\begin{defn}[Upper-density and lower-density]\label{defn:density}
	Recall that the \emph{upper-density} of a subset $S$ of the naturals is \[\overline{d}(S)\coloneqq \limsup_{k\to \infty }\frac 1 k \#\set{S\cap [1, k-1]}.\] For a subsequence $(n_i)_{i\in\N}$ of $\N$, we define its upper-density to be the upper-density of the set $\set{n_i}_{i\in \N}$. Similarly, the \emph{lower-density} of $S$  is \[\underline{d}(S)\coloneqq \liminf_{k\to \infty }\frac 1 k \#\set{S\cap [1, k-1]}.\] 
\end{defn}

\begin{lem}[\cite{MR4635340}, Lemma 6.15]\label{lem:C0Bound}
	Let $\kappa$ be a K\"ahler form. For all $\rho > 0$, there exists a subset $S_\rho\subset \Omega$ with $\mu^\Z(S_\rho) > 1 - \rho$ such that, for all $\omega\in S_\rho$, one can find $C > 0$ and a subsequence $(n_i)_{i\in \N}$ of $\N$ with upper-density greater than $1 - \rho$ satisfying
	\begin{align*}
		\frac{1}{\norm{(f_\omega^{n_i})^*}}\norm{\phi((f_\omega^{n_i})^*\kappa)}_{C^0} < C
	\end{align*}
	for all $i\in \N$. 
\end{lem} 
Note that Lemma 6.15 in \cite{MR4635340} is stated without the upper-density assertion, but, as we will see, their proof provides it. An immediate corollary is the following.
\begin{cor}\label{cor:C0Bound}
		Let $\kappa$ be a K\"ahler form. For $\mu^\Z$-almost every $\omega\in \Omega$ and every $\rho > 0$, one can find a constant $C> 0$ and a subsequence $(n_i)_{i\in \N}$ with upper-density greater than $1 - \rho$ satisfying
		\begin{align*}
		\frac{1}{\norm{(f_\omega^{n_i})^*}}\norm{\phi((f_\omega^{n_i})^*\kappa)}_{C^0} < C
	\end{align*}
	for all $i\in \N$. 
\end{cor}
\begin{proof}[Proof of Corollary \ref{cor:C0Bound}]
	The set
	\begin{align*}
		\bigcap_{N\in \N}  \left(\bigcup_{i \geq N} S_{\frac 1 i }\right)
	\end{align*}
		has full $\mu^\Z$-measure and the desired properties.
\end{proof}

\begin{proof}[Proof of Lemma \ref{lem:C0Bound}]
	The starting point is the elementary computation
	\begin{align*}
		\phi(f^* \alpha) &= \phi(f^* \Theta(\alpha)) + \phi(\alpha)\circ f - c(f,\alpha)
	\end{align*}
	for every smooth closed real $(1,1)$-form $\alpha$, where $c(f,\alpha)$ is a normalizing constant. Iterating, we find
	\begin{align*}
		\phi((\omega_1\circ \omega_0)^* \kappa) &= \phi(\omega_0^* \omega_1^* \kappa) \\
		&=\phi(\omega_0^* \Theta(\omega_1^*\kappa)) + \phi(\omega_1^*\kappa)\circ \omega_0 - c_1
		 \\
		&=\phi(\omega_0^* \Theta(\omega_1^*\kappa)) + \phi(\omega_1^* \Theta(\kappa))\circ \omega_0 + \phi(\kappa)\circ \omega_1 \circ \omega_0 - c_1 - c_2
\end{align*}
where $c_1$ and $c_2$ are normalizing constants. More generally,
	\begin{align*}
			\phi((f_\omega^n)^* \kappa) + c(\omega, n)=   \phi(\kappa) \circ f_\omega^n + \sum_{i=1}^{n}\phi(\omega_{i-1}^*\Theta((f_{\sigma^i(\omega)}^{n-i})^*\kappa)) \circ f_\omega^{i-1} 
	\end{align*}
	for all $n\geq 1$, where again $c(\omega,n)$ is a normalizing constant.

Since the normalizing constants are controlled by the $C^0$-norm of the right-hand side,
\begin{align}\label{ugh2}
	\norm{\phi((f_\omega^n)^*\kappa)}_{C^0} & \lesssim  
	\norm{\phi(\kappa)}_{C^0} + \sum_{i=1}^{n}\norm{\phi(\omega_{i-1}^*\Theta((f_{\sigma^i(\omega)}^{n-i})^*\kappa))}_{C^0} \notag \\&\lesssim  
	\sum_{i=1}^{n}  \norm{(f_{\sigma^i(\omega)}^{n-i})^* [\kappa]} \norm{\omega_{i-1}}_{C^1}^2 
\end{align} where in the second line we apply Lemma \ref{cor:expansionBound}, and $\lesssim$ denotes $\leq$ up to multiplicative and additive constants independent of $\omega$ and $n$. 

	Therefore, to prove the lemma it is sufficient to find $(n_i)$ with large upper-density such that there exists a finite constant $R(\omega)$ depending only on $\omega$ satisfying
	\begin{align*}
			\frac{1}{\norm{(f_\omega^{n_i})^*}}\sum_{j=1}^{n_i}  \norm{(f_{\sigma^j(\omega)}^{n_i-j})^*}\norm{\omega_{j-1}}^2_{C^1}< R(\omega)
		\end{align*} for all $i$, where we used that $M(\cdot)\asymp \norm{\cdot}$ on $\Psef(X)$.

	By the moment condition \eqref{int} and Borel--Cantelli, for all $\epsilon > 0$, $e^{-\epsilon i }\norm{\omega_i}^2_{C^1}$ is bounded independently of $i\in \N$. Therefore, it is enough to find $(n_i)$ such that $\norm{(f_{\sigma^j(\omega)}^{n_i-j})^*}/\norm{(f_\omega^{n_i})^*}$ decays exponentially in $j$, in a sense made precise below.
	To do so, one applies a theorem of Gou\"ezel and Karlsson \cite[Theorem~1.1]{MR4092901} to the sub-additive cocycle $a(n,\omega)\coloneqq \ln \norm{(f_\omega^n)^*}$.

	\begin{thm}[Gou\"ezel-Karlsson \cite{MR4092901}, Theorem 1.1 and Remark 1.2]\label{thm:GoodTimes}
		Let $a(n,\omega)$ be an ergodic sub-additive cocycle with finite asymptotic average 
		\begin{align*}
				A\coloneqq \lim_{n\to \infty} \frac 1 n \int_\Omega a(n,\omega) \,d\mu^\Z(\omega).
			\end{align*} For all $\rho>0$, there exists a sequence of positive reals $\delta_\ell \to 0$ and a subset $S\subset \Omega$ with $\mu^\Z(S) > 1 - \rho$ such that for all $\omega \in S$ there exists  a sequence of naturals $(n_i)$  satisfying 
		\begin{enumerate}
			\item $(n_i)$ has upper-density greater than $1 - \rho$ and 
			\item  $a(n,\omega)$ is \emph{tight} along $(n_i)$, i.e. 
			\begin{enumerate}
				\item[(i)] for all $i$ and $0\leq \ell \leq n_i$, 
				\begin{align*}
					|- a(n_i, \omega) + a(n_i - \ell, \sigma^\ell(\omega)) + A\ell| \leq \ell \delta_\ell
				\end{align*} and 
				\item[(ii)] for all $i$ and $0\leq \ell \leq n_i$
				\begin{align*}
					a(n_i, \omega) - a(n_i - \ell, \omega) \geq (A - \delta_\ell)\ell.
				\end{align*}
			\end{enumerate}
		\end{enumerate}
	\end{thm}
		While $(ii)$ is not explicit in the statement of Gou\"ezel--Karlsson, it is easily deduced by Cantat and Dujardin \cite[Theorem~5.11]{MR4635340}. 
	
	In our case, $A = \lambda_\mu$. Now taking $\omega\in S$ and $(n_i)$ as in the theorem, and any $ 0 < \epsilon < \lambda_\mu$,
	\begin{align*}
			\frac{1}{\norm{(f_\omega^{n_i})^*}}\sum_{j=1}^{n_i}  \norm{(f_{\sigma^j(\omega)}^{n_i-j})^*}\norm{\omega_{j-1}}^2_{C^1} &\lesssim   \sum_{j=1}^{n_i} e^{j \epsilon} \frac{\norm{(f_{\sigma^j(\omega)}^{n_i -  j})^*}}{\norm{(f_\omega^{n_i})^*}}\\ & =  \sum_{j=1}^{n_i} \frac{e^{j \epsilon + a(n_i - j, \sigma^j(\omega))}}{e^{a(n_i,\omega)}}  \\ &\leq \sum_{j=1}^{n_i} \frac{e^{j \epsilon + j \delta_j - j \lambda_\mu + a(n_i,\omega)}}{e^{a(n_i,\omega)}}  = \sum_{j=1}^{n_i} e^{j (\epsilon + \delta_j - \lambda_\mu)}
	\end{align*} which is a convergent geometric series since $\epsilon < \lambda_\mu$ and $\delta_j\to 0$, where we used (i) in the final inequality step.
\end{proof}
\begin{rmk}\label{remark:later}
	Since it will be useful later, we remark that the analog to \eqref{ugh2} in backwards time is given by
	\begin{align*}
		\norm{\phi((f_\omega^{-n})^*\kappa)}_{C^0} & \lesssim  
		\sum_{i=1}^{n}  \norm{(f_{\sigma^{-i}(\omega)}^{-n+i})^* [\kappa]} \norm{\omega_{-i}^{-1}}_{C^1}^2,
	\end{align*} where $n\geq 1$. 
\end{rmk}
Consider the exponential moment condition 
\begin{align}\label{exponential}
		\text{$\exists \ \alpha > 0$ such that }\int ( \norm{f}_{C^1} +  \norm{f^{-1}}_{C^1})^\alpha \,d\mu(f) < \infty.
\end{align}
Then Cantat and Dujardin prove the following.
\begin{prop}[\cite{MR4635340}, Lemma 6.15 and Proposition 5.14]\label{prop:exponentialC0}
	If the exponential moment condition \eqref{exponential} is satisfied, then the sequence $(n_i)$ in the conclusion of Corollary \ref{cor:C0Bound} can be taken to be $\N$. 
\end{prop}

		\subsection{Wedges of currents}  \label{sec:wedges}
		A complete reference for the following can be found in \cite{Demailly2009}. Let $T$ and $S$ be closed positive $(1,1)$-currents on $X$. Assume that $S$ has continuous potentials, meaning that $\phi(S)$ is continuous. Let $T$ be supported on some open set $U$ in $X$. Then $T\wedge S$ is a Radon measure on $U$ defined locally by
	\begin{align}\label{wedge}
			\int_X \phi\,d(S\wedge T)  \coloneqq (T\mid u_S dd^c \phi )
	\end{align} where $\phi$ is a compactly supported test function with support contained in some open set $V\subset U$ and $u_S:V\to \R$ satisfies $dd^c u_S = \restr{S}{V}$ (in the distributional sense). When both $S$ and $T$ have continuous potentials, this definition is symmetric. 

		Now suppose $T$ is not assumed to be closed. Then $T\wedge S$ still defines a Radon measure on $X\setminus \supp \partial T$ as in \eqref{wedge}. The measure $T\wedge S$ extends naturally to a Radon measure on $X$ by taking the product with $\mathds{1}_{X\setminus \supp \partial T}$. If $\rho$ is a nonnegative function on $X$ that is integrable with respect to $T\wedge S$, then the product $ \rho(T\wedge S)$ is a positive measure, and we define 
	\begin{align*}
	(\rho T) \wedge S = S\wedge (\rho T) \coloneqq \rho(T\wedge S).
	\end{align*}
	For $f\in \Aut(X)$, we define  
	\begin{align*}
		 	f_*(\psi T) \wedge S = S\wedge f_*(\psi T)  \coloneqq f_*((f^* S) \wedge (\psi T)).
	\end{align*}

	\begin{lem}\label{lem:currClosed}
	Suppose $(T_n)$ is a sequence of closed positive $(1,1)$-currents converging weakly as $n\to \infty$ to a current $T$ with continuous potentials.  Let $(B_n)$ be a sequence of closed positive $(1,1)$-currents such that $\set{[B_n]}$ is bounded in cohomology and $\set{\phi(B_n)}$ is an equicontinuous family admitting a uniform $C^0$ bound. Then 
	\begin{align*}
		T_n\wedge  B_n - T \wedge B_n \to 0
	\end{align*} in the weak sense of measures. 
\end{lem}

\begin{proof}[Proof of Lemma \ref{lem:currClosed}]
	Let $g \in C^\infty(X)$ be supported in an open bidisk $V\subset X$. For each $\kappa_i$ in Section \ref{sec:Potentials}, choose local potentials $\restr{\kappa_i}{V} = dd^c v_i$. Then if $\Theta(B_n) = \sum a_i \kappa_i$, let $u_{\Theta(B_n)} = \sum a_i v_i$ so that $\restr{\Theta(B_n)}{V} = dd^c u_{\Theta(B_n)} $. 
	Then
	$u_{B_n}\coloneqq u_{\Theta(B_n)}+\phi(B_n)$ is a local potential for $B_n$ on $V$.  By assumption, $\set{[B_n]}$ is bounded in cohomology and $\set{\phi(B_n)}$ is a uniformly bounded equicontinuous family. It then follows from our definition of $u_{B_n}$ and $u_{\Theta(B_n)}$ that $\set{u_{B_n}dd^cg}$ is precompact in the $C^0$-topology on coefficient functions. Recalling the definition of the wedge product from Section \ref{sec:wedges}, we have
	\begin{align*}
		(T_n \wedge B_n \mid g)
		&= (  T_n  \mid u_{B_n} dd^cg).
	\end{align*} Since the $T_n$'s are positive currents, the local coefficient measures of $T_n$ in $V$ converge to the corresponding coefficient measures of $T$. By precompactness of the family $\set{u_{B_n}dd^cg}$, it follows that $(T_n \mid u_{B_n}dd^c g) - (T \mid u_{B_n}dd^c g)\to 0$. Since $(T \mid u_{B_n}dd^c g) = (T \wedge B_n \mid g)$, the lemma is proven.
\end{proof}

\begin{lem}\label{lem:m_omega_measurability}
	The assignment $\omega \mapsto m_\omega$ is $\mu^\Z$-measurable with the weak topology on $\Prob(X)$. 
\end{lem}
\begin{proof}[Proof of Lemma \ref{lem:m_omega_measurability}]
 	For all $\rho > 0$, we may use Lusin's theorem and Lemma \ref{lem:equicontinuous} to find a compact set $K\subset \Omega$ such that
 	\begin{enumerate}
 		\item $\mu^\Z(K) > 1 - \rho$,
 		\item the assignments $\omega \mapsto T_\omega^\pm$ are continuous on $K$,
 		\item  the family \[\set{\phi(T_\omega^+) \mid \omega \in K}\] is equicontinuous and admits a uniform $C^0$-bound,
 		\item and the function 
 		\begin{align*}
 			\beta(\omega)\coloneqq \int_X T_\omega^+\wedge T_\omega^- = \angle{e^+(\omega), e^-(\omega)}
 		\end{align*} is continuous and positive on $K$.
 	\end{enumerate} 
 By Lemma \ref{lem:currClosed}, the restriction of $\omega \mapsto T_\omega^+\wedge T_\omega^-$ to $K$ is continuous. After normalizing by $\beta$, the assignment $\omega\mapsto m_\omega$ is also continuous on $K$. Since $\rho$ was arbitrary, the result follows.
\end{proof}

		\subsection{Slice measures}\label{sec:slices}
		Let $C$ be a Riemann surface and $S$ a closed positive $(1,1)$-current with continuous potentials. Given an injective holomorphic immersion $\zeta:C \to X$ and a continuous local potential $u_S$ for $S$, the pull-back $\zeta^*S$ is a positive $(1,1)$-current defined locally by the potential $u_S\circ \zeta$, meaning 
	\begin{align*}
		(\zeta^* S \mid f)  = \int_C u_S \circ \zeta \, dd^c f 
	\end{align*} locally for test functions. 
		 Since $\zeta^*S$ is a positive current on $C$ of top degree, it defines a positive Radon measure. Abusing notation, we denote this measure and the pushed-forward measure on $\zeta(C)\cong C$ by $\restr{S}{\zeta(C)}$. This is called the \emph{slice} of $S$ by $\zeta(C)$, and, as a measure on $\zeta(C)$, it does not depend on the specific parametrization $\zeta$.
	
	Given a Borel-measurable set $A\subset C$, define a positive measure $\restr{S}{\zeta(A)}$ on $X$ by
	\begin{align*}
		\restr{S}{\zeta(A)} \coloneqq \mathds{1}_{\zeta(A)} 	\restr{S}{\zeta(C)}
	\end{align*} in the sense that
	\begin{align*}
		\int_{X} f \,d \big(\restr{S}{\zeta(A)}\big) \coloneqq 	\int_{\zeta(C)} f \mathds{1}_{\zeta(A)}\,d \big(\restr{S}{\zeta(C)}\big)
	\end{align*} for measurable $f\colon X\to \R$. If $A$ is contained in a compact subset of $C$, then $\restr{S}{\zeta(A)}$ is finite and Radon.

As suggested by the following lemma, slicing $S$ is essentially the same as wedging $S$ with a current of integration. 
\begin{lem}\label{lem:elementary}
	Let $\zeta:\C \to X$ be an injective holomorphic immersion, $\overline \Delta \subset \C$ a closed disk, $S$ a closed positive $(1,1)$-current with continuous potentials, and $\psi\colon X\to \R$ a smooth nonnegative function satisfying $\supp \psi \cap \zeta(\partial  \overline \Delta) = \varnothing$. Let $[\zeta(\overline \Delta)]$ denote the current of integration of $\zeta(\overline \Delta)$. Then
	\begin{align*}
		S \wedge \psi [\zeta(\overline \Delta)] = \psi \restr{S}{\zeta(\overline \Delta)}.
	\end{align*}
\end{lem}
\begin{proof}[Proof of Lemma \ref{lem:elementary}]
	Let $h\colon X\to \R$ be a test function on a neighborhood where $S = dd^c u_S$. Integrating $h$ over the left side gives 
	\begin{align*}
	\int_X h \,d  (S \wedge \psi [\zeta(\overline \Delta)] )&= \int_X h \cdot \psi \, d (S \wedge  [\zeta(\overline \Delta)] ) \\
	&= \int_X h \cdot  \psi \cdot \mathds{1}_{X\setminus  \zeta(\partial \overline \Delta)} \,d  (S \wedge  [\zeta(\overline \Delta)] ) \\
	&= \int_{\overline \Delta }(u_S \circ \zeta) dd^c((h\cdot \psi)\circ \zeta)  
\end{align*} where the second line holds by definition of the extension of $S \wedge  [\zeta(\overline \Delta)]$ to $X$, and the third line holds since $h \cdot \psi \cdot \mathds{1}_{X\setminus  \zeta(\partial \overline \Delta)} = h\cdot \psi$ is smooth. Similarly,
	\begin{align*}
	\int_X h\cdot  \psi \,d \big(\restr{S}{\zeta(\overline \Delta)}\big) &= \int_{\zeta(\C)} h \cdot \psi \cdot \mathds{1}_{\zeta(\overline \Delta)}\, d \big(\restr{S}{\zeta(\C)} \big)\\
	&= \int_{\C} (u_S\circ \zeta) \, dd^c((h \cdot \psi\cdot \mathds{1}_{\zeta(\overline \Delta)})\circ \zeta) \\
	&= \int_{\overline \Delta} (u_S\circ \zeta)  \, dd^c((h \cdot \psi)\circ \zeta)
		\end{align*}  where the second line holds since $(h\cdot  \psi \cdot \mathds{1}_{\zeta(\overline \Delta)})\circ \zeta$ is smooth on $\C$.
\end{proof}

We now prove two lemmas which will be useful in the next section.
	\begin{lem}\label{lem:weakconvergence}
		Let $(T_n)$ be a sequence of closed positive $(1,1)$-currents on $X$, and let $T$ be a closed positive $(1,1)$-current. Assume that $T_n$ (resp. $T$) has continuous normalized potential $\phi(T_n)$ (resp. $\phi(T)$), that $T_n\to T$ weakly, and that $\phi(T_n) \to \phi(T)$ uniformly. Let $U$ be an open subset of $\C$, and let $\zeta_n: U \to X$ be a sequence of holomorphic maps converging locally uniformly to a holomorphic map $\zeta:U\to X$. Then the measures $\zeta_n^*T_n$ converge weakly to $\zeta^*T$.
	\end{lem}
	\begin{proof}[Proof of Lemma \ref{lem:weakconvergence}]
		Let $V\subset X$ be biholomorphic to $\mathbb D\times \mathbb D$ and $h:\C\to \R$ a smooth test function whose support $K$ satisfies $\zeta(K)\subset V$. For all sufficiently large $n$, local uniform convergence gives $\zeta_n(K)\subset V$. On $V$, write
		\begin{align*}
			\restr{T_n}{V}=dd^c u_n,\hspace{1 cm} \restr{T}{V}=dd^c u,
		\end{align*}
		where $u_n$ (resp. $u$) is a continuous plurisubharmonic potential for $\restr{T_n}{V}$ (resp. $\restr{T}{V}$). Since $T_n\to T$ weakly, we have $\Theta(T_n)\to \Theta(T)$. The potentials $u_n$ and $u$ are obtained by adding $\phi(T_n)$ and $\phi(T)$ to local smooth potentials for $\Theta(T_n)$ and $\Theta(T)$, respectively. Hence $\phi(T_n)\to\phi(T)$ uniformly implies that $u_n\to u$ locally uniformly on $V$. Thus $u_n\circ \zeta_n\to u\circ \zeta$ uniformly on $K$, and by dominated convergence
		\begin{align*}
			\int h \, d \zeta_n^*T_n
			= \int u_n\circ \zeta_n\, dd^c h \to 	\int u\circ \zeta\, dd^c h =\int h\, d \zeta^*T.
		\end{align*}
	\end{proof}

	\begin{lem}\label{lem:weakconvergence-slices}
		In the setting of Lemma \ref{lem:weakconvergence}, let $\psi\colon U\to \R_{\geq 0}$ be smooth and compactly supported in $U$, so $\supp \psi\cap \partial U=\varnothing$. Then the measures
		\begin{align*}
			(\zeta_n)_*\bigl(\psi\,\zeta_n^*T_n\bigr)
		\end{align*}
		converge weakly to $\zeta_*(\psi\,\zeta^*T)$. Note here that these are measures on $X$. 
		
	\end{lem}
	\begin{proof}[Proof of Lemma \ref{lem:weakconvergence-slices}]
		Let $g\colon X\to\R$ be a smooth test function. As in Lemma \ref{lem:weakconvergence}, we work in a neighborhood where $T_n=dd^c u_n$ and $T=dd^c u$, with $u_n\circ \zeta_n\to u\circ \zeta$ uniformly on $\supp\psi$. Since $\zeta_n\to \zeta$ locally uniformly and the maps are holomorphic, the convergence is $C^\infty$ on $\supp\psi$. Hence $dd^c(\psi\cdot (g\circ \zeta_n))\to dd^c(\psi\cdot (g\circ \zeta))$ uniformly on $\supp\psi$. Therefore
		\begin{align*}
			\int_X g\, d(\zeta_n)_*\bigl(\psi\,\zeta_n^*T_n\bigr)
			=
			\int_U u_n\circ \zeta_n\, dd^c(\psi\cdot (g\circ \zeta_n)) 
			&\to
			\int_U u\circ \zeta\, dd^c(\psi\cdot (g\circ \zeta)) \\
			& =
			\int_X g\,d\zeta_*(\psi\,\zeta^*T).
		\end{align*}
	\end{proof}

	\subsection{Measurability of slices of limit currents along stable/unstable manifolds}
	Given an ergodic and hyperbolic measure $\nu \in S(\mu)$, Pesin theory (see Proposition 7.9 of \cite{MR4635340}, or Theorem 6.4 of \cite{MR3671937} for our skew-product setting) gives local stable and unstable manifolds
	\begin{align*}
		W^{s/u,\mathrm{loc}}_\omega(x)\subset W^{s/u}_\omega(x)
	\end{align*}
	for $\nu$-almost every $(\omega,x)$. In our setting, these	local manifolds are holomorphic disks, which may be controlled on Pesin sets.
		More precisely, for every $\epsilon>0$ there is a compact set $P\subset \Omega\times X$ with $\nu(P)>1-\epsilon$ such that the local stable and unstable disks have uniform size on $P$ and vary continuously there in the $C^1$ topology on disks, in the sense of Sections~7.4.1 and 7.4.3 of \cite{MR4635340}. 

	\begin{lem}\label{lem:measurable-slices}
		Let $\nu\in S(\mu)$ be ergodic and hyperbolic. Fix $\star\in\set{+,-}$ and $\bullet\in\set{s,u}$. Then the assignment to slices of $T_\omega^\star$ along the local $\bullet$-manifolds,
		\begin{align*}
			(\omega, x)\mapsto \restr{T_\omega^\star}{W^{\bullet,\mathrm{loc}}_\omega(x)},
		\end{align*}
		is $\nu$-measurable as a map into finite Borel measures on $X$ endowed with the weak topology. That is, the assignment above is a finite $\nu$-measurable kernel from $\Omega\times X$ to $X$. 
	\end{lem}
	\begin{proof}[Proof of Lemma \ref{lem:measurable-slices}]
		It is enough to prove the claim on compact sets of arbitrarily large $\nu$-measure. Fix $\epsilon>0$. By the Pesin-set statement above, Lusin's theorem applied to $\omega\mapsto T_\omega^\star$, and the joint continuity obtained in the proof of Lemma \ref{lem:equicontinuous} (with the analogous statement for $T_\omega^-$), choose a compact set $P\subset \Omega\times X$ with $\nu(P)>1-\epsilon$ on which the local $\bullet$-manifolds have uniform size and vary continuously, and such that on the projection of $P$ to $\Omega$ the maps
		\begin{align*}
			\omega\mapsto T_\omega^\star, \hspace{1 cm} \omega\mapsto \phi(T_\omega^\star)
		\end{align*}
		are continuous, where the first map is taken with respect to the weak topology on currents and the second with respect to the uniform topology on $C^0(X)$.

			Fix a finite set of charts as in Section~7.4.3 of \cite{MR4635340} covering $P$. In one such chart $V$, choose holomorphic parametrizations $\zeta_{(\omega,x)}:U\to X$ and a disk $\Delta\subset U$ whose closure is contained in $U$, such that $\zeta_{(\omega,x)}(\Delta)=W^{\bullet,\mathrm{loc}}_\omega(x)$ and $(\omega,x)\mapsto \zeta_{(\omega,x)}$ is continuous in the topology of local uniform convergence. Let $(\psi_m)$ be smooth nonnegative functions compactly supported in $\Delta$, viewed as functions on $U$, with $\psi_m\nearrow \mathds{1}_{\Delta}$. For fixed $m$ and $g\in C(X)$, Lemma \ref{lem:weakconvergence-slices} implies that
		\begin{align*}
			(\omega,x)\mapsto
			\int_X g\,d(\zeta_{(\omega,x)})_*
			\bigl(\psi_m\,\zeta_{(\omega,x)}^*T_\omega^\star\bigr)
		\end{align*}
			is continuous on $V\cap P$. By the definition of slices over Borel subsets in Section \ref{sec:slices} and monotone convergence, for $g\geq0$,
		\begin{align*}
			\int_X g\,d\restr{T_\omega^\star}{W^{\bullet,\mathrm{loc}}_\omega(x)}
			=
			\lim_{m\to\infty}
			\int_X g\,d(\zeta_{(\omega,x)})_*
			\bigl(\psi_m\,\zeta_{(\omega,x)}^*T_\omega^\star\bigr).
			\end{align*}
			The left-hand side is therefore measurable as a function on $P\cap V$. Since we cover $P$ with finitely many such charts, the left-hand side is measurable on all of $P$. Hence 
				\begin{align*}
				(\omega, x)\mapsto \restr{T_\omega^\star}{W^{\bullet,\mathrm{loc}}_\omega(x)},
			\end{align*} is measurable on $P$. Since $P$ may be chosen with arbitrarily large $\nu$-measure, the result follows.
	\end{proof}
	\subsection{Limit currents and stable/unstable manifolds}
		Throughout this section, we suppose $\nu\in S(\mu)$ is ergodic.
		\begin{lem}
			If $\nu$ gives positive measure to some $(\omega, x)$, then $(\omega, x)$ is periodic and $\nu$ is the uniform probability measure on the $F$-orbit of $(\omega, x)$.
		\end{lem}
		\begin{proof}
			Let $(\omega', x')$ be an atom of $\nu$ with maximal measure. Then each $F$-translate of $(\omega', x')$ has the same measure; thus, there are finitely many of them. By ergodicity, $\nu$ must be supported on the orbit of $(\omega', x')$. The point $(\omega, x)$ is contained in this orbit because it has positive measure, so the result follows.
		\end{proof}
	Recall the measures $\nu_\omega$ from Remark \ref{disintegration}, which by $F$-invariance satisfy 
	\begin{align}\label{one}
		F_*(\delta_{\omega}\otimes \nu_\omega ) = \delta_{\sigma(\omega)}\otimes \nu_{\sigma(\omega)}.
	\end{align}
	In the following lemmas, assume $\lambda^s(\nu) < 0 \leq \lambda^u(\nu)$.
	\begin{lem}\label{atoms}
			If $\nu_\omega$ has an atom on a positive $\mu^\Z$-measure set, then $\mu^\Z$-almost every $\nu_\omega$ is a uniform measure on a finite set.
	\end{lem}
	\begin{proof}
			Whenever $\nu_\omega$ is defined, let $M_\omega$ denote the set of atoms of $\nu_\omega$ with maximal $\nu_\omega$-measure. Since $\nu_\omega$ is a probability measure, each $M_\omega$ is finite. The same measurable-kernel argument used in the proof of Corollary \ref{lem:atomic}, applied to the measurable assignment $\omega\mapsto\nu_\omega$, shows that
			\begin{align*}
				(\omega,x)\mapsto \nu_\omega(\set{x})
			\end{align*}
			is $\nu$-measurable. Hence, the function
			\begin{align*}
				a(\omega)\coloneqq \sup_{x\in X}\nu_\omega(\set{x})
			\end{align*}
			is $\mu^\Z$-measurable, since it can be realized as the limit, as $r\to0$, of the supremum of $\nu_\omega(B_r(x))$ over a fixed countable dense subset of $X$. Hence
			\begin{align*}
				\bigcup_{\omega \in \Omega}(\set{\omega}\times M_\omega)
				=
				\set{(\omega,x)\colon \nu_\omega(\set{x})=a(\omega)}
			\end{align*}
			is $\nu$-measurable. Now \eqref{one} implies $F(\set{\omega} \times M_\omega) = \set{\sigma(\omega)} \times M_{\sigma(\omega)}$. So the set $\cup_{\omega \in \Omega} (\set{\omega} \times M_\omega)$ is $F$-invariant, and has positive measure by the assumption of the lemma. Thus, by ergodicity the lemma follows.
	\end{proof}

	\begin{lem}[Lemma 8.6 of \cite{MR4635340}]\label{key}
	For $\nu$-almost every $(\omega, x)$ and any embedded open disk $\Delta \subset  W_\omega^s(x)$, we have $\restr{T^+_{\omega}}{\Delta} = 0$. 
\end{lem}
Cantat and Dujardin's proof of Lemma \ref{key} translates to our setting without modification, so we omit it.
		\begin{lem}\label{well}
			Suppose $W^s_\omega(x)$ is not Zariski dense in $X$ on a positive $\nu$-measure set. Suppose further that for $\mu^\Z$-almost every $\omega$, $\nu_\omega$ is non-atomic. Then there exists an irreducible algebraic curve $D$ such that $ \nu(\Omega \times (\Gamma_\mu\cdot D)) = 1$. Moreover, for $\mu^\Z$-almost every $\omega$, $\nu_\omega$ is supported on a finite union of $\Gamma_\mu$-translates of $D$, and $\nu$-almost every $W^s_\omega(x)$ is contained in one of these translates.
	\end{lem}
	\begin{proof}
		Let $A\subset \Omega \times X$ denote a positive-measure set on which $W^s_\omega(x)$ is not Zariski dense. Then for each $(\omega, x)\in A$, the Zariski closure $D(\omega, x)$ of $W^s_\omega(x)$ is a (possibly singular) algebraic curve. Since $\NS(X,\Z)$ is countable, there exists some irreducible curve $D$ with $S = \set{(\omega, x)\in A\colon [D(\omega, x)] = [D]}$ having positive $\nu$-measure. The argument in \cite{MR4635340}, Lemma 7.1, shows that $D(\omega, x)$ is almost surely rational with negative self-intersection. Classes in $\NS(X,\Z)$ of such curves have unique representatives, so $[D(\omega, x)] = [D]$ implies $D(\omega, x) =  D$. 
		
		 Then the $F$-orbit of $S$ is $F$-invariant, so by ergodicity there is a full-measure subset of $\Omega \times X$ where $W^s_\omega(x)$ is contained in a $\Gamma_\mu$-translate of $D$. 
			For $\omega$ on which $\nu_\omega$ exists, let $R_\omega$ denote the union of $\Gamma_\mu$-translates of $D$ having maximal $\nu_\omega$-measure. Since $\nu_\omega$ is a probability measure, $R_\omega$ is nonempty generically. Moreover, since $\nu_\omega$ is finite, $R_\omega$ is a finite union; indeed, distinct translates overlap in $\nu_\omega$-measure zero, because algebraic curves have only finitely many intersection points and $\nu_\omega$ is non-atomic generically.
			
		By \eqref{one}, 
		\begin{align*}
			F(\set{\omega} \times R_\omega) = \set{\sigma(\omega)}\times R_{\sigma(\omega)}
		\end{align*}
		so $\cup_{\omega} \set{\omega} \times R_\omega$ is $F$-invariant. The family of distinct $\Gamma_\mu$-translates of $D$ is countable, and since $\omega\mapsto \nu_\omega$ is measurable, the set of translates having maximal $\nu_\omega$-measure varies measurably in $\omega$. By ergodicity, $\nu$ gives full measure to $\cup_{\omega} \set{\omega} \times R_\omega$, concluding the lemma.
	\end{proof}

	Using Lemma \ref{key} and the theory of laminar currents (specifically Lemma 8.8 of \cite{MR4635340}), Cantat and Dujardin prove the following theorem. We refer to Section~8 of \cite{MR4635340} for the definition of a normalized Ahlfors-Nevanlinna current. Again, their proof applies to our setting without modification.
		\begin{thm}[Theorem 8.2 of \cite{MR4635340}] \label{big}
		Let $\mu$ satisfy \eqref{int} and suppose $\nu\in S(\mu)$ is ergodic with $\lambda^s(\nu) < 0 \leq \lambda^u(\nu)$. Suppose that, $\nu$-almost surely, $W_\omega^s(x)$ is Zariski dense in $X$. Then $T_\omega^{+}$ is the unique Ahlfors-Nevanlinna current with unit mass associated to generic $W_\omega^s(x)$.
	\end{thm}

	\begin{rmk}
				 By Remark \ref{invertibility}, we obtain the analogous statement to Lemma \ref{well} in the unstable direction by applying the results to $\hat F$. Similarly, we obtain the analogous statement to Theorem \ref{big} for unstables and $T_\omega^-$.
	\end{rmk}
	\begin{rmk}\label{entropy_inverse} 
		From the formula \eqref{align:2}, one can show that $h_{\nu}^\mathcal{F} = 
			h_{\hat \nu}^\mathcal{F}$.
	\end{rmk}
	
	\begin{lem}	\label{posent}
		Suppose $h_{\nu}^\mathcal{F} > 0$. Then $\nu$ is hyperbolic, and $\nu$-almost surely the manifolds $W^{s/u}_\omega(x)$ are Zariski dense in $X$.
	\end{lem}
	\begin{proof}
		By the Margulis-Ruelle inequality in our setting, %
			$\lambda^u \geq \frac 1 2 h_{\nu}^\mathcal{F}  > 0$. By invertibility (Remark \ref{invertibility}) and recalling $h_{\nu}^\mathcal{F} = h_{\hat \nu}^\mathcal{F}$ (Remark \ref{entropy_inverse}), we obtain $\lambda^s  \leq -\frac 1 2 h_{\nu}^\mathcal{F} <  0$. So $\nu$ is hyperbolic. Now, by Corollary \ref{lem:atomic}, $h_{\nu}^\mathcal{F} > 0$ if and only if the conditional measures $\nu_{\eta_\omega(x)}$ are non-atomic almost everywhere. Therefore, by Lemma \ref{atoms}, $\nu_\omega$ is non-atomic $\mu^\Z$-almost everywhere. Indeed, if $\nu_\omega$ had atoms on a positive base set, Lemma \ref{atoms} would make the fiber measures uniform on finite sets, forcing subordinate conditionals to be atomic, contradicting Corollary  \ref{lem:atomic}. 
		
		Suppose that there exists a positive $\nu$-measure set on which the manifolds $W^{s}_\omega(x)$ are not Zariski dense in $X$. By Lemma \ref{well}, 
		 for $\mu^\Z$-almost every $\omega$, we have that $\nu_\omega$ is supported on an algebraic curve $C_{\omega}$ and, for $\nu_\omega$-almost every $x$, $W_\omega^s(x)$ is contained in an irreducible component of $C_\omega$. For such $\omega$, the singular points of the irreducible components of $C_\omega$, together with the intersections of distinct components, form a finite set. Since $\nu_\omega$ is non-atomic, $\nu_\omega$-almost every $x$ is therefore a smooth point of a unique irreducible component of $C_\omega$, which we denote by $D_{\omega,x}$.

	 	By holomorphicity, $W_\omega^{u}(x)$ is either also contained in $D_{\omega,x}$ or $W_\omega^{u}(x)\cap D_{\omega,x}$ is countable. But $W_\omega^s(x)$ and $W_\omega^u(x)$ cannot both be contained in $D_{\omega,x}$, since their tangent directions at the smooth point $x$ would then both equal $T_xD_{\omega,x}$, contradicting hyperbolicity. Therefore $W_\omega^{u}(x)\cap D_{\omega,x}$ is countable. 
	 	Moreover, $W_\omega^u(x)$ cannot be contained in any other irreducible component of $C_\omega$, since any component containing $W_\omega^u(x)$ would also contain $x$, contradicting the uniqueness of $D_{\omega,x}$. By holomorphicity, the intersection of $W_\omega^{u}(x)$ with each such component is therefore countable. So $W_\omega^{u}(x)\cap C_\omega$ is countable. Since $\nu_\omega(C_\omega) = 1$, disintegration implies that the unstable conditional at $x$ gives full measure to $W_\omega^u(x)\cap C_\omega$ and hence is atomic. But this would contradict Corollary \ref{lem:atomic}.
		
			 By invertibility (Remarks \ref{invertibility} and \ref{entropy_inverse}), the conclusion of the lemma also holds for unstable manifolds.
	\end{proof}

	\section{Convergence to limit currents}
	Let $C$ be a positive $(1,1)$-current and suppose $\psi\in C^\infty(X)$ is nonnegative and satisfies
	\begin{align}\label{nonint}
		\supp(\psi) \cap \supp(\partial C) = \varnothing.
	\end{align} For $f\in \Aut(X)$ with positive topological entropy, it is known that
	\begin{align*}
		\frac{1}{\lambda^n} f_*^{ \pm n}(\psi C) \to \left(\int_X(T_f^\pm \wedge  \psi C)\right)T_f^{\mp}
	\end{align*} in the weak topology as $n\to \infty$ \cite{MR1188127,fornaess_sibony_1994_complex_dynamics}.
	For each natural $n\geq 0$, let
	\begin{align*}
		C^n_\psi(\omega)\coloneqq (f_\omega^n)_* (\psi C). 
	\end{align*} 
	Recall from Section \ref{sec:wedges} that we may wedge $C_\psi^n(\omega)$ with closed positive currents having continuous potentials. The main result of this section is the following. 
	\begin{prop}\label{lem:last2}  Take $C$, $\psi$, and $C_\psi^n(\omega)$ as above. For $\mu^\Z$-almost every $\omega\in \Omega$ and every $\rho > 0$, one can find a subsequence $(\tilde n_k)_{k\in \N}$ of upper-density greater than $1 - \rho$ for which
			\begin{equation}\label{important:seq}
			\begin{aligned}
		&\left( \int_X T_{\omega}^+ \wedge \psi C \right)T_{\sigma^{\tilde n_k}(\omega)}^- \\
		 &\hspace{3 cm} - \left(\int_X T_{\sigma^{\tilde n_k}(\omega)}^+ \wedge T_{\sigma^{\tilde n_k}(\omega)}^-\right) M((f_\omega^{\tilde n_k})_* T_\omega^+) C_\psi^{\tilde n_k}(\omega) \to 0
			\end{aligned}
		\end{equation} in the weak sense of currents as $k\to \infty$. Furthermore, 
	\begin{equation}\label{important:currs}
		\begin{aligned}
				&\left( \int_X T_{\omega}^+ \wedge \psi C \right)T_{\sigma^{\tilde n_k}(\omega)}^+ \wedge T_{\sigma^{\tilde n_k}(\omega)}^-  \\ & \hspace{3 cm} - \left(\int_X T_{\sigma^{\tilde n_k}(\omega)}^+ \wedge T_{\sigma^{\tilde n_k}(\omega)}^-\right) (f_\omega^{\tilde n_k})_*(T_{\omega}^+ \wedge \psi C) \to 0
		\end{aligned}
	\end{equation}
	in the weak sense of measures as $k\to \infty$.	If the exponential moment condition \eqref{exponential} is satisfied, one can additionally require the lower-density of $(\tilde n_k)_{k\in \N}$ to be greater than $1 - \rho$.
	\end{prop}
	\begin{rmk}\label{rmk:prop31-independent}
	The subset of $\Omega$ on which Proposition \ref{lem:last2} holds may be chosen independently of $C$ and $\psi$.
\end{rmk}

	\begin{rmk}\label{rmk:identity}
		Notice 
		\begin{align*}
			T_{\sigma^{n}(\omega)}^+ \wedge M(( f_\omega^n)_*T_{\omega}^+) C_\psi^{n} &= T_{\sigma^{n}(\omega)}^+ \wedge M(( f_\omega^n)_*T_{\omega}^+) ( f_\omega^n)_*(\psi C) \\
			&= (f_\omega^n)_*(T_{\omega}^+ \wedge \psi C).
		\end{align*} Integrating, we find 
		\begin{align*}
			\int_X T_{\sigma^{n}(\omega)}^+ \wedge M(( f_\omega^n)_*T_{\omega}^+) C_\psi^{n} = \int_X (f_\omega^n)_*(T_{\omega}^+ \wedge \psi C) = \int_X T_{\omega}^+ \wedge \psi C
		\end{align*} so \eqref{important:seq} may be rewritten as 
		\begin{align*}
			&	\left(\int_X T_{\sigma^{\tilde n_k}(\omega)}^+ \wedge M(( f_\omega^{\tilde n_k})_*T_{\omega}^+) C_\psi^{\tilde n_k}\right)T_{\sigma^{\tilde n_k}(\omega)}^- \\
			&\hspace{3.5 cm} - \left(\int_X T_{\sigma^{\tilde n_k}(\omega)}^+ \wedge T_{\sigma^{\tilde n_k}(\omega)}^-\right) M((f_\omega^{\tilde n_k})_* T_\omega^+) C_\psi^{\tilde n_k}(\omega) \to 0.
		\end{align*}
	\end{rmk}
	\begin{rmk}\label{tbd}
		Since $\mu_{\mathcal{F}}$ assigns zero measure to proper subspaces of $\P(\NS(X)_+)$ (by strong irreducibility) and $e^+(\eta)$ is independent of $e^-(\eta)$, we have 
		\begin{align*}
			D_\epsilon \coloneqq \set{\eta \colon e^+(\eta)\wedge e^-(\eta) > \epsilon}
		\end{align*} satisfies 
		\begin{align*}
			\mu^\Z(D_\epsilon) \to 1 
		\end{align*} as $\epsilon \to 0$. Therefore by the ergodic theorem we may assume that 
		\begin{align*}
			 \left(\int_X T_{\sigma^{\tilde n_k}(\omega)}^+ \wedge T_{\sigma^{\tilde n_k}(\omega)}^-\right)  =  e^+(\sigma^{\tilde n_k}(\omega))\wedge e^-(\sigma^{\tilde n_k}(\omega)) 
		\end{align*} is bounded away from $0$. 
		\end{rmk}
\begin{rmk} \label{tbd2} By Remark \ref{tbd}, it follows from \eqref{important:currs} that 
	\begin{align*}
		\frac{\left( \int_X T_{\omega}^+ \wedge \psi C \right)}{\left(\int_X T_{\sigma^{\tilde n_k}(\omega)}^+ \wedge T_{\sigma^{\tilde n_k}(\omega)}^-\right)}T_{\sigma^{\tilde n_k}(\omega)}^+ \wedge T_{\sigma^{\tilde n_k}(\omega)}^-  -  (f_\omega^{\tilde n_k})_*(T_{\omega}^+ \wedge \psi C) 
	\end{align*} converges weakly to the zero measure as $k\to \infty$. 
	Assuming 
	$\int_X T_\omega^+ \wedge \psi C \neq 0$, we obtain 
			\begin{align*}
			m_{\sigma^{\tilde n_k}(\omega)}  -  (f_\omega^{\tilde n_k})_* \cN(T_{\omega}^+ \wedge \psi C) \to 0.
		\end{align*} 
\end{rmk}

	Before proving Proposition \ref{lem:last2}, we deduce a few supporting results. 
	
	\subsection{Supporting results}
	Here we prove a few results that will be useful for the remainder of the section. We assume $C$ is a positive $(1,1)$-current and $\psi\in C^\infty(X)$ is nonnegative and satisfies \eqref{nonint}. The following lemma echoes Lemma \ref{lem:currClosed}.
	\begin{lem}\label{lem:curr}
		 Suppose $\omega\in \Omega$, $\set{a_k}_{k\in \N}\subset \R$, and $(n_k)$ are such that $a_kC_\psi^{n_k}(\omega)$ converges weakly as $k\to \infty$ to a closed positive current $T$ with continuous potentials. Suppose further that $a_k\to 0$ and there exists nonnegative $\psi_0\in C^\infty(X)$ with $\psi_0 \equiv 1$ on $\supp \psi$, $\supp \psi_0 \cap \supp \partial C = \varnothing$, and $a_k(C_{\psi_0}^{n_k}(\omega)\mid \kappa_0)^{\frac 1 2}\to 0$.
	 Let $(B_k)$ be a sequence of closed positive $(1,1)$-currents such that $\set{[B_k]}$ is bounded in cohomology and $\set{\phi(B_k)}$ is an equicontinuous family admitting a uniform $C^0$ bound. Then 
	 \begin{align*}
	 	(a_k C_\psi^{n_k}(\omega) \wedge B_k)  - (T \wedge B_k ) \to 0
	 \end{align*} in the weak sense of measures. 
	\end{lem}
\begin{proof}[Proof of Lemma \ref{lem:curr}]
	Let $g \in C^\infty(X)$ be supported in an open bidisk $V\subset X$. Let $u_{B_k}, u_{\Theta(B_k)}$ denote the local potentials of $B_k, \Theta(B_k)$ on $V$ specified in the proof of Lemma \ref{lem:currClosed}, so that $u_{B_k} = u_{\Theta(B_k)} + \phi(B_k)$. As in the proof of Lemma \ref{lem:currClosed}, the family $\set{u_{B_k}dd^cg}$ is precompact in the $C^0$-topology on coefficient functions.
	
	Then, recalling Section \ref{sec:wedges}, we have
	 \begin{align*}
	 	(a_k C_\psi^{n_k}(\omega) \wedge B_k \mid g)& = a_k( (f_\omega^{n_k})_*(\psi C) \wedge B_k \mid g) \\
	 	& = a_k( (\psi C) \wedge (f_\omega^{n_k})^* B_k \mid g \circ f_\omega^{n_k} )\\
	 	& = a_k( C \wedge (f_\omega^{n_k})^* B_k \mid \psi \cdot (g \circ f_\omega^{n_k}) )\\  
	 	 &= a_k(  C  \mid (u_{B_k} \circ f_\omega^{n_k} )dd^c(\psi \cdot (g\circ f_\omega^{n_k})) )\\  
	 	 &= (a_k C_\psi^{n_k}(\omega)  \mid u_{B_k}dd^c g) + a_k e_{k}(g)
	 \end{align*}  where
 	\begin{align*}
 	e_{k}(g) &=  \big(  C  \mid (u_{B_k} \circ f_\omega^{n_k} )dd^c(\psi \cdot (g\circ f_\omega^{n_k})) \big) -  \big(C_\psi^{n_k}(\omega)  \mid u_{B_k}dd^c g\big).
 \end{align*} Since the $C_\psi^{n_k}(\omega)$'s are positive currents, the local coefficient measures of $a_k C_\psi^{n_k}(\omega)$ in $V$ converge to the corresponding coefficient measures of $T$. By precompactness of the family $\set{u_{B_k}dd^cg}$, it follows that $(a_k C_\psi^{n_k}(\omega) \mid u_{B_k}dd^c g) - (T \mid u_{B_k}dd^c g)\to 0$. Since $(T \mid u_{B_k}dd^c g) = (T \wedge B_k \mid g)$, it only remains to prove that $a_ke_{k}(g)\to 0$.

 Indeed,
 	\begin{align*}
 		e_{k}(g) &=  \big(  C  \mid (u_{B_k} \circ f_\omega^{n_k} )dd^c(\psi \cdot (g\circ f_\omega^{n_k})) \big) -  \big(C_\psi^{n_k}(\omega)  \mid u_{B_k}dd^c g\big) \\
 		&=  \big(  C  \mid (u_{B_k} \circ f_\omega^{n_k} )dd^c(\psi \cdot (g\circ f_\omega^{n_k})) \big) -  \big(C\mid   (u_{B_k} \circ f_\omega^{n_k}) \psi dd^c (g\circ f_\omega^{n_k})\big)\\
 		&=  \big(C  \mid (u_{B_k} \circ f_\omega^{n_k} ) d\psi \wedge d^c (g\circ f_\omega^{n_k})\big) + \big(C  \mid (u_{B_k} \circ f_\omega^{n_k} )  d(g\circ f_\omega^{n_k})\wedge d^c \psi \big) \\ &\hspace{7.1 cm }+ \big(C  \mid (u_{B_k} \circ f_\omega^{n_k} ) (g\circ f_\omega^{n_k}) dd^c \psi\big).
	\end{align*} Now, since $\set{\norm{u_{B_k}}_{C^0}}$ is uniformly bounded, there exists $A\in \R$ such that $-  A\kappa_0 \leq (u_{B_k} \circ f_\omega^{n_k} ) (g\circ f_\omega^{n_k}) dd^c \psi \leq A\kappa_0$ for all $k$; hence, by the positivity of $C$, we have \[|\big(C \mid  (u_{B_k} \circ f_\omega^{n_k} ) (g\circ f_\omega^{n_k}) dd^c \psi\big)| \leq A (C\mid \kappa_0)\] for all $k$. So $a_k\to 0$ implies \[a_k \big(C  \mid (u_{B_k} \circ f_\omega^{n_k} ) (g\circ f_\omega^{n_k}) dd^c \psi\big) \to 0.\] Now, 
 	\begin{align*}
 		&|\big(C  \mid (u_{B_k} \circ f_\omega^{n_k} ) d\psi \wedge d^c (g\circ f_\omega^{n_k})\big)|  \\ & \hspace{3 cm }=|\big(\psi_0 C  \mid (u_{B_k} \circ f_\omega^{n_k} ) d\psi \wedge d^c (g\circ f_\omega^{n_k})\big)| 
 		\\ & \hspace{3 cm }\leq
		 \norm{u_{B_k}}_{C^0} (\psi_0 C\mid i \partial \psi \wedge \overline \partial \psi)^{\frac 1 2}(\psi_0 C\mid i \partial (g\circ f_\omega^{n_k}) \wedge \overline \partial (g\circ f_\omega^{n_k}))^{\frac 1 2}\\ & \hspace{3 cm }\leq \norm{u_{B_k}}_{C^0} (\psi_0 C\mid i \partial \psi \wedge \overline \partial \psi)^{\frac 1 2}(\psi_0 C\mid (f_\omega^{n_k})^*(i \partial g \wedge \overline \partial g) )^{\frac 1 2}
 	\end{align*} where in the first equality we used $\psi_0 \equiv 1$ on $\supp \psi$ and afterwards we use positivity of $\psi_0 C$ and Cauchy--Schwarz. Choosing $B\in \R$ so that $-B\kappa_0 \leq i \partial g \wedge \overline \partial g\leq B\kappa_0$, we have 
 	\begin{align*}
	&|\big(C  \mid (u_{B_k} \circ f_\omega^{n_k} ) d\psi \wedge d^c (g\circ f_\omega^{n_k})\big)| \\& \hspace{3 cm}\leq B^{\frac 1 2} \norm{u_{B_k}}_{C^0} (\psi_0 C\mid i \partial \psi \wedge \overline \partial \psi)^{\frac 1 2} (\psi_0 C\mid (f_\omega^{n_k})^*(\kappa_0) )^{\frac 1 2}.
 \end{align*}  Since $(\psi_0 C\mid (f_\omega^{n_k})^*(\kappa_0) ) = (C_{\psi_0}^{n_k}(\omega) \mid\kappa_0)$ and $a_k (C_{\psi_0}^{n_k}(\omega) \mid\kappa_0)^{\frac 1 2} \to 0$ by assumption, and $\set{\norm{u_{B_k}}_{C^0}}$ is uniformly bounded, we conclude \[a_k\big(C  \mid (u_{B_k} \circ f_\omega^{n_k} ) d\psi \wedge d^c (g\circ f_\omega^{n_k})\big) \to 0.\] The proof that 
	 $a_k\big(C  \mid (u_{B_k} \circ f_\omega^{n_k} )  d(g\circ f_\omega^{n_k})\wedge d^c \psi \big) \to 0$ is similar. Hence we have shown $a_ke_{k}(g)\to 0$, concluding the proof.
\end{proof} 
		\begin{prop}\label{lem:MassBound2}
		For $\mu^\Z$-almost every $\omega\in \Omega$ and every $\rho > 0$, one can find a subsequence $(n_i)_{i\in \N}$ of upper-density greater than $1 - \rho$  for which
		\begin{align*}
			M((f_\omega^{n_i})_* T_\omega^+)M(C^{n_i}_\psi)
		\end{align*}
		 is bounded uniformly in $i$. If the exponential moment condition \eqref{exponential} is satisfied, then the subsequence $(n_i)_{i\in \N}$ has lower-density greater than $1 - \rho$.
		\end{prop} 
An immediate corollary is the following.
	\begin{cor}
			For any $\rho>0$, if $\omega$ and $(n_i)$ are as in Proposition \ref{lem:MassBound2}, the collection $\set{M((f_\omega^{n_i})_* T_\omega^+)C_\psi^{n_i}}$ is precompact in the weak topology on currents.
	\end{cor}
	\begin{rmk}\label{rmk:tightness}
		Without the exponential moment assumption, the sequence $(n_i)_{i\in\N}$ in Proposition \ref{lem:MassBound2} is chosen so that it satisfies the tightness conditions of Gou\"ezel--Karlsson, namely condition (2) of Theorem \ref{thm:GoodTimes}. Under the exponential moment assumption, we produce our sequence using Proposition \ref{prop:exponentialC0}, and tightness is not guaranteed. 
	\end{rmk}
		\begin{prop}\label{lem:closed}
			Let $\omega\in\Omega$ be generic in the sense of Proposition \ref{lem:MassBound2}. Given $\rho>0$, let $(n_i)_{i\in\N}$ be a sequence obtained from the construction in its proof. Then every limit point of 
		\begin{align*}
			\set{M((f_\omega^{n_i})_* T_\omega^+)C_\psi^{n_i}(\omega)}
		\end{align*}
		is closed.
		\end{prop}

	  	Before proving Propositions \ref{lem:MassBound2} and \ref{lem:closed}, we prove Lemma \ref{lem:MassBound} and Claim  \ref{claim:mass}. 
		
	\begin{lem}\label{lem:MassBound}
	For $\mu^\Z$-almost every $\omega\in \Omega$ and every $\rho > 0$, one can find a subsequence $(n_i)_{i\in \N}$ of upper-density greater than $1 - \rho$  for which
		\[\frac{1}{\norm{(f_\omega^{n_i})^*}}M(C^{n_i}_\psi)\]
		 is bounded uniformly in $i$. Furthermore, if the exponential moment condition \eqref{exponential} is satisfied, one can take $\set{n_i}_{i\in \N} = \N$. 
		\end{lem} 
		\begin{rmk}
			Proposition \ref{lem:MassBound2} is identical to Lemma \ref{lem:MassBound}, except we replace $\norm{(f_\omega^{n_i})^*}^{-1}$ with $M((f_\omega^{n_i})_* T_\omega^+)$.
		\end{rmk}
		
		Now we explain how to obtain Lemma \ref{lem:MassBound} from Lemma \ref{lem:C0Bound} (i.e. \cite{MR4635340}, Lemma 6.15) and elementary estimates. 
		\begin{proof}[Proof of Lemma \ref{lem:MassBound}]
		Let $\omega$ be generic in the sense of Corollary \ref{cor:C0Bound}, and further require that $e^+(\omega)$ exists. Let $\rho>0$, and let $(n_i)$ be the subsequence obtained in Corollary \ref{cor:C0Bound}. Then
		\begin{align*}
			\frac{1}{\norm{(f_\omega^{n_i})^*}}M(C^{n_i}_\psi(\omega)) &= \frac{1}{\norm{(f_\omega^{n_i})^*}}( C^{n_i}_\psi(\omega) \mid \kappa_0) \\
			&=\frac{1}{\norm{(f_\omega^{n_i})^*}} (\psi C \mid  (f_\omega^{n_i})^*\kappa_0) \\
			&=\frac{1}{\norm{(f_\omega^{n_i})^*}} (\psi C \mid  \Theta((f_\omega^{n_i})^*\kappa_0)) \\
			&\hspace{2 cm} +  \frac{1}{\norm{(f_\omega^{n_i})^*}}(\psi C \mid dd^c \phi((f_\omega^{n_i})^*\kappa_0)).
		\end{align*}  Since $e^+(\omega)$ exists, there exists a bounded subset of $H^{1,1}$ that contains
	\begin{align*}
		 \frac{1}{\norm{(f_\omega^n)^*}}(f_\omega^n)^* [\kappa_0]
	\end{align*} for all $n$. 
 Hence $\frac{1}{\norm{(f_\omega^{n_i})^*}} (\psi C \mid  \Theta((f_\omega^{n_i})^*\kappa_0)) $ is bounded by a finite constant independent of $i$. Now
		\begin{align*}
			\frac{1}{\norm{(f_\omega^{n_i})^*}}(\psi C \mid dd^c \phi((f_\omega^{n_i})^*\kappa_0)) &= \frac{1}{\norm{(f_\omega^{n_i})^*}}(C \mid (dd^c \psi) \cdot  \phi((f_\omega^{n_i})^*\kappa_0)) \\
			&\leq  \frac{1}{\norm{(f_\omega^{n_i})^*}}A_\psi (C\mid \kappa_0) \norm{\phi((f_\omega^{n_i})^*\kappa_0)}_{C^0}
		\end{align*} where in the first line we used Stokes' theorem and $\supp \psi \cap \supp(\partial C) = \varnothing$ and in the last line we chose $A_\psi > 0$ such that $-A_\psi \kappa_0 \leq dd^c \psi \leq A_\psi \kappa_0$. But the right side is uniformly bounded by Corollary \ref{cor:C0Bound}, concluding the proof.
	
		Since $(n_i)$ is the subsequence obtained in Corollary \ref{cor:C0Bound}, if $\mu$ satisfies the exponential moment condition \eqref{exponential} then $\set{n_i}_{i\in \N} = \N$ by Proposition \ref{prop:exponentialC0}.
	\end{proof}

\begin{rmk}\label{rmk:equiv}
		Using equivariance of the limit currents (see Remark \ref{equivariance}), we have
	\begin{align*}
		M((f_\omega^n)_* T_\omega^+) = \frac{1}{M((f_\omega^n)^* T_{\sigma^n(\omega)}^+)} \asymp \frac{1}{\norm{(f_\omega^n)^* e^+(\sigma^n(\omega))}} 
	\end{align*} where we are using that $\norm{\cdot} \asymp M(\cdot)$ on $\Psef(X)$. 
\end{rmk}

	\begin{claim}\label{claim:mass}
		There exists a set $\Lambda \subset \Omega$ with full $\mu^\Z$-measure such that the following holds: for every
		$\omega\in \Lambda$ and $\delta>0$ there exists a constant $\epsilon> 0$ and a subsequence $(n_k)_{k\in \N}$ of $\N$ such that
		\begin{enumerate}
			\item for all $k$, 
			\begin{align*}
				\norm{(f_\omega^{n_k})^* e^+(\sigma^{n_k}(\omega))} \geq \epsilon \norm{(f_\omega^{n_k})^*},
			\end{align*} and
			\item $\set{n_k}_{k\in \N}$ has lower density greater than $1-\delta$, i.e.
			\begin{align*}
				\liminf_{N\to \infty} \frac 1 N \#\bigl(\set{n_k}_{k\in \N} \cap [1, N - 1]\bigr) > 1-\delta.
			\end{align*}  
		\end{enumerate}
	\end{claim}

	\begin{proof}[Proof of Claim \ref{claim:mass}]
		Given $A\in \Orth^+(H^{1,1})$ with $\norm{A}>1$, as in Section \ref{sec:adapted} we let $e_+(A)$ and $e_-(A)$ denote unit-norm singular vectors in the closure of the positive cone corresponding respectively to the largest and smallest singular values of $A$. If $\norm{A}=1$, set $e_+(A)=e_-(A)=\frac{[\kappa_0]}{\norm{[\kappa_0]}}$, where $[\kappa_0]$ is the cohomology class of our fixed K\"ahler form with unit mass.
		
		Let $\Lambda\subset \Omega$ be a full-measure, $\sigma$-invariant set such that if $\omega\in \Lambda$ then $e^+(\omega)$ exists (see Theorem \ref{FurstenbergLimit}). We will show that $\Lambda$ satisfies the conclusion of the claim. Fix $\omega \in \Lambda$.  By definition,
		\begin{align*}
			(f_\omega^{n})^* e^+(\sigma^{n}(\omega)) =  \omega_{0}^*\omega_1^* \dots \omega_{n - 1}^* e^+(\sigma^{n}(\omega)).
		\end{align*}
	Recall that in Section \ref{sec:actionCohomology} we let $\angle{,}_{\mathrm{ad}}$ denote a positive-definite inner product adapted to the intersection form. Let 
	\begin{align*}
		\tau([w]_{\P},[v]_{\P}) \coloneqq \frac{|\angle{w,v}_{\mathrm{ad}}|}{\norm{w}\norm{v}}
	\end{align*} for $[w]_{\P},[v]_{\P}\in \P H^{1,1}$.	For each $n\in \N$, define $h_n:\Lambda \to \R$ by
		\begin{align*}
			h_n(\omega) \coloneqq \tau([e_+((f_\omega^n)^*)]_{\P}, [e^+(\sigma^n(\omega))]_{\P}).
		\end{align*} 
	For each $n$, there exists a unique $w \in H^{1,1}$ with $\angle{w, e_+((f_\omega^n)^*)}_{\mathrm{ad}} = 0$ and \[e^+(\sigma^{n}(\omega)) = \angle{e_+((f_\omega^n)^*), e^+(\sigma^n(\omega))}_{\mathrm{ad}} e_+((f_\omega^n)^*) + w.\] By singular value decomposition, since $w$ is orthogonal to $e_+((f_\omega^n)^*)$ it satisfies $\norm{(f_\omega^n)^* w} \leq \norm{w} \leq \norm{e^+(\sigma^n(\omega))}$.
	Hence
	\begin{align*}
		\norm{(f_\omega^{n})^* e^+(\sigma^{n}(\omega))} \geq \norm{e^+(\sigma^{n}(\omega))}\left(h_n(\omega)	\norm{(f_\omega^{n})^*} - 1\right).
	\end{align*} 
	Since every limit class has unit mass, there exist $c,C >0$ so that $c \leq \norm{e^+(\sigma^{n}(\omega))} \leq C$ for all $\omega$ and $n$. Thus 
	\begin{align*}
		\norm{(f_\omega^{n})^* e^+(\sigma^{n}(\omega))} \geq c h_n(\omega)	\norm{(f_\omega^{n})^*} - C.
	\end{align*}
	We prove that
		\begin{align}\label{delta_eps}
			\delta_\epsilon\coloneqq \limsup_{N\to \infty} \frac 1 N \# \set{n\leq N \colon h_n(\omega) < \epsilon}
		\end{align} satisfies $\lim_{\epsilon \to 0} \delta_\epsilon = 0$, from which the claim follows by taking 
		\[\set{n_k}_{k\in \N} = \set{n\in \N \colon h_n(\omega) \geq \epsilon},\] noting $\norm{(f_\omega^{n})^*}  \to \infty$ and shrinking $\epsilon$. 
	
		Notice that
		\begin{equation} \label{ineq}
	\begin{aligned}
		&|h_n(\sigma^k(\omega)) - h_{n+k}(\omega)| \\ & \hspace{1.3 cm}\leq \left|\bigg \langle e_+(\omega_k^* \dots \omega_{n+k-1}^*) -e_+(\omega_0^*\dots \omega_{{n+k-1}}^*), \frac{e^+(\sigma^{n+k}(\omega))}{ \norm{ e^+(\sigma^{n+k}(\omega))} } \bigg \rangle_{\mathrm{ad}}\right | \\
		&\hspace{1.3 cm} \leq \norm{e_+(\omega_k^* \dots \omega_{n+k-1}^*) -e_+(\omega_0^* \dots \omega_{n+k-1}^*)} 
	\end{aligned}
\end{equation}
		Furthermore, the sequence $(e_+(\eta_{-n}^*\dots \eta_{-1}^*))_{n\in \N}$ converges for $\mu^\Z$-almost every $\eta \in \Omega$. Indeed, recall from Section \ref{sec:adapted} that each expanding singular direction lives in $\NS(X)_+$, so we may consider the product $\eta_{-n}^*\dots \eta_{-1}^*$ as a product of matrices in $\Orth^+(\NS(X)_+)$. Now the random action induced by $\mu$ on $\NS(X)_+$ is strongly irreducible and contracting in the sense of \cite{MR886674}, so by \cite{MR886674}, Proposition III.3.2, the sequence $(e_+(\eta_{-n}^*\dots \eta_{-1}^*))_{n\in \N}$ converges in direction. But $\Aut(X)$ preserves the K\"ahler cone and each vector has unit norm, so the sequence converges to a vector. Therefore, for fixed $\epsilon > 0$, the sets
		\begin{align*}
			S_\epsilon ^n\coloneqq \set{\eta \in \Omega \colon \norm{e_+(\eta_{-\ell}^* \dots \eta_{-1}^*) -e_+(\eta_{-m}^* \dots \eta_{-1}^*)} <\frac \epsilon  C \text{ for all } \ell,m \geq n}
		\end{align*} satisfy $S_\epsilon^n\subset S_\epsilon^{n+1}$ and $ \mu^\Z(S_\epsilon^n)\to 1$. 
			After shrinking $\Lambda$, we may assume that $\omega$ is Birkhoff-generic for the countable collection of indicator functions used below. For all $n\in \N$, from the Birkhoff ergodic theorem we have
		\begin{align}\label{sum}
			\frac 1 M \sum_{k= 1}^M \mathds{1}_{S_\epsilon^n}(\sigma^{n+k}(\omega)) \to \mu^\Z(S_\epsilon^n).
		\end{align}
		Now 
		\begin{equation}\label{ineqq2}
			\begin{aligned}
				 & \liminf_{M \to \infty} \frac 1 M \# \set{k \leq M \colon |h_n(\sigma^k(\omega)) - h_{n+k}(\omega)|  < \epsilon} \\
				& \hspace{.2 cm}\geq  \liminf_{M \to \infty} \frac 1 M \# \set{k \leq M \colon \norm{e_+(\omega_k^* \dots \omega_{n+k-1}^*) -e_+(\omega_0^* \dots \omega_{n+k-1}^*)} <  \epsilon}\\ 
					& \hspace{.2 cm} \geq  \liminf_{M \to \infty} \frac 1 M \sum_{k=1}^M \mathds{1}_{S_\epsilon^n}(\sigma^{n+k}(\omega))  = \mu^\Z(S_\epsilon^n)
			\end{aligned}
		\end{equation} where the first inequality follows from \eqref{ineq}, the second from the definition of $S_\epsilon^n$, and the equality in the last line from \eqref{sum}.
			Moreover, 
		\begin{equation}\label{ineqq3}
			\begin{aligned}
				&  \set{k \leq M \colon |h_n(\sigma^k(\omega))| < 2\epsilon} \\
				& \hspace{1.5 cm}\supset \set{k \leq M \colon |h_{n+k}(\omega)| < \epsilon \text{ and } |h_n(\sigma^k(\omega)) - h_{n+k}(\omega)| < \epsilon }.
			\end{aligned}
		\end{equation}
			Since $h_n$ depends only on the future, we may integrate it against $\mu^\N$. Putting these together, 
 		\begin{align*}
 			 & \int \mathds{1}_{\set{h_n<  2\epsilon}}(\eta)   \,d\mu^\N(\eta) = \limsup_{M \to \infty}  \frac 1 M \sum_{k = 1}^M \mathds{1}_{\set{h_n < 2\epsilon}}(\sigma^k(\omega)) \\
 			 	 &  \hspace{1 cm} \geq \limsup_{M \to \infty} \frac 1 M \# \set{k \leq M \colon |h_{n+k}(\omega)| < \epsilon \text{ and } |h_n(\sigma^k(\omega)) - h_{n+k}(\omega)| < \epsilon }   \\
 			 &  \hspace{1 cm} \geq \limsup_{M \to \infty} \frac 1 M \bigg(\# \set{k \leq M \colon |h_{n+k}(\omega)| < \epsilon} \\
 			 & \hspace{4 cm} - \# \set{k \leq M  \colon |h_n(\sigma^k(\omega)) - h_{n+k}(\omega)| \geq \epsilon }\bigg)   \\
 			  &  \hspace{1 cm} \geq \limsup_{M \to \infty} \frac 1 M \# \set{k \leq M \colon |h_{n+k}(\omega)| < \epsilon}  \\
 			  &\hspace{3 cm}- \limsup_{M \to \infty} \frac 1 M \# \set{k \leq M  \colon |h_n(\sigma^k(\omega)) - h_{n+k}(\omega)| \geq \epsilon }   \\
 			 	&  \hspace{1 cm} =  \left(\limsup_{M \to \infty} \frac 1 M \# \set{k \leq M \colon |h_{n+k}(\omega)| < \epsilon}  \right)  \\ 
 			 	& \hspace{3 cm} - \left(1 -  \liminf_{M \to \infty} \frac 1 M \# \set{k \leq M \colon |h_n(\sigma^k(\omega)) - h_{n+k}(\omega)| < \epsilon} \right)\\
 			 	& \hspace{1 cm} \geq 	\left(\limsup_{M \to \infty} \frac 1 M \# \set{k \leq M \colon |h_{n+k}(\omega)| < \epsilon}  \right) - 1 + \mu^\Z(S_\epsilon^n)  
		\end{align*}  where in the first line we use the Birkhoff ergodic theorem, the second line we used \eqref{ineqq3}, and the last line we used \eqref{ineqq2}.
	Now, recalling \eqref{delta_eps} and allowing a harmless shift by $n$,
		\begin{align*}
			\delta_\epsilon  = \limsup_{M \to \infty} \frac 1 M \# \set{k \leq M \colon |h_{n+k}(\omega)| < \epsilon},
			\end{align*} for all $n\in \N$, so we have shown
			\begin{align}\label{close}
			&\delta_\epsilon - 1 + \mu^\Z(S_\epsilon^n) \leq  \int \mathds{1}_{\set{h_n<  2\epsilon}}(\eta) \,d\mu^\N(\eta).
		\end{align} 
		Finally, taking $ \mathds{1}_{\set{\tau<  2\epsilon}}$ to be the indicator function for
			\[\set{(v,w)\in \P H^{1,1}\times \P H^{1,1}\mid \tau(v,w)<  2\epsilon},\]
		we have
		\begin{align*}
			\int\mathds{1}_{\set{h_n<  2\epsilon}}(\eta)  \,d\mu^\N(\eta)
			&= \int \mathds{1}_{\set{\tau<  2\epsilon}}([e_+((f_\eta^n)^*)]_{\P}, [e^+(\sigma^n(\eta))]_{\P}) \,d\mu^\N(\eta) \\
			&= \int \int \mathds{1}_{\set{\tau<  2\epsilon}}([e_+((f_\eta^n)^*)]_{\P}, [e^+(\xi)]_{\P}) \,d\mu^\N(\xi) \,d\mu^\N(\eta) \\
			&= \int \int \mathds{1}_{\set{\tau<  2\epsilon}}([e_+((f_\eta^n)^*)]_{\P}, [v]_\P) \,d\mu_{\mathcal{F}}([v]_\P) \,d\mu^\N(\eta) \\
			&= \int \int \mathds{1}_{\set{\tau<  2\epsilon}}([e_+(\eta_0^* \dots \eta_{n-1}^*)]_{\P},[v]_\P) \,d\mu^\N(\eta) \,d\mu_{\mathcal{F}}([v]_\P) \\
			&= \int \int \mathds{1}_{\set{\tau<  2\epsilon}}([e_+(\eta_{n-1}^* \dots \eta_0^*)]_{\P},[v]_\P) \,d\mu^\N(\eta) \,d\mu_{\mathcal{F}}([v]_\P).
		\end{align*}
		Here $\mu_{\mathcal{F}} = \int \delta_{[e^+(\eta)]_{\P} } \,d\mu^\N\in \Prob(\partial \mathbb H_\mu)$ is the Furstenberg measure for $\mu$ (see Remark \ref{wekk}). For each $\eta = (\eta_0, \eta_1, \dots) \in \Omega_+$, let $\tilde\eta$ be an associated sequence of linear transformations defined by
	\[\tilde  \eta\coloneqq ((\eta_0^*)^t, (\eta_1^*)^t, \dots)\]
	 where we recall from Section \ref{sec:adapted} that $(\eta_i^*)^t$ denotes the adjoint with respect to the adapted inner product $\angle{\cdot, \cdot}_{\mathrm{ad}}$.
	 
	  For $\mu^\N$ almost every $\eta$, Furstenberg convergence for the adjoint cocycle says that 
	 \begin{align*}
	 	e_+(\eta_{n-1}^* \dots \eta_0^*) \to \frac{e^+(\tilde \eta)}{\norm{e^+(\tilde \eta)}}
	 \end{align*}
	 as $n\to \infty$. Hence, for any vector $v$, 
		\[ \limsup_{n\to \infty} \mathds{1}_{\set{\tau<  2\epsilon}}([e_+(\eta_{n-1}^* \dots \eta_0^*)]_{\P}, [v]_\P) \leq \mathds{1}_{\set{\tau\leq  2\epsilon}} ([e^+(\tilde \eta)]_{\P}, [v]_\P)\] almost everywhere. By reverse Fatou, 
		\begin{align*}
			\limsup_{n\to \infty}\int  \mathds{1}_{\set{\tau<  2\epsilon}}([e_+(\eta_{n-1}^* \dots \eta_0^*)]_{\P}, [v]_\P)  \,d\mu^\N(\eta)& \leq  \int  \mathds{1}_{\set{\tau\leq 2\epsilon}}([e^+(\tilde \eta)]_{\P}, [v]_\P) \,d\mu^\N(\eta) \\
			&= \int  \mathds{1}_{\set{\tau\leq  2\epsilon}}([w]_\P, [v]_\P)\,d\tilde \mu_{\mathcal{F}}([w]_\P)
			\end{align*} where $\tilde \mu_{\mathcal{F}}\coloneqq \int \delta_{[e^+(\tilde \eta)]_{\P}}  \,d\mu^\N(\eta)$. Thus 
		\begin{align}\label{positive}
			\limsup_{n\to \infty} \int\mathds{1}_{\set{h_n<  2\epsilon}}(\eta)\,d\mu^\N(\eta)  \leq \int \int \mathds{1}_{\set{\tau\leq 2\epsilon}}([w]_\P,[v]_\P)\, d \tilde \mu_{\mathcal{F}}([w]_\P)\,d \mu_{\mathcal{F}}([v]_\P),
		\end{align} and 
		taking $n\to \infty$ in \eqref{close} and recalling $\lim_{n\to \infty} \mu^\Z(S_\epsilon^n) = 1$ yields  
		\begin{align}\label{almost}
			\delta_{\epsilon} \leq    \int \int \mathds{1}_{\set{\tau\leq 2\epsilon}}([w]_\P,[v]_\P)\, d \tilde \mu_{\mathcal{F}}([w]_\P)\,d \mu_{\mathcal{F}}([v]_\P). 
		\end{align} Since the adjoint action remains strongly irreducible, $\tilde \mu_{\mathcal{F}}$ gives no mass to any proper projective subspace of $\P(\NS(X)_+)$. Hence
			\[\int \mathds{1}_{\set{\tau\leq 2\epsilon}}([w]_\P,[v]_\P)\, d \tilde \mu_{\mathcal{F}}([w]_\P) \to 0\] 
	as $\epsilon \to 0$ for every $v$. Then dominated convergence and \eqref{almost} imply $\lim_{\epsilon \to 0} \delta_\epsilon = 0$. 
	\end{proof}
	It remains to prove Propositions \ref{lem:MassBound2} and \ref{lem:closed}.
\begin{proof}[Proof of Proposition \ref{lem:MassBound2}]
	Fix $\rho>0$. Choose $\omega$ generic in the sense of both Lemma \ref{lem:MassBound} and Claim \ref{claim:mass}. Applying Lemma \ref{lem:MassBound} with $\rho/3$ in place of $\rho$, we obtain a subset $S_1\subset \N$ with upper-density greater than $1-\rho/3$ such that
	\begin{align*}
		\frac{1}{\norm{(f_\omega^n)^*}} M(C_\psi^n)
	\end{align*}
	is uniformly bounded for $n\in S_1$. Applying Claim \ref{claim:mass} with $\rho/3$ in place of $\rho$, we obtain a subset $S_2\subset \N$ with lower-density greater than $1-\rho/3$ such that
	\begin{align*}
		M((f_\omega^n)_*T_\omega^+)\norm{(f_\omega^n)^*}
	\end{align*}
	is uniformly bounded for $n\in S_2$.
	
		Let $S=S_1\cap S_2$. Since $S_1$ has upper-density greater than $1-\rho/3$ and $S_2$ has lower-density greater than $1-\rho/3$, the set $S$ has upper-density greater than $1-\rho$. For $n\in S$, we have
	\begin{align*}
		M((f_\omega^n)_*T_\omega^+)M(C_\psi^n)
		&=
		\left(M((f_\omega^n)_*T_\omega^+)\norm{(f_\omega^n)^*}\right)
		\left(\frac{1}{\norm{(f_\omega^n)^*}}M(C_\psi^n)\right),
	\end{align*}
	which is uniformly bounded. Enumerating $S$ gives the desired subsequence.
	
	If the exponential moment condition \eqref{exponential} is satisfied, then Lemma \ref{lem:MassBound} gives $S_1=\N$. Hence the same argument gives a subsequence of lower-density greater than $1-\rho$.
\end{proof}
	\begin{rmk}\label{rmk:sequence-independent}
	Notice that the sequences constructed in Claim \ref{claim:mass} and Lemma \ref{lem:MassBound} have no dependence on $C$ or $\psi$. This is clearly the case in Claim \ref{claim:mass}, which involved neither $C$ nor $\psi$ in the statement. In Lemma \ref{lem:MassBound}, we have two cases: with the exponential moment assumption and without it. Under the exponential moment assumption that sequence may be taken to be all of $\N$; otherwise, it comes from Corollary \ref{cor:C0Bound} which does not involve $C$ or $\psi$.
	
	Viewing the sequences constructed in Claim \ref{claim:mass} and Lemma \ref{lem:MassBound} as subsets of $\N$, the sequence $(n_i)_{i\in\N}$ in Proposition \ref{lem:MassBound2} is obtained by intersecting the two. Hence $(n_i)_{i\in\N}$ in Proposition \ref{lem:MassBound2} may be chosen independently of $C$ and $\psi$.
\end{rmk}
		\begin{proof}[Proof of Proposition \ref{lem:closed}]
		Let $\psi$ be as before, and let $\psi_0\in C^\infty(X)$ be nonnegative with $\supp \psi_0 \cap \supp(\partial C) = \varnothing$ and $\restr{\psi_0}{\supp \psi}\equiv 1$. Let $\beta$ be a smooth $(0,1)$-form on $X$; the proof is analogous for $(1,0)$-forms. Then 
		\begin{align*}
			&M((f_\omega^{n_i})_* T_\omega^+)|(C_\psi^{n_i} \mid d\beta)| \\ &\hspace{.7 cm}= M((f_\omega^{n_i})_* T_\omega^+)|(C \mid  \psi  d  (f_\omega^{n_i})^* \beta)| \\ 
			&\hspace{.7 cm}= M((f_\omega^{n_i})_* T_\omega^+)|(C \mid  (d \psi) \wedge (f_\omega^{n_i})^* \beta)|  \\
			&\hspace{.7 cm}= M((f_\omega^{n_i})_* T_\omega^+)|(\psi_0 C \mid  (d \psi) \wedge (f_\omega^{n_i})^* \beta)|  \\
			&\hspace{.7 cm}= M((f_\omega^{n_i})_* T_\omega^+)|(\psi_0 C \mid  (\partial \psi) \wedge (f_\omega^{n_i})^* \beta)|  \\
			&\hspace{.7 cm}\leq M((f_\omega^{n_i})_* T_\omega^+)(\psi_0 C \mid  i \partial \psi \wedge \overline \partial \psi)^{\frac 1 2} (\psi_0 C \mid  (f_\omega^{n_i})^* (-i \beta \wedge \overline \beta))^{\frac 1 2}  \\
			&\hspace{.7 cm} = M((f_\omega^{n_i})_* T_\omega^+)^{\frac 1 2}(\psi_0 C \mid  i \partial \psi \wedge \overline \partial \psi)^{\frac 1 2}  M((f_\omega^{n_i})_* T_\omega^+)^{\frac 1 2} (\psi_0 C \mid  (f_\omega^{n_i})^* (-i \beta \wedge \overline \beta))^{\frac 1 2}  
			\end{align*} where in the third line we used $\supp \psi \cap \supp(\partial C) = \varnothing$, in the fourth line we used that $\restr{\psi_0}{\supp \psi}\equiv 1$, and in the fifth line we used positivity of $\psi_0 C$ and Cauchy--Schwarz. Now 
		Choose $B_\beta>0$ such that
		\begin{align*}
			-i\beta\wedge\overline\beta\leq B_\beta\kappa_0
		\end{align*}
		as smooth $(1,1)$-forms. Then
		\begin{align*}
			M((f_\omega^{n_i})_* T_\omega^+)^{\frac 1 2} (\psi_0 C \mid  (f_\omega^{n_i})^* (-i \beta \wedge \overline \beta))^{\frac 1 2} 
			&\leq B_\beta^{\frac12}M((f_\omega^{n_i})_* T_\omega^+)^{\frac 1 2} (\psi_0 C \mid  (f_\omega^{n_i})^*\kappa_0)^{\frac 1 2} \\
			&=B_\beta^{\frac12}M((f_\omega^{n_i})_* T_\omega^+)^{\frac 1 2}(C_{\psi_0}^{n_i}(\omega)\mid\kappa_0)^{\frac 1 2},
		\end{align*}
		which is uniformly bounded by Proposition \ref{lem:MassBound2} and Remark \ref{rmk:sequence-independent}, applied with $\psi_0$ in place of $\psi$. On the other hand, $ M((f_\omega^{n_i})_* T_\omega^+)^{\frac 1 2} \to 0$. Indeed, by Remark \ref{rmk:equiv} and Claim \ref{claim:mass},
	 	$M((f_\omega^{n_i})_* T_\omega^+) \asymp \frac{1}{\norm{(f_\omega^{n_i})^*}} \to 0$. Hence $M((f_\omega^{n_i})_* T_\omega^+)(C_\psi^{n_i} \mid d\beta) \to 0$.
	\end{proof}
\subsection{Proof outline for Proposition \ref{lem:last2}}
Here we give an outline of the proof of Proposition \ref{lem:last2}. Let $\omega\in \Omega$. For all $n\in \Z$, we have the identity
\begin{align}\label{inverse}
	f_\omega^{-n} = (f_{\sigma^{-n}(\omega)}^n )^{-1}
\end{align}
so 
\begin{align*}
	((f_{\sigma^\ell(\omega)}^{-n})^{-1})^* (C_\psi^\ell) &= (f_{\sigma^\ell(\omega)}^{-n})_*(C_\psi^\ell) \\
	&=( f_{\sigma^\ell(\omega)}^{-n})_*(f_\omega^\ell)_*(\psi C) \\
	&= ( f_{\sigma^\ell(\omega)}^{-n})_*(f_{\sigma^{\ell - n}(\omega)}^{n})_* (f_{\omega}^{\ell - n})_*(\psi C) \\
	&= (f_{\omega}^{\ell - n})_*(\psi C)\\
	&= C_\psi^{\ell - n}
\end{align*}
where in the second to last line we used \eqref{inverse} applied to $\sigma^\ell(\omega)$. Hence 
\begin{align}\label{inverse2}
	(f_{\sigma^\ell(\omega)}^{-n})_*(M((f_\omega^\ell)_*T_\omega^+)C_\psi^\ell) 
	 = M((f_\omega^\ell)_*T_\omega^+) C_\psi^{\ell - n}.
\end{align} 

Assuming $\omega$ to be Furstenberg-generic, by Theorem \ref{FurstenbergLimit} and Remark \ref{rmk:FurstBack} combined with standard results in random products of matrices, the contracting direction of $((f_{\sigma^\ell(\omega)}^{-n})^{-1})^*$ in cohomology is well-approximated by $e^-(\sigma^\ell(\omega))$ for most $n$ when $\ell$ is large. Setting $\ell = n$ in \eqref{inverse2} yields
\begin{align*}
	((f_{\sigma^n(\omega)}^{-n})^{-1})^*(M((f_\omega^n)_*T_\omega^+)C_\psi^n) = M((f_\omega^n)_*T_\omega^+) \psi C
\end{align*} which has mass tending to zero as $n\to \infty$, if one excludes a subset of $n$'s with small upper-density (see Remark \ref{rmk:equiv} and Claim \ref{claim:mass}). If each $C_\psi^n$ were closed, then we could pass to cohomology and obtain
\begin{align*}
d_\P\left(M((f_\omega^n)_*T_\omega^+)[C_\psi^n], [T_{\sigma^n(\omega)}^-]\right) \to 0,
\end{align*}
where $d_\P$ denotes projective distance, since $[T_{\sigma^n(\omega)}^-]$ approaches the unique contracting direction of $((f_{\sigma^n(\omega)}^{-n})^{-1})^*=(f_\omega^n)^*$. After passing to a high-density subsequence on which the classes $[T_{\sigma^n(\omega)}^-]$ have limit points only in a compact subset of
cohomology where closed positive currents are unique in their cohomology classes, this would imply
\begin{align*}
	\frac{C_\psi^n}{M(C_\psi^n)}
	-
	T_{\sigma^n(\omega)}^-
	\to 0
\end{align*}
weakly as currents along this subsequence.
Taking $\omega$, $\rho$, and $(n_i)_{i\in\N}$ to be as in Proposition \ref{lem:MassBound2}, uniformly bounded mass of the sequence
\begin{align*}
		(M((f_\omega^{n_i})_*T_\omega^+)C_\psi^{n_i})_{i\in \N}
	\end{align*} implies that $(\tilde n_i)_{i\in \N} \coloneqq (n_i)_{i\in \N}$ satisfies \eqref{important:seq}. Then \eqref{important:currs} will follow from \eqref{important:seq} after restricting to a subsequence on which the collection of normalized potentials $\set{\phi(T^+_{\sigma^{ n_i}(\omega)}) \colon i\in \N}$ is precompact in the $C^0$-topology (see Lemma \ref{lem:curr}).

The challenge then becomes overcoming the non-closedness of the $C_\psi^{n_i}$'s. To do so, we extract a subsequence $(\tilde n_k)_{k\in \N}$ of $(n_i)_{i\in \N}$ with upper-density greater than $1 - 2\rho$, a sequence $(\tilde C_\psi^{\tilde n_k})_{k\in \N}$ of closed positive currents with bounded mass, and an associated sequence of naturals $\tilde m_k \to \infty$ such that 
\begin{align}\label{need1}
	d_\Theta(\tilde C_\psi^{\tilde n_k}, M((f_\omega^{\tilde n_k})_*T_\omega^+)C_\psi^{\tilde n_k}) \to 0
\end{align} and
\begin{align}\label{need2}
	((f_{\sigma^{\tilde n_k}(\omega)}^{-\tilde m_k})^{-1})^*(\tilde C_\psi^{\tilde n_k}) \to 0
		\end{align} weakly (or equivalently in cohomology) as $k\to \infty$. Since $\tilde m_k \to \infty$ and the $\tilde C_\psi^{\tilde n_k}$'s are closed, \eqref{need2} will be enough to show that the classes of $\tilde C_\psi^{\tilde n_k}$ and $T^-_{\sigma^{\tilde n_k}(\omega)}$ become close projectively. After restricting to times for which uniqueness holds in the class of $T^-_{\sigma^{\tilde n_k}(\omega)}$, this gives projective closeness of the currents themselves. Combined with \eqref{need1} this will prove \eqref{important:seq}. 
		Proposition \ref{lem:closed} is what allows us to find the closed approximants in \eqref{need1} after passing to a subsequence.
			
On the other hand,
\begin{equation}\label{eq:1}
	\begin{aligned}
		M(((f_{\sigma^{\tilde n_k}(\omega)}^{-\tilde m_k})^{-1})^*(\tilde C_\psi^{\tilde n_k}))
	&= 	(((f_{\sigma^{\tilde n_k}(\omega)}^{-\tilde m_k})^{-1})^*(\tilde C_\psi^{\tilde n_k})\mid \kappa_0) \\
	&= 	(\tilde C_\psi^{\tilde n_k}\mid \Theta((f_{\sigma^{\tilde n_k}(\omega)}^{-\tilde m_k})^* \kappa_0)) \\
		&= M((f_\omega^{\tilde n_k})_*T_\omega^+) (C_\psi^{\tilde n_k} \mid \Theta((f_{\sigma^{\tilde n_k}(\omega)}^{-\tilde m_k})^* \kappa_0))  \\
	&\hspace{.3 cm}+  (\tilde C_\psi^{\tilde n_k} - 
	M((f_\omega^{\tilde n_k})_*T_\omega^+) C_\psi^{\tilde n_k} \mid  \Theta((f_{\sigma^{\tilde n_k}(\omega)}^{-\tilde m_k})^* \kappa_0))
	\end{aligned}
\end{equation}
where in the second line we used closedness of $\tilde C_\psi^{\tilde n_k}$. So $(\tilde n_k)$ satisfies \eqref{need2} if the terms in the last two lines of \eqref{eq:1} tend to $0$.

Addressing the first term,
\begin{align*}
	(f_{\sigma^{\tilde n_k}(\omega)}^{-\tilde m_k})^* \kappa_0 = \Theta((f_{\sigma^{\tilde n_k}(\omega)}^{-\tilde m_k})^* \kappa_0) + dd^c \phi((f_{\sigma^{\tilde n_k}(\omega)}^{-\tilde m_k})^* \kappa_0)
\end{align*} implies
\begin{equation}\label{eq:2}
	\begin{aligned}
		(C_\psi^{\tilde n_k} \mid \Theta((f_{\sigma^{\tilde n_k}(\omega)}^{-\tilde m_k})^* \kappa_0))  &\leq (C_\psi^{\tilde n_k} \mid (f_{\sigma^{\tilde n_k}(\omega)}^{-\tilde m_k})^* \kappa_0) \\
		& \hspace{1.5 cm }+ \big |(C_\psi^{\tilde n_k} \mid dd^c \phi((f_{\sigma^{\tilde n_k}(\omega)}^{-\tilde m_k})^* \kappa_0))\big |  \\
		&= M(((f_{\sigma^{\tilde n_k}(\omega)}^{-\tilde m_k})^{-1})^*( C_\psi^{\tilde n_k})) \\
		& \hspace{1.5 cm } + \big |(C_\psi^{\tilde n_k} \mid dd^c \phi((f_{\sigma^{\tilde n_k}(\omega)}^{-\tilde m_k})^* \kappa_0))\big |  \\
		&= M(C_\psi^{\tilde n_k - \tilde m_k})  + \big |(C_\psi^{\tilde n_k} \mid dd^c \phi((f_{\sigma^{\tilde n_k}(\omega)}^{-\tilde m_k})^* \kappa_0))\big | 
	\end{aligned}
\end{equation} where the second equality holds by \eqref{inverse2}. Hence we will want to choose sequences $(\tilde n_k), (\tilde m_k)$ such that 
\begin{align}\label{ab}
	M((f_\omega^{\tilde n_k})_*T_\omega^+) M(C_\psi^{\tilde n_k - \tilde m_k}) \to 0
\end{align} and 
\begin{align}\label{abc}
	M((f_\omega^{\tilde n_k})_*T_\omega^+)  \big |(C_\psi^{\tilde n_k} \mid dd^c \phi((f_{\sigma^{\tilde n_k}(\omega)}^{-\tilde m_k})^* \kappa_0))\big | \to 0
\end{align} as $k\to \infty$.
We will construct them such that $\tilde n_k - \tilde m_k \in \set{n_i}_{i\in \N}$ for all $k$; thus, 
\begin{align*}
		M((f_\omega^{\tilde n_k - \tilde m_k})_*T_\omega^+) C_\psi^{\tilde n_k - \tilde m_k}
		\end{align*} has uniformly bounded mass. Choosing the $(\tilde n_k)$ and $(\tilde m_k)$ carefully, we will have 
			\begin{align*}
					\frac{M((f_\omega^{\tilde n_k})_*T_\omega^+)}{M((f_\omega^{\tilde n_k - \tilde m_k})_*T_\omega^+)} \approx \frac{\norm{(f_\omega^{\tilde n_k-\tilde m_k})^*}}{\norm{(f_\omega^{\tilde n_k})^*}} \to 0
				\end{align*} so \eqref{ab} holds. In the upper-density case, this ratio will be controlled using the Gou\"ezel--Karlsson tightness condition for the sequence $(n_i)$. Then we show that
\begin{align*}
			\big |(C_\psi^{\tilde n_k} \mid dd^c \phi((f_{\sigma^{\tilde n_k}(\omega)}^{-\tilde m_k})^* \kappa_0))\big | \lesssim \norm{\phi (( f_{\sigma^{\tilde n_k}(\omega)}^{-\tilde m_k})^*    \kappa_0) }_{C^0}
		\end{align*} which by choice of subsequences will grow roughly like $e^{(\lambda_\mu+\epsilon)\tilde m_k}$. Hence to satisfy \eqref{abc} we want
\begin{align*}
	M((f_\omega^{\tilde n_k})_*T_\omega^+)e^{(\lambda_\mu+\epsilon)\tilde m_k} \to 0
\end{align*} as $k\to \infty$; this is accomplished by picking $\tilde n_k$ much larger than $\tilde m_k$.

Addressing the term in the last line of \eqref{eq:1}, we have
\begin{equation}\label{later}
	\begin{aligned}
				&\big |(\tilde C_\psi^{\tilde n_k} - 
			M((f_\omega^{\tilde n_k})_*T_\omega^+) C_\psi^{\tilde n_k} \mid \Theta((f_{\sigma^{\tilde n_k}(\omega)}^{-\tilde m_k})^* \kappa_0))\big | \\
		&\hspace{4 cm} \leq  
		\norm{(f_{\sigma^{\tilde n_k}(\omega)}^{-\tilde m_k})^* \kappa_0}
		d_\Theta(\tilde C_\psi^{\tilde n_k} ,
		M((f_\omega^{\tilde n_k})_*T_\omega^+) C_\psi^{\tilde n_k}), 
	\end{aligned}
\end{equation}
and $\norm{(f_{\sigma^{\tilde n_k}(\omega)}^{-\tilde m_k})^* \kappa_0}$ grows roughly like $e^{(\lambda_\mu+\epsilon)\tilde m_k}$; hence, we will choose subsequences such that 
		\[d_\Theta(\tilde C_\psi^{\tilde n_k} , M((f_\omega^{\tilde n_k})_*T_\omega^+) C_\psi^{\tilde n_k}) e^{(\lambda_\mu+\epsilon)\tilde m_k}\to 0\] as $k\to \infty$. Again, this is accomplished by picking $\tilde n_k$ much larger than $\tilde m_k$.
		
The preceding discussion describes the upper-density argument. For the lower-density statement under the exponential moment condition \eqref{exponential}, the same convergence mechanism is used, but the choice of the auxiliary integer $m$ must be made more flexibly. The reason is that the lower-density construction no longer uses the Gou\"ezel--Karlsson tightness condition for the sequences from Proposition \ref{lem:MassBound2}; see Remark \ref{rmk:tightness}. Instead, Proposition 5.14 of Cantat--Dujardin \cite{MR4635340} controls the ratios
	\begin{align*}
		\frac{\norm{(f_\omega^{n-m})^*}}{\norm{(f_\omega^n)^*}}
	\end{align*}
	well enough to choose, for most large $n$, an admissible $m$ for which this ratio is small.

\subsection{Proof of Proposition \ref{lem:last2}}
\begin{proof}[Proof of Proposition \ref{lem:last2}]
Let $\omega$ be generic in the sense of Proposition \ref{lem:MassBound2}, and apply it with $\rho'$, to be chosen later, in place of $\rho$ to obtain $(n_i)_{i\in \N}$. The proof follows the strategy from the preceding outline. First, we choose auxiliary sequences $(m_i)$ and $(c_i)$ so that density, growth, approximation, and Birkhoff estimates needed later hold simultaneously. Second, we use these sequences to select a high-density subsequence $(\tilde n_k)$ of $(n_i)$, together with associated integers $(\tilde m_k)$, and closed currents $\tilde C_\psi^{\tilde n_k}$ close to $M((f_\omega^{\tilde n_k})_*T_\omega^+)C_\psi^{\tilde n_k}$. By construction, the classes $[\tilde C_\psi^{\tilde n_k}]$ become small after applying $((f_{\sigma^{\tilde n_k}(\omega)}^{-\tilde m_k})^*)^{-1}$. This forces either $[\tilde C_\psi^{\tilde n_k}]\to 0$, or their projective directions to approach the Furstenberg limit classes $e^-(\sigma^{\tilde n_k}(\omega))$. Finally, we use the closeness of $\tilde C_\psi^{\tilde n_k}$ to $M((f_\omega^{\tilde n_k})_*T_\omega^+)C_\psi^{\tilde n_k}$, together with the various estimates from the construction, to identify the possible limits obtained along common convergent subsequences of $M((f_\omega^{\tilde n_k})_*T_\omega^+)C_\psi^{\tilde n_k}$ and $T^-_{\sigma^{\tilde n_k}(\omega)}$. From there, we prove \eqref{important:seq} and \eqref{important:currs}. At the end we explain how the same construction gives lower density under the exponential moment assumption.
	
	By construction, there exists $R > 0$ such that 
\begin{align}\label{boundd}
	M((f_\omega^{n_i})_*T_\omega^+) M(C_\psi^{n_i}) &\leq R
\end{align} for all $i \in \N$. Let
\begin{align*}
	S \coloneqq \set{n_i}_{i\in \N}.
\end{align*}
By the construction of the sequence in Proposition \ref{lem:MassBound2}, the set $S$ is contained in the good set supplied by Claim \ref{claim:mass}. Hence, by Remark \ref{rmk:equiv},
\begin{align*}
	M((f_\omega^n)_*T_\omega^+)\asymp \frac{1}{\norm{(f_\omega^n)^*}}
\end{align*}
uniformly for $n\in S$. Together with \eqref{boundd}, this implies
\begin{align}\label{bounddd}
	\frac{1}{\norm{(f_\omega^{n_i})^*}} M(C_\psi^{n_i})
	\end{align} is uniformly bounded as well.

		We begin by introducing auxiliary sequences of integers $(m_i)_{i\in \N}$ and $(c_i)_{i\in \N_0}$ which will be used in the construction of $(\tilde m_k)$ and $(\tilde n_k)$. We choose $m_i,c_i$ in pairs $(m_i,c_i)$ with $m_i\to \infty$ and $c_i\to\infty$. %
	
			The integers $(m_i)$ will be chosen as a subsequence of the sequence $(\hat m_j)$ produced by the following claim.
	\begin{claim}\label{claim:1}
	There exists $\hat m_j\to \infty$ such that 
\begin{align*}
\overline d(S\cap (S +\hat m_j)) \geq  \overline{d}(S)^2-\frac1j.
\end{align*} 
\end{claim}
\begin{proof}[Proof of Claim \ref{claim:1}]
	Set $\alpha \coloneqq \overline d(S)$, and let $\xi \in \set{0,1}^\Z$ be the indicator sequence of $S$, so that $\xi_n = 1$ exactly when $n \in S$. Notice that $\xi_k = 0$ whenever $k\leq 0$. Let
\begin{align*}
	C \coloneqq \set{\eta \in \set{0,1}^\Z : \eta_1 = 1}.
\end{align*}

Since $S$ has upper-density $\alpha$, we may choose $N_i \to \infty$ such that
\begin{align*}
	\frac{1}{N_i} \#(S \cap \set{1,\dots,N_i}) \to \alpha.
\end{align*}
For all $i\in\N$, define
\begin{align*}
			P_i \coloneqq \frac{1}{N_i}\sum_{n=0}^{N_i-1}\delta_{\sigma^{n}(\xi)}.
\end{align*}
Each $P_i$ is a probability measure, so passing to a subsequence we may assume $P_i \to P$ weakly for some $\sigma$-invariant probability measure $P$. By construction,
\begin{align*}
	P(C)=\lim_{i\to\infty} P_i(C)=\alpha.
\end{align*}
Moreover, for each $k \geq 1$,
\begin{align*}
	P(C\cap \sigma^{k}C) =
	\lim_{i\to\infty}\frac{1}{N_i} \# (\set{1,\dots,N_i}\cap S\cap (S+k)) \leq \overline d\bigl(S\cap (S+k)).
\end{align*}

\noindent Now let $f=\mathds{1}_C$. By the mean ergodic theorem,
\begin{align*}
	\frac{1}{N}\sum_{k=1}^N f\circ \sigma^k \to f^*
\end{align*}
in $L^2(\set{0,1}^\Z,P)$, where $f^*$ is the orthogonal projection in $L^2(\set{0,1}^\Z,P)$ of $f$ onto the subspace of $\sigma$-invariant functions. Hence
\begin{align*}
	\frac{1}{N}\sum_{k=1}^N P(C\cap \sigma^{k}C)
	= \int f \left( \frac{1}{N}\sum_{k=1}^N f\circ \sigma^{-k} \right) \,d P \to
	\int f\cdot f^*\,d P =
	\int (f^*)^2\,d P.
\end{align*}
Now
\begin{align*}
	\int f^*\,dP=\int f\,dP=P(C)=\alpha
\end{align*}
and by Cauchy--Schwarz,
\begin{align*}
	\int (f^*)^2\,dP \geq \left(\int f^*\,dP\right)^2 = \alpha^2.
\end{align*}
It follows that
\begin{align*}
	\liminf_{N\to\infty}\frac{1}{N}\sum_{k=1}^N \overline d\bigl(S\cap (S+k)\bigr)\geq \alpha^2.
\end{align*}
In particular, for each $j$ there are arbitrarily large $m$ such that
\begin{align*}
	\overline d\bigl(S\cap (S+m)\bigr)\geq \alpha^2-\frac1j.
\end{align*}
Choosing such $m$ recursively gives the desired sequence $(\hat m_j)$.
	\end{proof}

				 Fix $0<\epsilon<\lambda_\mu$. Next we define the pairs $(m_i,c_i)$ recursively, depending on this choice of $\epsilon$, so that several estimates hold simultaneously.
			
		\begin{claim}\label{claim:2}
			There exists a subsequence of $(\hat m_i)$ from Claim \ref{claim:1}, denoted by $(m_i)$, and an increasing sequence $(c_i)_{i\in\N_0}$ of natural numbers such that the pairs $(m_i,c_i)$ have the following properties:
		\begin{enumerate}
			\item Density satisfies
				\begin{align*} 
					\liminf_{i\to\infty}
					\frac{1}{c_{i+1}}
					\bigl|
					S \cap (S+m_i) \cap [1,c_{i+1}]
				\bigr|
				\geq \overline d(S)^2.
			\end{align*}  \label{claim2:density}
		\item We have $c_i \to \infty$ fast enough that $c_i > m_i$ for all $i$,
		\begin{align*}
		 c_i - m_i\to \infty\hspace{.75 cm }\text{and}  \hspace{.75 cm }\frac{c_i}{c_{i+1}} \to 0.
		\end{align*} \label{claim2:separation}
		\item For every $i\in \N$ and $n\geq  c_i$,
		\begin{align*}
			\norm{(f_\omega^n)^*} > i e^{(\lambda_\mu + \epsilon)m_i}.
		\end{align*} \label{claim2:growth}
	\item For every $i\in \N$ and $n\in S$ with $n\geq c_i$, there exists a closed positive current $T$ such that
	\begin{align*}
		d_\Theta(M((f_\omega^{n})_*T_\omega^+) C_\psi^{n}, T) < \frac 1 i e^{-(\lambda_\mu+\epsilon)m_i},
	\end{align*} recalling the definition of $d_\Theta$ from Section \ref{sec:norms}. \label{claim2:closed-approx}
	\item For every $i\in \N$ and $N\geq c_i$,
\begin{align*}
		\frac 1 {N} \sum_{k = 1}^{N} \mathds{1}_{\set{\eta \colon \left| \frac 1 {m_i} \ln \norm{(f_{\eta}^{-m_i})^*}  - \lambda_\mu \right| < \epsilon}}(\sigma^k(\omega)) > 1 - \frac 1 i 
\end{align*}  
 \label{claim2:lyap-average}
	\item Let 
	\begin{align*}
		G_N(\omega)\coloneqq \sum_{j=1}^{N}  \norm{(f_{\sigma^{-j}(\omega)}^{-N + j})^*}\norm{\omega _{-j}^{-1}}_{C^1}^2.
	\end{align*} For every $i\in \N$ and $N\geq c_i$,
	\begin{align*}
		\frac{1} N	\sum_{k=1}^{N} \mathds{1}_{\set{\eta \colon G_{m_i}(\eta) <  e^{(\lambda_\mu + \epsilon)m_i}}}(\sigma^{k}(\omega)) > 1 - \frac 1 i
	\end{align*}
	\label{claim2:G-average}
	\item
		Let 
		\begin{align*}
			H_N(\eta)\coloneqq d_\P\left(e_+(((f^{-N}_{\eta})^*)^t), e^-(\eta)\right).
		\end{align*} For every $i\in \N$ and $N\geq c_i$,
	\begin{align*}
		\frac 1 N \sum_{k = 1}^N\mathds{1}_{\set{\eta \colon H_{m_i}(\eta )< \frac 1 i }} (\sigma^k(\omega)) > 1 - \frac 1 i
	\end{align*} \label{claim2:H-average}
	\end{enumerate}
\end{claim}

		\begin{proof}[Proof of Claim \ref{claim:2}]
			Let $(\hat m_j)_{j\in\N}$ be the sequence obtained in Claim \ref{claim:1}. We will choose a subsequence of $(\hat m_j)$, denoted $(m_i)$, and its companion sequence $(c_i)$ recursively. For $m,i\in\N$, set
			\begin{align*}
				&\mathcal L_m
				\coloneqq
				\set{\eta \colon
				\left|\frac 1 m \ln \norm{(f_\eta^{-m})^*}-\lambda_\mu\right|<\epsilon}, \hspace{.25 cm }
				\mathcal R_m \coloneqq
				\set{\eta \colon G_m(\eta)<e^{(\lambda_\mu+\epsilon)m}},\\
				 & \hspace{4 cm } \mathcal H_{m,i}
				\coloneqq
				\set{\eta \colon H_m(\eta)<\frac 1 i}.
			\end{align*}
			We first show that
			\begin{align}\label{claim2:good-measures}
				\mu^\Z(\mathcal L_m)\to 1,\qquad
				\mu^\Z(\mathcal R_m)\to 1,\qquad
				\mu^\Z(\mathcal H_{m,i})\to 1
			\end{align}
			as $m\to\infty$, where the last convergence is for each fixed $i$.
			The first convergence follows from
			\begin{align*}
				\lim_{m\to\infty}
				\frac 1 m \ln \norm{(f_\eta^{-m})^*}
				= \lambda_\mu
			\end{align*}
			for $\mu^\Z$-almost every $\eta$.
			
			For the second convergence, fix $\delta>0$ and $r>1$. Since
			\begin{align*}
				\lim_{n\to\infty}\frac 1 n \ln\norm{(f_\eta^n)^*}
				= \lambda_\mu
			\end{align*}
			for $\mu^\Z$-almost every $\eta$, we have
			\[ \norm{(f_\eta^n)^*}\leq e^{(\lambda_\mu+\delta)n}	\]
			for all sufficiently large $n$. Moreover, by \eqref{int} and Borel--Cantelli, for $\mu^\Z$-almost every $\eta$ we have $\norm{\eta_{-j}^{-1}}_{C^1}<r^j$
			for all sufficiently large $j$. Hence, for $\mu^\Z$-almost every $\eta$,
			\begin{align*}
				\limsup_{N\to\infty}\frac 1 N\ln G_N(\eta)
				&\leq
				\lim_{N\to\infty}
				\frac 1 N
				\ln\sum_{j=0}^{N-1}e^{(\lambda_\mu+\delta)j}r^{2j}\\
				&\leq \lambda_\mu+\delta+2\ln r.
			\end{align*}
			Letting $\delta\to0$ and $r\to1$ gives
			\begin{align*}
				\limsup_{N\to\infty}\frac 1 N\ln G_N(\eta)\leq \lambda_\mu
			\end{align*}
			almost surely, and therefore $\mu^\Z(\mathcal R_m)\to1$.
			Finally, $H_m(\eta)\to0$ for $\mu^\Z$-almost every $\eta$ by Proposition III.3.2 in \cite{MR886674}, applied to the measure $\tilde \mu\in\Prob(\Orth^+(1,n))$ defined by
			\[
				\tilde \mu(f^*)\coloneqq \mu(((f^*)^{-1})^t).
			\]
			This proves the convergence of $\mu^\Z(\mathcal H_{m,i})$ for each fixed $i$.
			
			Since the sets $\mathcal L_m$, $\mathcal R_m$, and $\mathcal H_{m,i}$ form a countable family, we may also assume that $\omega$ is Birkhoff-generic for each of them. Thus, for each such set $\mathcal E\in \set{\mathcal L_m, \mathcal R_m, \mathcal H_{m,i}}$,
			\begin{align}\label{claim2:birkhoff-generic}
				\frac 1 N\sum_{\ell=1}^{N}
				\mathds 1_{\mathcal E}(\sigma^\ell(\omega))
				\to \mu^\Z(\mathcal E)
			\end{align}
			as $N\to\infty$.
			
			We also record the two facts that will be useful for choosing $c_i$. First,
			\begin{align}\label{limitt}
				\lim_{N\to\infty}\norm{(f_\omega^N)^*}=\infty;
			\end{align} this follows from choosing $\omega$ to be Furstenberg-generic.
			Second, let $\mathcal Z$ denote the set of closed positive currents. We claim that
			\begin{align}\label{claim2:closed-distance}
				d_\Theta\bigl(M((f_\omega^{n_i})_*T_\omega^+)C_\psi^{n_i},\mathcal Z\bigr)
				\to 0
			\end{align} as $i\to \infty$.
			Indeed, the sequence
			\[
				\set{M((f_\omega^{n_i})_*T_\omega^+)C_\psi^{n_i}\colon i\in\N}
			\]
			is precompact, and each of its limit points is closed by Proposition \ref{lem:closed}. %
			Therefore every convergent subsequence has $d_\Theta$-distance to $\mathcal Z$ tending to zero, which proves \eqref{claim2:closed-distance}.
			
			We now construct the pairs $(m_i, c_i)$. Set $c_0=0$. Suppose $m_1,\dots,m_{i-1}$ and $c_0,\dots,c_{i-1}$ have already been chosen. By \eqref{claim2:good-measures}, we may choose $m_i=\hat m_{j_i}$ sufficiently far into the sequence $(\hat m_j)$, with $j_i\geq 2i$ and $m_i>m_{i-1}$ if $i\geq2$, so that
			\begin{align*}
				\mu^\Z(\mathcal L_{m_i})>1-\frac{1}{2i},\qquad
				\mu^\Z(\mathcal R_{m_i})>1-\frac{1}{2i},\qquad
				\mu^\Z(\mathcal H_{m_i,i})>1-\frac{1}{2i}.
			\end{align*}
			Then choose $c_i$ so large that the following all hold:
			\begin{enumerate}[label=(\alph*)]
				\item if $i\geq2$, then
				\begin{align*}
					\frac{1}{c_i}
					\bigl|S\cap(S+m_{i-1})\cap[1,c_i]\bigr|
					> \overline d(S)^2-\frac{1}{i-1};
				\end{align*}
				\item $\frac{c_{i-1}}{c_i}<\frac 1 i$ and $c_i - m_i > i$;
				\item (\ref{claim2:growth}) -- (\ref{claim2:H-average}) hold for the index $i$.
			\end{enumerate}
			This is possible as follows. When $i\geq2$, item (a) can be met for arbitrarily large values of $c_i$ by Claim \ref{claim:1} and the choice $j_{i-1}\geq2(i-1)$. Item (b) only requires taking $c_i$ large. Then (\ref{claim2:growth}) and (\ref{claim2:closed-approx}) in item (c) follow, after increasing $c_i$ if necessary, from \eqref{limitt} and from \eqref{claim2:closed-distance}. Finally  (\ref{claim2:lyap-average}), (\ref{claim2:G-average}), and (\ref{claim2:H-average}) in item (c) follow,  after increasing $c_i$ if necessary, from \eqref{claim2:birkhoff-generic} and the choice of $m_i$.
			
			The recursive construction gives (\ref{claim2:growth})--(\ref{claim2:H-average}) directly. The inequalities $\frac{c_{i-1}}{c_i}<\frac 1 i$ and $c_i - m_i > i$ imply (\ref{claim2:separation}). Finally, applying (a) at stage $i+1$ gives
			\begin{align*}
				\frac{1}{c_{i+1}}
				\bigl|S\cap(S+m_i)\cap[1,c_{i+1}]\bigr|
				> \overline d(S)^2-\frac 1 i,
			\end{align*}
			and taking the $\liminf$ proves (\ref{claim2:density}).
		\end{proof}
			Next, we use the pairs $(m_i,c_i)$ corresponding to the fixed $0<\epsilon<\lambda_\mu$ to choose the final subsequence. Let $E_i$ denote the set of $k\in S\cap (S + m_i)$ where the three conditions whose Birkhoff averages appear in \eqref{claim2:lyap-average}--\eqref{claim2:H-average} hold:
\begin{center}
	\begin{enumerate*}[label=(\roman*), itemjoin=\ \ \ \ \ \  ]
	\item $\sigma^k(\omega)\in \mathcal L_{m_i}$, \label{good:lyap} 
	\item $\sigma^k(\omega)\in \mathcal R_{m_i}$,  \label{good:G} 
	\vspace{.2 cm}
	\item  $\sigma^k(\omega)\in \mathcal H_{m_i, i}$ \label{good:H}
\end{enumerate*}
\end{center}
	Define 
	\begin{align*}
		A\coloneqq \bigcup_{i\geq 1} (E_i\cap (c_i, c_{i+1}]).
	\end{align*}
	Now
	\begin{align*}
		\overline d(A) &\geq \limsup_{i\to \infty }\frac{1}{c_{i+1}} \# (A \cap [1, c_{i+1}])\\
		&\geq \limsup_{i\to \infty }\frac{1}{c_{i+1}} \#(E_i \cap (c_i, c_{i+1}])\\
		&\geq  \limsup_{i\to \infty } \left(\frac{1}{c_{i+1}}\# \left(S \cap (S + m_i) \cap [1, c_{i+1}]\right)- \frac{c_i}{c_{i+1}}- \frac 3 i \right)  \\
	& \geq \overline{d}(S)^2
\end{align*} where the second to last line uses  (\ref{claim2:lyap-average}), (\ref{claim2:G-average}), and (\ref{claim2:H-average}), and the last line uses (\ref{claim2:density}) and (\ref{claim2:separation}). Hence
	 \begin{align*}
		\overline d (A) \geq \overline{d}(S)^2 > \left (1 - \rho' \right)^2 > 1 - \rho.
	\end{align*} once we choose $\rho'$ sufficiently small.

	We now impose additional equicontinuity, uniqueness-in-class, and transversality requirements. Each corresponds to a large set, defined in the following paragraphs, for which we assume $\omega$ is Birkhoff-generic. We then keep only the times at which the orbit of $\omega$ belongs to these sets. These restrictions lose only a small amount of density. 
	
	Fix $\rho_0 > 0$, to be chosen sufficiently small. By Lemma \ref{lem:equicontinuous}, we may pick a compact subset $K\subset \Omega$ with measure $\mu^\Z(K)>1 - \rho_0$ such that
\begin{align*}
	\set{\phi(T_\eta^+) \colon \eta \in K}
\end{align*} 
is an equicontinuous family admitting a uniform $C^0$-bound. %

Let \(\Lambda^- \subset \Omega\) denote the full-measure set on which the current
\(T^-_\eta\) is defined and is the unique closed positive current of unit mass in the class \(e^-(\eta)\). Letting $\rho_0>0$ be as above, by regularity of $\mu^\Z$ and Lusin's theorem we may find a compact subset $K'\subset \Lambda^-$ on which $\eta \mapsto e^-(\eta)$ is continuous and $\mu^\Z(K')>1-\rho_0$.

Echoing Remark \ref{tbd}, since $\mu_{\mathcal{F}}$ is non-atomic and $e^+(\eta)$ is independent of $e^-(\eta)$ we have 
\begin{align*}
	D_\tau \coloneqq \set{\eta \colon e^+(\eta)\wedge e^-(\eta) > \tau}
\end{align*} satisfies $\mu^\Z(D_\tau) \to 1$ as $\tau \to 0$.

Choose $\tau>0$ sufficiently small. We now further require $\omega$ to be Birkhoff-generic for $K$, $K'$, and $D_\tau$. %
 With $\rho_0$ and $\tau$ chosen sufficiently small, 
\begin{align*}
	\set{\tilde n_k}_{k\in \N} \coloneqq A\cap \set{k\in\N\colon \sigma^k(\omega)\in K \cap K' \cap D_\tau} 
	\end{align*} satisfies $\overline d(\set{\tilde n_k}_{k\in \N}) > 1- \rho$. We will prove this subsequence satisfies \eqref{important:seq} and \eqref{important:currs}.

			For each $\tilde n_k$, we now choose the associated integer $\tilde m_k$ and closed approximant described in the outline. Let $i_k$ be the unique index such that $c_{i_k} < \tilde n_k \leq c_{i_k+1}$, and set $\tilde m_k\coloneqq m_{i_k}$. Since $\tilde n_k\in A$, we have $\tilde n_k\in E_{i_k}\cap (S+\tilde m_k)$; in particular, $\tilde n_k-\tilde m_k\in S$.

 Since $\set{\tilde n_k}_{k\in \N}\subset S$, by (\ref{claim2:closed-approx}) there exists $\tilde C^{\tilde n_k}_\psi\in \mathcal{Z}$ such that
\begin{align}\label{ugh}
	d_\Theta(M((f_\omega^{\tilde n_k})_*T_\omega^+) C_\psi^{\tilde n_k} , \tilde C_\psi^{\tilde n_k}) < \frac{1}{i_k} e^{-\tilde m_k(\lambda_\mu + \epsilon )}.
\end{align} Note that the currents $\set{\tilde C_\psi^{\tilde n_k}}_{k\in\N}$ have uniformly bounded mass since the currents $\tilde C_\psi^{\tilde n_k}$ are $d_\Theta$-close to currents with uniformly bounded mass; see Remark \ref{rmk:dtheta-mass}.

	The key estimate is that these closed approximants become negligible after applying $((f_{\sigma^{\tilde n_k}(\omega)}^{-\tilde m_k})^*)^{-1}$.
	
	\begin{claim}\label{claim:shrink} 
	We have
\begin{align*}
	\norm{ ((f_{\sigma^{\tilde n_k}(\omega)}^{-\tilde m_k})^*)^{-1} [\tilde C_\psi^{\tilde n_k}]} \to 0
\end{align*} as $k\to \infty$. 
\end{claim}
\begin{proof}[Proof of Claim \ref{claim:shrink}]
	
Recalling \eqref{eq:1} and \eqref{later},
\begin{align*}
	 \norm{ ((f_{\sigma^{\tilde n_k}(
			\omega)}^{-\tilde m_k})^*)^{-1} [\tilde C_\psi^{\tilde n_k}]} 
	&\asymp M(((f_{\sigma^{\tilde n_k}(
		\omega)}^{-\tilde m_k})^*)^{-1} \tilde C_\psi^{\tilde n_k})  \\
	& \hspace{0 cm} = M((f_\omega^{\tilde n_k})_*T_\omega^+) (C_\psi^{\tilde n_k} \mid \Theta((f_{\sigma^{\tilde n_k}(\omega)}^{-\tilde m_k})^* \kappa_0))  \\
	&\hspace{1  cm}+  (\tilde C_\psi^{\tilde n_k} - 
	M((f_\omega^{\tilde n_k})_*T_\omega^+) C_\psi^{\tilde n_k} \mid  \Theta((f_{\sigma^{\tilde n_k}(\omega)}^{-\tilde m_k})^* \kappa_0))\\
	& \hspace{0 cm} \leq M((f_\omega^{\tilde n_k})_*T_\omega^+) \big |(C_\psi^{\tilde n_k} \mid \Theta((f_{\sigma^{\tilde n_k}(\omega)}^{-\tilde m_k})^* \kappa_0))\big |  \\
	&  \hspace{1  cm} + \norm{(f_{\sigma^{\tilde n_k}(\omega)}^{-\tilde m_k})^* [\kappa_0]}
	d_\Theta(\tilde C_\psi^{\tilde n_k} ,
	M((f_\omega^{\tilde n_k})_*T_\omega^+) C_\psi^{\tilde n_k})\\
	& \hspace{0 cm}\leq  \big |(M((f_\omega^{\tilde n_k})_*T_\omega^+) C_\psi^{\tilde n_k} \mid  \Theta( (f_{\sigma^{\tilde n_k}(
		\omega)}^{-\tilde m_k})^*  \kappa_0))\big |  \\
	&\hspace{1 cm }+ \frac{1}{i_k }e^{-(\lambda_\mu+\epsilon)\tilde m_k}\norm{(f_{\sigma^{\tilde n_k}(
			\omega)}^{-\tilde m_k})^*  [\kappa_0]} \\
	& \hspace{0 cm}\leq M((f_\omega^{\tilde n_k})_*T_\omega^+) \big |(C_\psi^{\tilde n_k} \mid\Theta( (f_{\sigma^{\tilde n_k}(
	\omega)}^{-\tilde m_k})^*   \kappa_0))\big |  +  \frac{\norm{[\kappa_0]}}{i_k}
\end{align*}
	where in the second inequality we used \eqref{ugh}, and in the last line we used \ref{good:lyap}. It remains to show that 
\begin{align*}
	 M((f_\omega^{\tilde n_k})_*T_\omega^+) \big |(C_\psi^{\tilde n_k}\mid 	\Theta((f_{\sigma^{\tilde n_k}(
		\omega)}^{-\tilde m_k})^*   \kappa_0))\big | \to 0
\end{align*} as $k\to \infty$. 
Now 
\begin{align*}
& M((f_\omega^{\tilde n_k})_*T_\omega^+) \big |(C_\psi^{\tilde n_k}\mid 	\Theta((f_{\sigma^{\tilde n_k}(
\omega)}^{-\tilde m_k})^*   \kappa_0))\big | \\ &\hspace{.2 cm} \leq  (M((f_\omega^{\tilde n_k})_*T_\omega^+) C_\psi^{\tilde n_k}\mid  (f_{\sigma^{\tilde n_k}(
\omega)}^{-\tilde m_k})^*  \kappa_0) + |(M((f_\omega^{\tilde n_k})_*T_\omega^+) C_\psi^{\tilde n_k}\mid dd^c \phi ((f_{\sigma^{\tilde n_k}(
\omega)}^{-\tilde m_k})^* \kappa_0))|  \\
& \hspace{.2 cm} \lesssim  \frac{1}{\norm{(f_\omega^{\tilde n_k})^*}} (C^{\tilde n_k}_\psi\mid  ( f_{\sigma^{\tilde n_k}(\omega)}^{-\tilde m_k})^* \kappa_0) +  \frac{1}{\norm{(f_\omega^{\tilde n_k})^*}}|  (C_\psi^{\tilde n_k}\mid dd^c \phi ((f_{\sigma^{\tilde n_k}(
		\omega)}^{-\tilde m_k})^*   \kappa_0))|
\end{align*}
where in the last line we used $M((f_\omega^{\tilde n_k})_*T_\omega^+)\asymp \frac{1}{\norm{(f_\omega^{\tilde n_k})^*}}$. 

Echoing \eqref{eq:2}, we have 
\begin{equation}\label{two}
	\begin{aligned}
		& |(M((f_\omega^{\tilde n_k})_*T_\omega^+) C_\psi^{\tilde n_k}\mid  \Theta( ( f_{\sigma^{\tilde n_k}(\omega)}^{-\tilde m_k})^{*}   \kappa_0))|  \\ & \hspace{1 cm} \lesssim  \frac{1}{\norm{(f_\omega^{\tilde n_k})^*}} M(C_\psi^{\tilde n_k - \tilde m_k}) + \frac{1}{\norm{(f_\omega^{\tilde n_k})^*}}|(C_\psi^{\tilde n_k}\mid dd^c \phi (( f_{\sigma^{\tilde n_k}(\omega)}^{-\tilde m_k})^* \kappa_0))|. \\
	\end{aligned}
\end{equation}
We handle the two terms in \eqref{two} separately.

Firstly,
\begin{align*}
	\frac{1}{\norm{(f_\omega^{\tilde n_k})^*}}M(C_\psi^{\tilde n_k - \tilde m_k}) & = 	\frac{\norm{(f_\omega^{\tilde n_k - \tilde m_k})^*}}{\norm{(f_\omega^{\tilde n_k})^*}}  \frac{1}{\norm{(f_\omega^{\tilde n_k - \tilde m_k})^*}}M(C^{\tilde n_k - \tilde m_k}_\psi)\lesssim   \frac{\norm{(f_\omega^{\tilde n_k - \tilde m_k})^*}}{\norm{(f_\omega^{\tilde n_k})^*}} 
\end{align*} using $\tilde n_k - \tilde m_k\in S$ and \eqref{bounddd}. Recalling Remark \ref{rmk:tightness}, in the upper-density case the sequence $(n_i)_{i\in\N}$ was chosen to satisfy the Gou\"ezel--Karlsson tightness conditions. Since $(\tilde n_k)_{k\in\N}$ is a subsequence of $(n_i)_{i\in\N}$,
it satisfies the same tightness conditions. Hence 
\begin{align*}
\frac{\norm{(f_\omega^{\tilde n_k - \tilde m_k})^*}}{\norm{(f_\omega^{\tilde n_k})^*}}   
	&\leq   e^{-\tilde m_k(\lambda_\mu - \delta_{\tilde m_k})} \to 0
\end{align*} as $k\to \infty$, where we used that \[\norm{(f_\omega^{\tilde n_k})^*} \geq e^{\tilde m_k(\lambda_\mu - \delta_{\tilde m_k})}\norm{(f_{\omega}^{\tilde n_k - \tilde m_k})^*}\] from Theorem \ref{thm:GoodTimes} part (ii).

Secondly, 
\begin{align*}
\frac{1}{\norm{(f_\omega^{\tilde n_k})^*}}|(C^{\tilde n_k}_\psi\mid dd^c \phi (( f_{\sigma^{\tilde n_k}(\omega)}^{-\tilde m_k})^*   \kappa_0))| & = \frac{1}{\norm{(f_\omega^{\tilde n_k})^*}}|(\psi C\mid (f_\omega^{\tilde n_k})^*dd^c \phi (( f_{\sigma^{\tilde n_k}(\omega)}^{-\tilde m_k})^*   \kappa_0))|  \\
& = \frac{1}{\norm{(f_\omega^{\tilde n_k})^*}}|(\psi C\mid dd^c (\phi (( f_{\sigma^{\tilde n_k}(\omega)}^{-\tilde m_k})^*   \kappa_0) \circ f_\omega^{\tilde n_k}))|  \\
& = \frac{1}{\norm{(f_\omega^{\tilde n_k})^*}}|(C\mid (\phi (( f_{\sigma^{\tilde n_k}(\omega)}^{-\tilde m_k})^*   \kappa_0) \circ f_\omega^{\tilde n_k}) dd^c \psi)| \\
& \leq \frac{1}{\norm{(f_\omega^{\tilde n_k})^*}}A_\psi  (C\mid \kappa_0) \norm{\phi (( f_{\sigma^{\tilde n_k}(\omega)}^{-\tilde m_k})^*    \kappa_0) }_{C^0}
\end{align*} where in the third line we use that $\supp (\psi ) \cap \supp( \partial C ) = \varnothing$ and in the last line we chose $A_\psi > 0$ such that $-A_\psi \kappa_0 \leq dd^c \psi \leq A_\psi \kappa_0$.

We have seen in Remark \ref{remark:later} that 
\begin{align*}
\norm{\phi((f_{\sigma^{\tilde n_k}(\omega)}^{-\tilde m_k})^*\kappa_0)}_{C^0}
&= \norm{\phi((\omega_{\tilde n_k - \tilde m_k}^{-1}\circ \dots \circ \omega_{\tilde n_k-1}^{-1})^*\kappa_0)}_{C^0}  \\
&\lesssim \sum_{i=1}^{\tilde m_k}  \norm{(f_{\sigma^{\tilde n_k - i}(\omega)}^{-\tilde m_k + i})^* [\kappa_0] } \norm{\sigma^{\tilde n_k}(\omega)_{-i}^{-1}}_{C^1}^2 \\
&\leq  \sum_{i=1}^{\tilde m_k}  \norm{(f_{\sigma^{\tilde n_k - i}(\omega)}^{-\tilde m_k + i})^*} \norm{\sigma^{\tilde n_k}(\omega)_{-i}^{-1}}_{C^1}^2 \\
&\leq e^{(\lambda_\mu + \epsilon)\tilde m_k}
	\end{align*} where the last line follows from \ref{good:G}. Moreover, by (\ref{claim2:growth}),
	\begin{align*}
		\norm{(f_\omega^{\tilde n_k})^*} \geq i_k e^{(\lambda_\mu + \epsilon)\tilde m_k}.
	\end{align*} Hence
\begin{align*}
\frac{1}{\norm{(f_\omega^{\tilde n_k})^*}}|(C_\psi^{\tilde n_k}\mid dd^c \phi ((f_{\sigma^{\tilde n_k}(\omega)}^{-\tilde m_k})^*\kappa_0))| & \lesssim A_\psi(C\mid\kappa_0)\frac 1 {i_k}.
	\end{align*} We have shown that the two terms in \eqref{two} tend to $0$ as $k\to \infty$, concluding the claim.
\end{proof} 
	In the next claim we prove that the cohomological shrinking from Claim \ref{claim:shrink} allows us to relate the sequence $[\tilde C_\psi^{\tilde n_k}]$ with backward Furstenberg directions.
	\begin{claim}\label{claim:addedlater}
		Let $(k_j)$ be a strictly increasing sequence of naturals such that $([\tilde C^{\tilde n_{k_j}}_\psi])_{j\in \N}$ converges to some class $\alpha\in H^{1,1}$. Then either $\alpha = 0$ or  \[d_\P([\alpha]_\P, [e^-(\sigma^{\tilde n_{k_j}}(\omega))]_\P) \to 0\] as $j\to \infty$, recalling that $d_\P$ denotes distance in $\P H^{1,1}$.
	\end{claim}
\begin{proof}[Proof of Claim \ref{claim:addedlater}]
For each $k\in \N$, let
\[
e_-(k)\coloneqq e_-\bigl(((f_{\sigma^{\tilde n_k}(\omega)}^{-\tilde m_k})^*)^{-1}\bigr)
	\] which we recall denotes the contracting singular unit vector in the boundary of the positive cone of $((f_{\sigma^{\tilde n_k}(\omega)}^{-\tilde m_k})^*)^{-1}$ when it exists. Here
	\begin{align*}
		(f_{\sigma^{\tilde n_k-\tilde m_k}(\omega)}^{\tilde m_k})^*
		=
		\left((f_{\sigma^{\tilde n_k}(\omega)}^{-\tilde m_k})^*\right)^{-1},
	\end{align*}
	and, for our adapted Euclidean norm, $\norm{A^{-1}}=\norm{A}$ for $A$ preserving the intersection form. Then for each $k\in \N$ we may find $a_k, b_k\in \R$, and unit vectors  $w_k\in e_-(k)^\perp$ such that
 \[[\tilde C^{\tilde n_k}_\psi]= a_k e_-(k)+  b_{k} w_k.\] By adaptedness of the Euclidean norm and Euclidean singular value decomposition for $((f_{\sigma^{\tilde n_k}(\omega)}^{-\tilde m_k})^*)^{-1}$, 
 \[\norm{((f_{\sigma^{\tilde n_k}(\omega)}^{-\tilde m_k})^*)^{-1} w_k} \geq \norm{w_k} = 1  \hspace{.5 cm}\text{and}\hspace{.5 cm} \norm{((f_{\sigma^{\tilde n_k}(\omega)}^{-\tilde m_k})^*)^{-1} e_-(k)} = \frac{1}{\norm{(f_{\sigma^{\tilde n_k}(\omega)}^{-\tilde m_k})^*}}  \] 
 so 
 \begin{align*}
 	\norm{((f_{\sigma^{\tilde n_k}(\omega)}^{-\tilde m_k})^*)^{-1} [\tilde C^{\tilde n_k}_\psi]}
 	&\geq - |a_k| \frac{1}{\norm{(f_{\sigma^{\tilde n_k}(\omega)}^{-\tilde m_k})^*}} + |b_k|.
 \end{align*}
	 As $\norm{[\tilde C^{\tilde n_k}_\psi]}^2 = a_k^2 + b_k^2$ is uniformly bounded in $k$, both $|a_k|$ and $|b_k|$ are uniformly bounded. Since $\tilde m_k\to\infty$, condition \ref{good:lyap} gives
	 \begin{align*}
	 	\norm{(f_{\sigma^{\tilde n_k}(\omega)}^{-\tilde m_k})^*}
	 	\geq e^{(\lambda_\mu-\epsilon)\tilde m_k}\to\infty,
	 \end{align*}
 and hence Claim \ref{claim:shrink} yields $b_k \to 0$. 
 
 Now let $(k_j)$ and $\alpha$ be as in the claim. By the above computation, if $\alpha\neq 0$, then $a_{k_j}$ is eventually bounded away from $0$ while $b_{k_j} \to 0$. Hence
 $d_\P([\tilde C^{\tilde n_{k_j}}_\psi], e_-(k_j)) \to 0$ as $j\to \infty$. 
 
 Moreover,
 \begin{align*}
 	e_-(k) = e_+(((f_{\sigma^{\tilde n_k}(\omega)}^{-\tilde m_k})^*)^t),
 \end{align*}
	 so by \ref{good:H} we have
	 \begin{align*}
	 	d_\P(e_-(k), e^-(\sigma^{\tilde n_k}(\omega))) < \frac{1}{i_k} \to 0.
	 \end{align*}
  Therefore
 $d_\P([\tilde C^{\tilde n_{k_j}}_\psi], e^-(\sigma^{\tilde n_{k_j}}(\omega))) \to 0$, proving the claim since $[\tilde C^{\tilde n_{k_j}}_\psi] \to \alpha$. 
\end{proof}

We are now ready to prove Proposition \ref{lem:last2}.
Suppose that the sequence
\begin{align}\label{contr}
	\left( \int_X T_{\omega}^+ \wedge \psi C \right)T_{\sigma^{\tilde n_k}(\omega)}^-
	- \left( \int_X T_{\sigma^{\tilde n_k}(\omega)}^+ \wedge T_{\sigma^{\tilde n_k}(\omega)}^-\right)
	M((f_\omega^{\tilde n_k})_* T_\omega^+)\, C_\psi^{\tilde n_k}(\omega)
\end{align}
does not converge to the zero current. By compactness, Claim \ref{claim:addedlater}, and Proposition \ref{lem:closed}, after passing to a subsequence $(\ell_k)$, we may assume that
\begin{align}\label{cool}
	T_{\sigma^{\ell_k}(\omega)}^- \to T_1
	\quad \text{and} \quad
	M((f_\omega^{\ell_k})_* T_\omega^+)\, C_\psi^{\ell_k}(\omega) \to T_2,
\end{align}
where $T_1$ and $T_2$ are closed positive currents satisfying $M(T_1) = 1$ and $M(T_2) \leq R$.%

By the triangle inequality, 
\begin{align*}
	d_\Theta(\tilde C_\psi^{\ell_k}, T_2) &\leq d_\Theta\big(\tilde C_\psi^{\ell_k}, M((f_\omega^{\ell_k})_* T_\omega^+)\, C_\psi^{\ell_k}(\omega)\big)  \\& \hspace{2 cm} + d_\Theta\big(M((f_\omega^{\ell_k})_* T_\omega^+)\, C_\psi^{\ell_k}(\omega), T_2 \big)  \to 0
\end{align*} where $d_\Theta\big(M((f_\omega^{\ell_k})_* T_\omega^+)\, C_\psi^{\ell_k}(\omega), T_2 \big)\to 0$ since $d_\Theta$ is continuous with the weak topology on currents and $d_\Theta\big(\tilde C_\psi^{\ell_k}, M((f_\omega^{\ell_k})_* T_\omega^+)\, C_\psi^{\ell_k}(\omega)\big) \to 0$ by construction. Hence 
\begin{align*}
	[\tilde C_\psi^{\ell_k}]\to [T_2],
\end{align*} and applying Claim \ref{claim:addedlater} yields that either $[T_2] = 0$ or
\begin{align}\label{scratch}
	d_\P([T_2], e^-(\sigma^{\ell_k}(\omega))) \to 0.
\end{align}

Assume first that $[T_2]=0$. Since $T_2$ is a closed positive current, $[T_2] = 0$ implies $M(T_2) = 0$, and hence $T_2=0$. Therefore
\begin{align*}
	M((f_\omega^{\ell_k})_* T_\omega^+)\, C_\psi^{\ell_k}(\omega) \to 0
\end{align*}
as currents. To apply Lemma \ref{lem:curr}, take $a_k =  M((f_\omega^{\ell_k})_* T_\omega^+)$ and choose any $\psi_0 \in C^\infty(X)$ with $\restr{\psi_0}{\supp \psi} \equiv 1$ and $\supp \psi_0 \cap \supp \partial C = \varnothing$. Notice $a_k \to 0$ by our choice of $\ell_k$. Now by Remark \ref{rmk:sequence-independent}, the sequence $(n_k)$ from Proposition \ref{lem:MassBound2} may be chosen independently of $C$ and $\psi$. Hence $a_k(C_{\psi_0}^{\ell_k}(\omega)\mid \kappa_0)$ remains uniformly bounded, and we have $a_k(C_{\psi_0}^{\ell_k}(\omega)\mid \kappa_0)^{\frac 1 2} \lesssim a_k^{\frac 1 2}\to 0$. We then apply Lemma \ref{lem:curr}, which yields
\begin{align*}
	T_{\sigma^{\ell_k}(\omega)}^+ \wedge M((f_\omega^{\ell_k})_* T_\omega^+) C_\psi^{\ell_k}(\omega) \to 0
\end{align*}
weakly as measures. On the other hand, Remark \ref{rmk:identity} gives
\begin{align}\label{constant}
		\int_X T_{\sigma^{\ell_k}(\omega)}^+ \wedge M((f_\omega^{\ell_k})_* T_\omega^+) C_\psi^{\ell_k}(\omega)
		&= 	\int_X T_\omega^+ \wedge   \psi C  
\end{align}
	which is independent of $k$. Hence this constant must be zero. Since $\int_X T_{\sigma^{\ell_k}(\omega)}^+\wedge T_{\sigma^{\ell_k}(\omega)}^-$ is uniformly bounded above in $k$, \eqref{contr} converges to the zero current, a contradiction.

	Assume $[T_2 ]\neq 0$.
Now $T^-_{\sigma^{\ell_k}(\omega)} \to T_1$ implies 
\begin{align*}
	e^-(\sigma^{\ell_k}(\omega))\to [T_1].
\end{align*} Combined with \eqref{scratch}, we obtain $d_\P([T_2], [T_1]) = 0$. Hence there exists $a \neq 0$ such that $[T_2] = a[T_1]$; namely, $a = M(T_2)$. Recall $K'\subset \Omega$ was a compact subset for which $\eta \mapsto e^-(\eta)$ is continuous and $T^-_\eta$ is the unique closed positive current in the class $e^-(\eta)$. Moreover, by construction, $\sigma^{\tilde n_k}(\omega)\in K'$ for all $k$. By compactness of $K'$ and continuity of $\eta \mapsto e^-(\eta)$ on $K'$, the set $e^-(K')$ is compact; hence, $[T_1 ] = e^-(\eta)$ for some $\eta\in K'$. By uniqueness, $[T_2] = M(T_2) [T_1]$ implies $T_2 = M(T_2) T_1$. So 
\begin{align}\label{cool2}
	T_{\sigma^{\ell_k}(\omega)}^- \to T_1
	\quad \text{and} \quad
	M((f_\omega^{\ell_k})_* T_\omega^+)\, C_\psi^{\ell_k}(\omega) \to T_2 = M(T_2) T_1.
\end{align}

By construction, the family
\begin{align*}
		\set{\phi(T_{\sigma^{\tilde n_k}(\omega)}^+) \colon k\in \N}
\end{align*}
is equicontinuous and uniformly bounded in $C^0$. Applying Lemma \ref{lem:curr} to \eqref{cool2}, we obtain
\begin{align}\label{hm1}
	T_{\sigma^{\ell_k}(\omega)}^+ \wedge M((f_\omega^{\ell_k})_* T_\omega^+) C_\psi^{\ell_k}  
	- T_{\sigma^{\ell_k}(\omega)}^+ \wedge T_2
	\to 0.
\end{align}
On the other hand, Lemma \ref{lem:currClosed} gives
\begin{align}\label{hm}
	T_{\sigma^{\ell_k}(\omega)}^+\wedge T_{\sigma^{\ell_k}(\omega)}^-  
	-
	T_{\sigma^{\ell_k}(\omega)}^+ \wedge T_1
	\to 0.
\end{align}
Combining \eqref{hm1} and \eqref{hm}, and recalling that $T_2 = M(T_2)T_1$, we obtain
\begin{align}\label{integrate}
	M(T_2) T_{\sigma^{\ell_k}(\omega)}^+ \wedge  T_{\sigma^{\ell_k}(\omega)}^- 
	- M((f_\omega^{\ell_k})_* T_\omega^+) T_{\sigma^{\ell_k}(\omega)}^+ \wedge C_\psi^{\ell_k}
	\to 0.
\end{align}
Integrating \eqref{integrate} over $X$ and recalling \eqref{constant}, we find
\begin{align}\label{grad}
	\frac{\int_X T_{\omega}^+ \wedge \psi C}
	{\int_X T_{\sigma^{\ell_k}(\omega)}^+ \wedge T_{\sigma^{\ell_k}(\omega)}^-}
	\to M(T_2),
\end{align}
	where Remark \ref{tbd} is used to justify division by $\int_X T_{\sigma^{\ell_k}(\omega)}^+ \wedge T_{\sigma^{\ell_k}(\omega)}^-$.

Since the currents $T_{\sigma^{\ell_k}(\omega)}^-$ have bounded mass (namely unit mass), it follows that
\begin{align}\label{up}
	\frac{\int_X T_{\omega}^+ \wedge \psi C}
	{\int_X T_{\sigma^{\ell_k}(\omega)}^+ \wedge T_{\sigma^{\ell_k}(\omega)}^-} T_{\sigma^{\ell_k}(\omega)}^-
	- M(T_2)T_{\sigma^{\ell_k}(\omega)}^- \to 0.
\end{align}
Moreover,
\begin{align}\label{down}
	M(T_2) T_{\sigma^{\ell_k}(\omega)}^-
	- M((f_\omega^{\ell_k})_* T_\omega^+) C_\psi^{\ell_k}(\omega)
	\to 0.
	\end{align} This follows from \eqref{cool} and the equality $T_2 = M(T_2) T_1$. Substituting \eqref{down} into \eqref{up}, and using the uniform upper bound on $\int_X T_{\sigma^{\ell_k}(\omega)}^+ \wedge T_{\sigma^{\ell_k}(\omega)}^-$ 
which follows from compactness of the unit ball in cohomology, we conclude that \eqref{contr} converges to the zero current, a contradiction. Thus \eqref{important:seq} holds.

To prove \eqref{important:currs}, we apply the same argument. Indeed, if \eqref{important:currs} does not hold then there exists a subsequence $(\ell_k)$ and closed positive currents $T_1, T_2$ such that 
\begin{align}
	T_{\sigma^{\ell_k}(\omega)}^- \to T_1
	\quad \text{and} \quad
	M((f_\omega^{\ell_k})_* T_\omega^+)\, C_\psi^{\ell_k}(\omega) \to T_2 
\end{align} but 
\begin{align}
	\left(\int_X T_{\omega}^+ \wedge \psi C\right)T_{\sigma^{\ell_k}(\omega)}^+ \wedge  T_{\sigma^{\ell_k}(\omega)}^- 
	- \left(\int_X T_{\sigma^{\ell_k}(\omega)}^+ \wedge T_{\sigma^{\ell_k}(\omega)}^-\right)M((f_\omega^{\ell_k})_* T_\omega^+) T_{\sigma^{\ell_k}(\omega)}^+ \wedge C_\psi^{\ell_k}
\end{align} does not converge to the zero measure. 
If $T_2=0$, then the same argument as above gives
\begin{align*}
	T_{\sigma^{\ell_k}(\omega)}^+ \wedge M((f_\omega^{\ell_k})_*T_\omega^+)C_\psi^{\ell_k}\to0.
\end{align*}
By \eqref{constant}, the scalar $\int_X T_\omega^+\wedge\psi C$ is zero, and hence the displayed sequence converges to the zero measure, a contradiction. Thus we may assume $T_2\neq0$.
But \eqref{integrate} and \eqref{grad} prove the above sequence does converge to the zero measure, a contradiction. 

	The preceding argument gives the upper-density statement. Assume now that $\mu$ satisfies the exponential moment condition \eqref{exponential}. We explain how to modify the construction to obtain lower density. The only point in the preceding convergence proof where the Gou\"ezel--Karlsson tightness condition was used was in Claim \ref{claim:shrink}, to show that
	\begin{align*}
		\frac{\norm{(f_\omega^{\tilde n_k-\tilde m_k})^*}}
		{\norm{(f_\omega^{\tilde n_k})^*}}
		\to0.
	\end{align*}
	Since we do not have tightness in the exponential moment case, we replace this by the following estimate from Cantat--Dujardin. By Proposition 5.14 of \cite{MR4635340}, applied with $D\equiv1$, there exists a finite $\mu^\Z$-measurable function $b(\omega)$ such that, for every $n\geq1$,
	\begin{align}\label{cd:ratio-sum}
		\sum_{m=1}^{n-1}
		\frac{\norm{(f_\omega^{n-m})^*}}
		{\norm{(f_\omega^n)^*}}
		\leq b(\omega).
	\end{align}
	The idea is that, given $n$, \eqref{cd:ratio-sum} forces $\norm{(f_\omega^{n-m})^*}/\norm{(f_\omega^n)^*}$ to be small for many values of $m$. We will then choose $m$ separately for each $n$, rather than using a single value of $m$ throughout a block of times.
	
	Choose the initial set $S$ from Proposition \ref{lem:MassBound2} with
	\begin{align*}
		\underline d(S)>1-\rho'.
	\end{align*}
		Fix $0<\epsilon<\lambda_\mu$. With $\mathcal L_m$, $\mathcal R_m$, and $\mathcal H_{m,i}$ as in the proof of Claim \ref{claim:2}, choose numbers $\gamma_i\to0$ and finite intervals $I_i\subset\N$ such that $m_i\coloneqq \min I_i\to\infty$ and
	\begin{align*}
		\frac{b(\omega)}{\gamma_i |I_i|}\to0.
	\end{align*}
	By \eqref{claim2:good-measures}, we may take the intervals $I_i$ far enough out so that, for every $m\geq m_i$,
	\begin{align*}
		\mu^\Z(\mathcal L_m)>1-\frac{1}{2i},\qquad
		\mu^\Z(\mathcal R_m)>1-\frac{1}{2i},\qquad
		\mu^\Z(\mathcal H_{m,i})>1-\frac{1}{2i}.
	\end{align*}
		Set $M_i\coloneqq \max I_i$. We choose $c_0=0$ and then choose $c_i$ recursively so that the interval analogs of conditions \eqref{claim2:separation}--\eqref{claim2:H-average} in Claim \ref{claim:2} hold. Namely, $c_i>M_i$, $c_i-M_i\to\infty$, $c_{i-1}/c_i\to0$, and the following estimates hold for every $N\geq c_i$. The analog of \eqref{claim2:growth} is
	\begin{align*}
		\norm{(f_\omega^N)^*}> i e^{(\lambda_\mu+\epsilon)M_i}.
	\end{align*}
		The analog of \eqref{claim2:closed-approx} is that, for every $n\in S$ with $n\geq c_i$, there exists a closed positive current $T$ such that
	\begin{align*}
		d_\Theta(M((f_\omega^n)_*T_\omega^+)C_\psi^n,T)
		<\frac1i e^{-(\lambda_\mu+\epsilon)M_i}.
	\end{align*}
		The analogs of \eqref{claim2:lyap-average}--\eqref{claim2:H-average} are that, for every $m\in I_i$ and $N\geq c_i$,
\begin{equation}\label{latere}
	\begin{aligned}
		&\frac1N\sum_{k=1}^N \mathds{1}_{\mathcal L_m}(\sigma^k(\omega))
		>1-\frac1i,\hspace{.7cm}
		\frac1N\sum_{k=1}^N \mathds{1}_{\mathcal R_m}(\sigma^k(\omega))
		>1-\frac1i,\\
		& \hspace{3 cm} \frac1N\sum_{k=1}^N \mathds{1}_{\mathcal H_{m,i}}(\sigma^k(\omega))
		>1-\frac1i.
	\end{aligned}
\end{equation}
	
	Finally, 
	\begin{align*}
		\underline d(S\cap(S+m))\geq 2\underline d(S)-1
	\end{align*}
	for every fixed $m\in\N$. Since $I_i$ is finite, after increasing $c_i$ we may arrange that, for every $m\in I_i$ and every $N\geq c_i$,
		\begin{align*}
			\#(S \cap (S + m)\cap[1,N]) \geq \left(2\underline d(S)-1-\frac1i\right)N.
		\end{align*} Hence
			\begin{align}\label{dense:new}
				\#\set{(n,m)\in [1,N]\times I_i\colon n\in S,\ n-m\in S}
				\geq
				\left(2\underline d(S)-1-\frac1i\right)N|I_i|.
			\end{align}

		For $m\in I_i$, write
		\begin{align*}
			Q_{i,m}\coloneqq \mathcal L_m \cap \mathcal R_m \cap \mathcal H_{m,i}.
		\end{align*}
		Then $\sigma^n(\omega)\in Q_{i,m}$ is precisely the condition that the three estimates \ref{good:lyap}--\ref{good:H} hold at $k=n$ with $m$ in place of $m_i$.
		
		For $n\in S$, say that $m\in I_i$ is admissible for $n$ at stage $i$ if
		\begin{align*}
			n-m\in S,\qquad
			\sigma^n(\omega)\in Q_{i,m},\qquad
			\frac{\norm{(f_\omega^{n-m})^*}}
			{\norm{(f_\omega^n)^*}}
			\leq \gamma_i.
		\end{align*}
		Let $E_i$ be the set of $n\in S$ for which there exists an admissible $m\in I_i$. For such an $n$, let $m(n)$ be the least admissible $m\in I_i$. By \eqref{cd:ratio-sum}, for each fixed $n$ there are at most $b(\omega)/\gamma_i$ values of $m$ for which the last displayed inequality fails. Combining this with \eqref{dense:new} and \eqref{latere}, we obtain, for every $N\geq c_i$,
		\begin{align*}
		\#(E_i\cap[1,N])\cdot 	|I_i|
			\geq
			\left(2\underline d(S)-1-\frac1i\right)N	|I_i|
			-\frac 3 i N|I_i|
			-\frac{b(\omega)}{\gamma_i}N.
		\end{align*}
		Indeed,	$\#(E_i\cap[1,N])\cdot 	|I_i|$ is at least the number of pairs $(n,m)\in [1,N]\times I_i$ such that $m$ is admissible for $n$. Then \eqref{dense:new} gives the first term. Then the three conditions from \eqref{latere} remove at most $3N|I_i|/i$ pairs, while \eqref{cd:ratio-sum} removes at most $b(\omega)N/\gamma_i$ pairs. 
		
		Dividing by $|I_i|$ yields
	\begin{align*}
		\#(E_i\cap[1,N]) \geq 
		\left(2\underline d(S)-1-\alpha_i\right)N,
	\end{align*} where
	\begin{align*}
		\alpha_i\coloneqq \frac4i+\frac{b(\omega)}{\gamma_i |I_i|}
	\end{align*}
	satisfies $\alpha_i\to0$.

	Now set
	\begin{align*}
		A\coloneqq \bigcup_{i\geq1}\left(E_i\cap(c_i,c_{i+1}]\right).
	\end{align*}
	If $N$ is large and $i=i(N)\geq2$ is determined by $c_i<N\leq c_{i+1}$, then
	\begin{align*}
		\#(A\cap[1,N])
		&\geq
		\#(E_i\cap[1,N])-c_i
		+\#(E_{i-1}\cap[1,c_i])-c_{i-1} \\
		&\geq
		(2\underline d(S)-1-\alpha_i)N
		+(2\underline d(S)-2-\alpha_{i-1})c_i
		-c_{i-1}.
	\end{align*}
	Dividing by $N$ and taking $\liminf$ gives
	\begin{align*}
		\underline d(A)\geq 4\underline d(S)-3,
	\end{align*}
	since $c_i/N\leq1$, $c_{i-1}/N\leq c_{i-1}/c_i\to0$, and $\alpha_i\to0$. Choosing $\rho'<\rho/4$, and then choosing the auxiliary sets $K$, $K'$, and $D_\tau$ with sufficiently small density losses, the final set
	\[
			\set{\tilde n_k}_{k\in\N}
			\coloneqq
			A\cap \set{n\in\N\colon \sigma^n(\omega)\in K\cap K'\cap D_\tau}
		\]
	has lower-density greater than $1-\rho$.

		For each $\tilde n_k\in A$, let $i_k$ be the unique index such that $c_{i_k}<\tilde n_k\leq c_{i_k+1}$ and set $\tilde m_k\coloneqq m(\tilde n_k)$. Then $\tilde m_k\to\infty$ and $\tilde n_k-\tilde m_k\to\infty$, because $\min I_i\to\infty$ and $c_i-M_i\to\infty$. Since $\tilde m_k\leq M_{i_k}$, the estimates corresponding to \eqref{claim2:growth} and \eqref{claim2:closed-approx} give
	\begin{align*}
		\norm{(f_\omega^{\tilde n_k})^*}
		> i_k e^{(\lambda_\mu+\epsilon)\tilde m_k} \hspace{.25 cm} \text{and} \hspace{.25 cm}
		d_\Theta(M((f_\omega^{\tilde n_k})_*T_\omega^+)C_\psi^{\tilde n_k},\tilde C_\psi^{\tilde n_k})
		<\frac1{i_k}e^{-(\lambda_\mu+\epsilon)\tilde m_k}
	\end{align*}
	for some closed approximants $\tilde C_\psi^{\tilde n_k}$. The rest of the convergence proof is unchanged aside from Claim \ref{claim:shrink}, where the term that previously used Gou\"ezel--Karlsson is now bounded by
	\begin{align*}
		\frac{1}{\norm{(f_\omega^{\tilde n_k})^*}}
		M(C_\psi^{\tilde n_k-\tilde m_k})
		&\lesssim
		\frac{\norm{(f_\omega^{\tilde n_k-\tilde m_k})^*}}
		{\norm{(f_\omega^{\tilde n_k})^*}}
		\leq \gamma_{i_k}\to0.
	\end{align*}
\end{proof}

	\section{Proof of Theorem \ref{thm:main}}
	
		Let $\nu \in S(\mu)$ have maximal fiber entropy. Notice that if $\nu$ were not ergodic, then by ergodicity of $\mu^\Z$ under the shift, almost every ergodic component would belong to $S(\mu)$. Moreover, almost every ergodic component would have maximal fiber entropy. Therefore it is enough to treat the case where $\nu$ is ergodic. 
		
		By Theorem \ref{thm:var}, $h_{\nu}^\mathcal{F}=\lambda_\mu>0$. Thus Lemma \ref{posent} implies that $\nu$ is hyperbolic and that the stable and unstable manifolds are Zariski dense for $\nu$-almost every $(\omega,x)$. Fix a $\nu$-measurable partition $\eta$ given by Proposition \ref{prop:partition}, subordinate to the unstable manifolds of $\nu$ and refining the partition into fibers of $p\colon \Omega\times X\to \Omega$. As in Definition \ref{def:subordinate}, write
		\begin{align*}
			\eta_\omega(x)\coloneqq \set{y\in X\colon (\omega,y)\in \eta(\omega,x)}.
		\end{align*}
		Recalling the outline in Section \ref{sec:MainOutline}, our proof is split into two steps.

		\medskip
		\noindent\textbf{Step 1.}
		\medskip
		
			Since $\eta_\omega(x)$ is almost everywhere bounded in $W^u_\omega(x)$ (see Remark \ref{bdd}), by Section \ref{sec:slices} we obtain a Borel measure 
		\begin{align*}
			\restr{T_\omega^+}{\eta_\omega(x)}  \coloneqq \mathds{1}_{\eta_\omega(x)}\restr{T_\omega^+}{W^u_\omega(x)}
		\end{align*} on $X$. Consider the function 
	\begin{align*}
			f(\omega, x) \coloneqq \int_X \restr{T_\omega^+}{\eta_\omega(x)} = \restr{T_\omega^+}{W^u_\omega(x)}(\eta_\omega(x)).
	\end{align*} 
Notice that $f(\omega, x)$ is defined almost-everywhere, and since  $\restr{T_\omega^+}{W^u_\omega(x)}$ is Radon and $\eta_\omega(x)$ is bounded in $W^u_\omega(x)$, $f(\omega, x)$ is almost-everywhere finite. 

\begin{claim}\label{claim:measurable}
	The assignment $(\omega, x)\mapsto f(\omega, x)$ is $\nu$-measurable.
	\begin{proof}
		Recall from Lemma \ref{lem:measurable-slices} that
		\begin{align*}
			(\omega,x)\mapsto \restr{T_\omega^+}{W^{u,\mathrm{loc}}_\omega(x)}
		\end{align*}
		is $\nu$-measurable, with values in the space of finite Borel measures on $X$ equipped with the weak topology; that is, it defines a measurable kernel $\Omega \times X \to X$.

		For each $n\in\N$,
		\begin{align*}
			f(\omega,x)
			&=
			\restr{T_\omega^+}{W^u_\omega(x)}(\eta_\omega(x)) \\
			&=
			\frac{1}{M((f_{\sigma^{-n}(\omega)}^n)_*T_{\sigma^{-n}(\omega)}^+)}
			(f_{\sigma^{-n}(\omega)}^n)_*
			\left(\restr{T_{\sigma^{-n}(\omega)}^+}{W^u_{\sigma^{-n}(\omega)}(f_\omega^{-n}(x))}\right)(\eta_\omega(x)) \\
			&=
			\frac{1}{M((f_{\sigma^{-n}(\omega)}^n)_*T_{\sigma^{-n}(\omega)}^+)}
			\restr{T_{\sigma^{-n}(\omega)}^+}{W^u_{\sigma^{-n}(\omega)}(f_\omega^{-n}(x))}
			(f_{\omega}^{-n}(\eta_\omega(x))).
		\end{align*}
		Almost surely, $f_{\omega}^{-n}(\eta_\omega(x))$ is contained in a local unstable manifold for all sufficiently large $n$. Hence
		\begin{align*}
			f(\omega,x)
			=
			\lim_{n\to \infty}
			\frac{1}{M((f_{\sigma^{-n}(\omega)}^n)_*T_{\sigma^{-n}(\omega)}^+)}
			\restr{T_{\sigma^{-n}(\omega)}^+}{W^{u,\mathrm{loc}}_{\sigma^{-n}(\omega)}(f_\omega^{-n}(x))}
			(f_{\omega}^{-n}(\eta_\omega(x))).
		\end{align*}
		It remains to show that the terms in this limit are $\nu$-measurable. More generally, we claim that if $\gamma$ is a $\nu$-measurable partition of $\Omega\times X$ and $(\omega, x) \mapsto \mu_{(\omega, x)}$ is a $\nu$-measurable kernel with values in finite positive Borel measures on $\Omega \times X$, endowed with the weak topology, then
			\begin{align*}
					(\omega,x)\mapsto
					\mu_{(\omega, x)}(\gamma(\omega ,x))
				\end{align*}
				is $\nu$-measurable. Applying this claim with $\gamma=F^{-n}\eta$ and then composing with $F^{-n}$, we obtain the measurability of the $n$-th term in the displayed limit. 
				
				To prove the claim, recall that since $\gamma$ is measurable, there exists a standard Borel space $Y$ and a $\nu$-measurable map $q:\Omega \times X \to Y$ such that $\gamma(\omega, x) = q^{-1}(q(\omega, x))$ almost surely. Then since $Y$ is standard Borel, the diagonal $\Delta \subset Y\times Y$ is measurable, so $(q\times q)^{-1}(\Delta) = \set{((\omega, x), (\omega', x')) \colon \gamma(\omega, x) = \gamma(\omega', x')}$ is measurable. Hence $((\omega, x), (\omega', x'))\mapsto \mathds{1}_{\gamma(\omega, x)}(\omega', x')$ is measurable and one applies Lemma 3.2 (i) of \cite{MR1876169} to conclude that 
				\begin{align*}
					(\omega,x)\mapsto
					\mu_{(\omega, x)}(\gamma(\omega ,x)) = \int_{\Omega \times X} \mathds{1}_{\gamma(\omega, x)}(\omega', x')\,d\mu_{(\omega, x)}(\omega', x')
				\end{align*} is $\nu$-measurable.
				
			\end{proof}
		\end{claim}

	\begin{claim}\label{claim:positive}
		$f(\omega, x)$ is almost-everywhere positive.
		\begin{proof}
		By the equivariance property \eqref{equivar} and since $F\eta \leq \eta$, the set \[\set{(\omega, x)\mid f(\omega, x) = 0}\] is $F$-invariant and thus has either full or null $\nu$-measure. 

		Assume the former. Then by \eqref{equivar} again we find that $\restr{T_\omega^+}{W^u_\omega(x)}((F^n \eta)_\omega(x)) = 0$ for all $n$. Moreover, $\nu$-almost everywhere, $\bigcup_{n\geq0}(F^n\eta)_\omega(x)=W^u_\omega(x)$;
		this follows from Remark \ref{bdd}, since $\eta_\omega(x)$ contains a relatively open subset of the unstable manifold, with radius bounded below by a positive measurable function.
	Arguing as in the proof of Theorem \ref{big} (see Theorem 8.2 and Lemma 8.8 of
		\cite{MR4635340}), if $T$ is a closed positive $(1,1)$-current with unit mass and continuous potentials, then $\restr{T}{W^u_\omega(x)}=0$ implies that $T$ is the unique Ahlfors--Nevanlinna current with unit mass associated to $W^u_\omega(x)$, namely $T=T_\omega^-$. We have shown that $\restr{T_\omega^+}{W^u_\omega(x)}=0$ generically, so this yields $T_\omega^+=T_\omega^-$ almost surely, a contradiction.
		\end{proof}
	\end{claim}
	
	By the same measurable-kernel argument used in Claim \ref{claim:measurable}, the assignment $(\omega,x)\mapsto \delta_\omega\otimes \restr{T_{\omega}^+}{\eta_\omega(x)}$ defines a $\nu$-measurable kernel. Since $f$ is positive almost everywhere, the same is true after dividing by $f$.
	
	\begin{claim}\label{lem:main} 
		Let $\nu^*$ be the measure on $\Omega \times X$ defined by 
		\begin{align*}
			\nu^* \coloneqq \int \frac{\delta_{\omega}\otimes \restr{T_{\omega}^+}{\eta_\omega(x)}}{f(\omega, x)} \,d\nu.
		\end{align*} Then $\nu = \nu^*.$
	\end{claim}
	\noindent By the uniqueness of conditional measures, Claim \ref{lem:main} implies:
	\begin{cor} For $\nu$-almost every $(\omega, x)$, 
		\begin{align*}
			\nu_{\eta_\omega(x)} = \frac{\restr{T_{\omega}^+}{\eta_\omega(x)}}{f(\omega, x)}.
		\end{align*}
	\end{cor}
	\begin{proof}[Proof of Claim \ref{lem:main}]
				By construction, $\nu$ and $\nu^*$ agree on $B(\eta)$, the collection of measurable sets that are unions of $\eta$-atoms. We will show that they also agree on $B(F^{-1}\eta)$. By applying the same argument successively to $B(F^{-i}\eta)$ for each $i\geq 1$, it follows that $\nu = \nu^*$ since $\lim_{n\to \infty}F^{-n}\eta$ is the partition into points.
		
		We begin by noting that $x\in \supp \nu_{\eta_\omega(x)}$ for $\nu$-almost every $(\omega,x)$; this follows from the defining properties of conditional measures. 
		Therefore we may consider the ratio 
		\begin{align*}
			\frac{\restr{T_{\omega}^+}{W^u_\omega(x)}((F^{-1}\eta)_\omega(x))}{f(\omega, x)\nu_{\eta_\omega(x)}((F^{-1}\eta)_\omega(x))}
		\end{align*} which one can see is positive almost surely by replacing $\eta$ with $F^{-1}\eta$ in Claim \ref{claim:positive}. From the defining properties of conditional measures and the fact that $F^{-1}\eta$ refines $\eta$, it is straightforward to check that 
		\begin{align*}
			\int	\frac{\restr{T_{\omega}^+}{W^u_\omega(x)}((F^{-1}\eta)_\omega(x))}{f(\omega, x)\nu_{\eta_\omega(x)}((F^{-1}\eta)_\omega(x))}\,d\nu= 1.
		\end{align*} So by Jensen's inequality,
		\begin{align*}
			\int \ln \left(	\frac{\restr{T_{\omega}^+}{W^u_\omega(x)}((F^{-1}\eta)_\omega(x))}{f(\omega, x)\nu_{\eta_\omega(x)}((F^{-1}\eta)_\omega(x))}  \right) \,d\nu  \leq 0
		\end{align*} with equality if and only if 
		\begin{align*}
				\frac{\restr{T_{\omega}^+}{W^u_\omega(x)}((F^{-1}\eta)_\omega(x))}{f(\omega, x)\nu_{\eta_\omega(x)}((F^{-1}\eta)_\omega(x))} =  1
		\end{align*} $\nu$-almost everywhere. Notice
	\begin{equation}\label{exp1}
		\begin{aligned}
			&\int \ln \left(	\frac{\restr{T_{\omega}^+}{W^u_\omega(x)}((F^{-1}\eta)_\omega(x))}{f(\omega, x)\nu_{\eta_\omega(x)}((F^{-1}\eta)_\omega(x))}  \right) \,d\nu \\
			&\hspace{4 cm} = \int \ln(\frac{\restr{T_{\omega}^+}{W^u_\omega(x)}((F^{-1}\eta)_\omega(x))}{f(\omega, x)})\,d\nu \\
			&\hspace{7 cm} -\int \ln \nu_{\eta_\omega(x)}((F^{-1}\eta)_\omega(x)) \,d\nu 
		\end{aligned}
	\end{equation}
	\noindent and by Proposition \ref{Prop:full},
	\begin{equation}\label{exppp2}
		-\int \ln \nu_{\eta_\omega(x)}((F^{-1}\eta)_\omega(x)) \,d\nu = H(F^{-1}\eta \mid \eta) = h_{\nu}^\mathcal{F}.
	\end{equation}
Furthermore, we have the identity
	\begin{equation}\label{identity}
		\begin{aligned}
				\restr{T_{\omega}^+}{W^u_\omega(x)}((F^{-1}\eta)_\omega(x)) &= \restr{T_{\omega}^+}{W^u_\omega(x)}(\omega_0^{-1}(\eta_{\sigma(\omega)} (\omega_0 x)))\\
			&= \restr{\left((\omega_0)_*T_{\omega}^+\right)}{W^u_{\sigma(\omega)}(\omega_0 x)}(\eta_{\sigma(\omega)}(\omega_0 x)) \\
			&= M((\omega_0)_*T_{\omega}^+) \restr{T_{\sigma(\omega)}^+}{W^u_{\sigma(\omega)}(\omega_0 x)}(\eta_{\sigma(\omega)}(\omega_0 x)) \\
		&= M((\omega_0)_*T_{\omega}^+) f(F(\omega, x))
		\end{aligned}
	\end{equation}
 which we apply to find 
		\begin{align}\label{exp2}
				&\int \ln(\frac{\restr{T_{\omega}^+}{W^u_\omega(x)}((F^{-1}\eta)_\omega(x))}{f(\omega, x)}) \,d\nu \notag  \\
			& \hspace{4 cm}= \int \ln \left(\frac{M((\omega_0)_*T_{\omega}^+) f(F(\omega, x))}{ f(\omega, x)} \right)	\,d\nu \notag\\
			& \hspace{4 cm}= \int \ln M((\omega_0)_*T_{\omega}^+) \,d\nu + \int \ln \frac{f(F(\omega, x))}{f(\omega, x)} \,d\nu
			\notag \\
			& \hspace{4 cm}= -\int \ln M(\omega_0^*T_{\sigma(\omega)}^+)\,d\nu + \int \ln \frac{f(F(\omega, x))} {f(\omega, x)} \,d\nu.
		\end{align} Now 
		\begin{align}\label{expone}
			\int \ln M(\omega_0^*T_{\sigma(\omega)}^+)\,d\nu = \int \ln M(\omega_0^*T_{\sigma(\omega)}^+)\,d\mu^\Z = \lambda_\mu
		\end{align} where the first equality holds since $p_* \nu = \mu^\Z$ and the second by the mass version of Furstenberg's formula (see Theorem \ref{thm:furst}). By \eqref{identity}, we also have
		\begin{align}\label{oneline}
			\frac{f(F(\omega, x))}{f(\omega, x)} &=   \frac{	\restr{T_{\omega}^+}{W^u_\omega(x)}((F^{-1}\eta)_\omega(x)) }{ M((\omega_0)_*T_{\omega}^+)  f(\omega, x)} \leq \frac{1}{M((\omega_0)_*T_{\omega}^+)}  =M(\omega_0^*T_{\sigma(\omega)}^+)
				\end{align} where the inequality uses that $F^{-1}\eta$ is a refinement of $\eta$, and the final equality follows from \eqref{equivar}. By compactness, we may pick $C > 1$ such that $M(g^*\alpha)\leq C\norm{g^*}$ for all $g\in \Aut(X)$ and all closed positive $(1,1)$-currents $\alpha$ with unit mass. Then by \eqref{oneline},
		\begin{align*}
			\int \ln^+\frac{f(F(\omega, x))}{f(\omega, x)} \,d\nu &\leq \int \ln^+ M(\omega_0^*T_{\sigma(\omega)}^+)\,d\nu  \\ 
				&\leq \int \ln^+ (C \norm{g^*})\,d\mu(g)  \\
				&\leq \ln C + \int \ln^+(\norm{g^*})\,d\mu(g)
		\end{align*} which is $<\infty$ by \eqref{int:cohomology}.
		Now $f$ is a positive, finite measurable function satisfying $\ln^+\frac{f(F(\omega, x))}{f(\omega, x)}\in L^1(\Omega\times X)$. By Lemma I.3.1 of \cite{MR1369243},
		\begin{align}\label{exptwo}
			\int\ln\frac{f(F(\omega, x))}{f(\omega, x)} \,d\nu= 0.
		\end{align} Plugging \eqref{expone} and \eqref{exptwo} into \eqref{exp2} we have 
		\begin{align}\label{expp2}
				\int \ln(\frac{\restr{T_{\omega}^+}{W^u_\omega(x)}((F^{-1}\eta)_\omega(x))}{f(\omega, x)}) \,d\nu =  - \lambda_\mu.
		\end{align} 
		Now combining \eqref{expp2} and \eqref{exppp2} with \eqref{exp1}, we have proven that 
		\begin{align*}
			h_{\nu}^\mathcal{F} - \lambda_\mu  \leq 0
		\end{align*} with equality if and only if $\nu = \nu^*$ on $B(F^{-1}\eta)$, as desired.
	\end{proof}
	
		\medskip
	\noindent\textbf{Step 2.}
	 \medskip
	
	Since $p_*\nu = \mu^\Z$, if a set has full $\mu^\Z$-measure in $\Omega$ then its inverse image under $p$ has full $\nu$-measure. We deduce that there exists a set $\Lambda\subset \Omega\times X$ with full $\nu$-measure such that for all $(\omega, x)\in \Lambda$ we may define the sequences
	\begin{align*}
		m_N \coloneqq \frac 1 N \sum_{n=1}^N \delta_{\sigma^n(\omega)} \otimes m_{\sigma^n(\omega)}
	\end{align*}
	and 
	\begin{align*}
		\nu_N \coloneqq \frac 1 N \sum_{n=1}^N \delta_{\sigma^n(\omega)} \otimes (f_\omega^n)_* \cN( \restr{T_\omega^+}{\eta_\omega(x)}).
	\end{align*}
	For every $f\in C_b(\Omega \times X)$, 
	\begin{align*}
		\omega \mapsto \int f_{\omega} \,dm_\omega
	\end{align*}
	is measurable and in $L^1(\mu^\Z)$, so the ergodicity of $\mu^\Z$ with respect to $\sigma$ and the ergodic theorem yield
	\begin{align*}
			\frac 1 N \sum_{n=1}^N \int f_{\sigma^n(\omega)} \,dm_{\sigma^n(\omega)} \to \int \int f_\omega  \,dm_\omega \,d\mu^\Z(\omega) = \int f  \,dm
	\end{align*}
	almost surely. Hence $m_N\to m$ weakly as $N\to \infty$. 
	
		In Step 1, we proved that $\nu$-almost surely,
	\begin{align*}
		\nu_N= \frac 1 N \sum_{n=1}^N F_*^n \nu_{\eta(\omega, x)}.
	\end{align*}
	Let $f\in C_b(\Omega\times X)$. By the ergodic theorem and the defining property of conditional measures, $\nu$-almost surely we have that $\nu_{\eta_\omega(x)}$-almost every $y\in \eta_\omega(x)$ satisfies
	\begin{align*}
		\frac 1 N \sum_{n=1}^N f(F^n(\omega, y)) \to \int f \,d\nu.
	\end{align*}
	Moreover, since $f$ is uniformly bounded, dominated convergence gives
	\begin{align*}
			\int f \,d\nu_N  = \frac 1 N \sum_{n=1}^N \int f(F^n(\omega, y)) \,d\nu_{\eta_\omega(x)}\to \int f \,d\nu.
	\end{align*}
		After applying the preceding two arguments to a countable family of test functions determining weak convergence and shrinking $\Lambda$, we may assume that, for every $(\omega,x)\in\Lambda$, $m_N\to m$ and $\nu_N\to\nu$ weakly.
	
	Now fix $(\omega,x)\in\Lambda$. Recalling Remark \ref{bdd}, we can choose a holomorphic parametrization $\zeta\colon \C\to W^u_\omega(x)$ with $\zeta(0)=x$ and open disks $\Delta_0,\Delta_1,\Delta_2\subset \C$, with $0\in\Delta_0$, such that
	\begin{align*}
		\zeta(\overline{\Delta_0})\subset \eta_\omega(x) \subset \zeta(\Delta_1) \hspace{.5 cm}\text{and}\hspace{.5 cm} \overline{\Delta_1}\subset \Delta_2.
	\end{align*}
	Set $D\coloneqq \zeta(\overline{\Delta_2})$.
		Therefore we may find smooth $\psi_1, \psi_2\colon X\to [0,1]$ such that
	\begin{enumerate}
		\item $\psi_1(x) > 0$ and $(\supp{\psi_1}) \cap D\subset \zeta(\Delta_0)$,
		\item $\restr{\psi_2}{\zeta(\overline{\Delta_1})} \equiv 1$, and
		\item $(\supp{\psi_2}) \cap D \subset  \zeta(\Delta_2)$.
	\end{enumerate}
	
	Recall from Section \ref{sec:slices} that
	\begin{equation}\label{indic}
		\begin{aligned}
			\restr{T_\omega^+}{\eta_\omega(x)} &= \mathds{1}_{\eta_\omega(x)} \restr{T_\omega^+}{W^u_\omega(x)} = \mathds{1}_{\eta_\omega(x)} \restr{T_\omega^+}{D}.
		\end{aligned}
	\end{equation}
	Since, on $D$,
	\begin{align*}
			\psi_1 \leq \mathds{1}_{\eta_\omega(x)}  \leq \psi_2,
	\end{align*}
	we have
	\begin{align}\label{ineq2}
		\psi_1	\restr{T_\omega^+}{D} \leq \restr{T_\omega^+}{\eta_\omega(x)} \leq \psi_2	\restr{T_\omega^+}{D}
	\end{align}
	as measures. 
	
		Notice that $x\in \supp \restr{T_\omega^+}{\eta_\omega(x)}$ on a full-measure set. Indeed, $\restr{T_\omega^+}{\eta_\omega(x)}$ is proportional to $\nu_{\eta_\omega(x)}$, and we proved $x\in \supp \nu_{\eta_\omega(x)}$ almost surely in Step 1. Since $\psi_1(x) >0$, it follows that
	\begin{align*}
		\psi_1	\restr{T_\omega^+}{D} \neq 0.
	\end{align*}
	
	Applying Lemma \ref{lem:elementary}, 
	\begin{align*}
		\psi_i	\restr{T_\omega^+}{D}  = T_\omega^+ \wedge (\psi_i [D])
	\end{align*}
		for $i = 1,2$. Now we define 
	\begin{align*}
		\nu_N^i \coloneqq \frac 1 N \sum_{n=1}^N \delta_{\sigma^n(\omega)} \otimes (f_{\omega}^n)_*\cN(T_{\omega}^+\wedge (\psi_i[D])).
	\end{align*}
	By \eqref{indic} and \eqref{ineq2}, and accounting for the normalization in $\nu_N^i$, we obtain
	\begin{align*}
			&\left (\int_X T_\omega^+ \wedge (\psi_1 [D])\right )\nu_N^1 (\omega, x) \\
			&\hspace{3 cm} \leq \left(\int_X \restr{ T_\omega^+}{\eta_\omega(x)} \right)\nu_N  (\omega, x)  \\
			&\hspace{6 cm}\leq  \left (\int_X T_\omega^+ \wedge (\psi_2 [D])\right )\nu_N^2 (\omega, x).
	\end{align*}
	Hence there exists $K > 0$ satisfying 
	\begin{align}\label{shorter}
		K^{-1} \nu_N^1 (\omega, x) \leq \nu_N  (\omega, x) \leq K \nu_N^2 (\omega, x)
	\end{align}
	for all $N$. 
	
		By Remark \ref{rmk:prop31-independent}, the full-measure set of $(\omega,x)$ on which Proposition \ref{lem:last2} holds may be chosen independently of $C$ and $\psi$. Then by Proposition \ref{lem:last2}, for each $\rho > 0$ there exists $A_\rho\subset \N$ with upper-density greater than $1 - \rho$. Let
	\begin{align*}
		\nu_N^1(\rho)&\coloneqq  \frac 1 {N} \sum_{n=1}^{N} \mathds{1}_{A_\rho}(n)\left(\delta_{\sigma^n(\omega)} \otimes (f_{\omega}^n)_*\cN\bigl(T_{\omega}^+\wedge (\psi_1[D])\bigr)\right), \\
		&\hspace{1.5 cm} m_N(\rho)\coloneqq  \frac 1 {N} \sum_{n=1}^{N} \mathds{1}_{A_\rho}(n)\left(\delta_{\sigma^n(\omega)} \otimes (f_{\omega}^n)_*m_\omega\right),
		\end{align*}.
		We first show that Proposition \ref{lem:last2} implies 
	\begin{align}\label{thing1}
		m_{N}(\rho)- \nu_N^1(\rho) \to 0
	\end{align} weakly as $N\to \infty$. Let $h\in C_b(\Omega\times X)$. Choose compact sets $K_j\subset\Omega$ with $\mu^\Z(K_j)\to 1$, and shrink $\Lambda$ so that $\omega$ is Birkhoff-generic for each $K_j$. On each $K_j\times X$, the family $\set{h_\eta\colon \eta\in K_j}$
		is precompact in $C^0(X)$. By Proposition \ref{lem:last2} and Remark \ref{tbd2}, we have
	\begin{align*}
		\mathds{1}_{A_\rho}(n)\left(
		(f_\omega^n)_*m_\omega
		-
		(f_\omega^n)_*\cN\bigl(T_\omega^+\wedge(\psi_1[D])\bigr)
		\right)\to 0
	\end{align*}
	weakly. The signed measures in this sequence have total variation at most $2$. Therefore their convergence to zero may be tested uniformly on the precompact family $\set{h_\eta\colon \eta\in K_j}$. Then Birkhoff genericity gives
	\begin{align*}
		\limsup_{N\to\infty}\left|\int h \,d m_N(\rho)-\int h\,d\nu_N^1(\rho)\right|
		\leq 2\mu^\Z(K_j^c)\norm{h}_{C^0}.
	\end{align*}
	Letting $j\to\infty$ gives
	\begin{align*}
		\int h \,d m_N(\rho)-\int h\,d\nu_N^1(\rho) \to 0.
	\end{align*}
	Since this holds for every $h\in C_b(\Omega\times X)$, we have $m_{N}(\rho)- \nu_N^1(\rho) \to 0$ weakly.
	
	Now, since $A_\rho$ has upper-density greater than $1 - \rho$, there exist $N_k\to \infty$ such that for every $f\in C_b(\Omega \times X)$, 
	\begin{align}\label{thing2}
		\left |	\int f  \,d \nu_{N_k}^1 -\int f \,d \nu_{N_k}^1(\rho)\right|,  \left |	\int f  \,d m_{N_k} -\int f \,d m_{N_k}(\rho)\right|\leq \rho \norm{f}_{C^0}
	\end{align}
	for all $k\in \N$. Combining \eqref{thing1}, \eqref{thing2}, and the weak convergence $m_N\to m$, we obtain
	\begin{align*}
		\left|\int f \,d \nu_{N_k}^1  - \int f \,d m\right| \leq 3 \rho \norm{f}_{C^0}
	\end{align*} 
	for all $k$ large enough, where we emphasize that the sequence $N_k$ depends on $\rho$. Therefore, for every nonnegative $f\in C_b(\Omega\times X)$, by \eqref{shorter} we have 
	\begin{align*}
		K^{-1}\left(\int f \,d m -  3\rho \norm{f}_{C^0}\right) \leq \int f \,d\nu_{N_k}
	\end{align*}
	for all $k$ large enough. Since $\nu_{N_k}  \to \nu$ weakly, it follows that 
	\begin{align*}
			K^{-1}\left(\int f \,d m -  3\rho \norm{f}_{C^0}\right) \leq \int f \,d\nu.
	\end{align*}
	Taking $\rho\to 0$ yields $K^{-1}m  \leq  \nu.$
	
	An analogous argument with $\psi_2$ proves $\nu \leq K m$, hence
	\begin{align*}
			K^{-1}m  \leq  \nu \leq Km. 
	\end{align*}
	Since $m$ and $\nu$ are $F$-invariant probability measures and $\nu$ is ergodic, the Radon--Nikodym derivative of $m$ with respect to $\nu$ is $F$-invariant and therefore $\nu$-almost surely constant. Hence $m=\nu$.

\section{Exponential moment implies $m$ is mixing}
The main theorem of this section is the following.
\begin{thm}\label{thm:mixing}
	Under the hypotheses of Theorem \ref{thm:main}, suppose in addition that $\mu$ satisfies the exponential moment condition \eqref{exponential}. Then $m$ is mixing for $F$.
\end{thm}

Before beginning the proof, we show that the exponential moment allows some uniformity in the sequence coming from Proposition \ref{lem:last2}. Precisely, suppose we are given an assignment
\[
	\omega\mapsto \psi_\omega C_\omega,
\]
where each $C_\omega$ is a positive $(1,1)$-current, each $\psi_\omega\colon X\to\R_{\geq0}$ is smooth and satisfies
\[
	\supp(\psi_\omega)\cap\supp(\partial C_\omega)=\varnothing,
\]
and the assignment $\omega\mapsto \psi_\omega C_\omega$ is $\mu^\Z$-measurable as a map into currents endowed with the weak topology.
\begin{lem}\label{lem:uniformity}
	Suppose $\mu$ satisfies the exponential moment condition \eqref{exponential}. Let $\rho>0$ be given. For every $\beta>0$, there exists a subset $\Omega_\beta\subset\Omega$ with $\mu^\Z(\Omega_\beta)>1-\beta$ and, for every $\omega\in\Omega_\beta$, a sequence
	\begin{align*}
		G(\omega)\coloneqq \set{\tilde n_k(\omega)}_{k\in\N}
	\end{align*}
	such that:
	\begin{enumerate}
		\item each $G(\omega)$ satisfies the conclusions of Proposition \ref{lem:last2} for $\rho$ and $\psi_\omega C_\omega$
		\item for every $n\in\N$, the set
		\begin{align*}
			Z_n^\beta\coloneqq \set{\omega\in\Omega_\beta\colon n\in G(\omega)}
		\end{align*}
		is $\mu^\Z$-measurable with \begin{align*}
			\liminf_{n\to\infty}\mu^\Z(Z_n^\beta)\geq s(\rho),
		\end{align*}
		for some constant $s(\rho)$ depending on $\rho$ where $s(\rho)\to1$ as $\rho\to 0$.
	\end{enumerate}
	
\end{lem}
\begin{proof}[Proof of Lemma \ref{lem:uniformity}]
For each fixed generic $\omega$, the proof of Proposition 3.2 begins by choosing a good-time set $S(\omega)\subset \mathbb N$ supplied by Proposition \ref{lem:MassBound2}, for an auxiliary density parameter $\rho'$ chosen in terms of $\rho$. Under the exponential moment assumption, we may make this set explicit. Indeed, the sequence produced by Lemma \ref{lem:MassBound} is $\N$. Hence, in the proof of Proposition \ref{lem:MassBound2}, the only restriction comes from Claim \ref{claim:mass}. Thus, for a small $\epsilon_0>0$ to be chosen below in terms of $\rho'$, we may take
\begin{align}\label{sdef}
	S(\omega)\coloneqq \set{n\in\N\colon h_n(\omega)\geq \epsilon_0},
\end{align}
recalling the definition of $h_n(\omega)$ from the proof of Claim \ref{claim:mass}.

For each $n\in\N$, define
\begin{align}\label{Vndefn}
	V_n \coloneqq \set{\omega\in\Omega\colon n\in S(\omega)} = \set{\omega\in\Omega\colon h_n(\omega)\geq \epsilon_0}.
\end{align}
	By the convergence \eqref{positive} established in the proof of Claim \ref{claim:mass},
	\begin{align*}
		\limsup_{n\to\infty}\mu^\Z(V_n^c)\leq r(\epsilon_0),
	\end{align*}
where
\begin{align*}
	r(t)
	&\coloneqq \int\int
		\mathds{1}_{\set{|\langle w,v\rangle|\leq t}}
	\,d\tilde\mu_{\mathcal F}(w)\,d\mu_{\mathcal F}(v).
\end{align*}
Here $\mu_{\mathcal F}$ is the Furstenberg measure associated to the forward action, and $\tilde\mu_{\mathcal F}$ is the corresponding Furstenberg measure for $\tilde \mu(f^*)\coloneqq \mu((f^*)^t)$. Since these measures give no mass to proper projective subspaces, $r(t)\to 0$ as $t\to 0$.

Choose auxiliary $\rho'>0$ small enough that the lower-density part of the proof of Proposition \ref{lem:last2} gives the conclusion with parameter $\rho$ whenever its initial set $S$ satisfies $\underline d(S)>1-\rho'$. Then choose $\epsilon_0=\epsilon_0(\rho)>0$, tending to $0$ as $\rho\to0$, such that $r(2\epsilon_0)<\rho'$. For almost every $\omega$, by the estimate \eqref{almost} in the proof of Claim \ref{claim:mass} we have
\begin{align*}
	\underline d(S(\omega))\geq 1-r(2\epsilon_0)>1-\rho'
\end{align*}
where $S(\omega)$ is defined by \eqref{sdef}. This choice of $\epsilon_0$ is independent of $\omega$. Set
\begin{align*}
	s'(\rho) = 1-2r(\epsilon_0).
\end{align*} 
As $\rho\to0$, we have $\epsilon_0(\rho)\to0$ and thus $s'(\rho)\to1$.
Choose $s(\rho)<s'(\rho)$ with $s(\rho)\to1$ as $\rho\to0$; this will eventually be our $s(\rho)$ in the conclusion of the lemma.

By Proposition 5.14 of \cite{MR4635340}, applied with $D\equiv1$, there is a finite $\mu^\Z$-measurable function $b(\omega)$ such that, for every $n\geq1$,
\begin{align}\label{estimate}
	\sum_{m=1}^{n-1}
	\frac{\norm{(f_\omega^{n-m})^*}}
	{\norm{(f_\omega^n)^*}}
	\leq b(\omega).
\end{align}

Let $\beta>0$ be given. Choose $0<\alpha<1$ so small that
\begin{align*}
	\frac{2\alpha}{1-\alpha}<\beta
	\quad\text{and}\quad
	\frac{2\alpha}{1-\alpha}<s'(\rho)-s(\rho).
\end{align*}
Choose $B_0<\infty$ so that
\begin{align*}
	\mu^\Z(\set{\omega\in\Omega\colon b(\omega)>B_0})
	<\frac{\alpha}{1-\alpha}.
\end{align*}
Fix $0<\epsilon<\lambda_\mu$. Recalling our notation from Claim \ref{claim:2}, choose numbers $\gamma_i\to0$ and finite intervals $I_i\subset\N$ such that $m_i\coloneqq\min I_i\to\infty$ and
\begin{align*}
	\frac{B_0}{\gamma_i|I_i|}\to0.
\end{align*}
Taking the intervals far enough out, we may assume that, for every $m\geq m_i$,
\begin{align}\label{partic}
	\mu^\Z(\mathcal L_m)>1-\frac{1}{2i},\qquad
	\mu^\Z(\mathcal R_m)>1-\frac{1}{2i},\qquad
	\mu^\Z(\mathcal H_{m,i})>1-\frac{1}{2i}.
\end{align}
Set $M_i\coloneqq \max I_i$. In the exponential moment case of Proposition \ref{lem:last2}, after setting $c_0=0$, the integers $c_i$ were chosen so that the interval analogs of  \eqref{claim2:separation}--\eqref{claim2:H-average} and the additional condition \eqref{dense:new} hold for fixed $\omega$ and the intervals $I_i$ chosen above. The same argument allows us to choose $(c_i)_{i\in \N_0}$ that works uniformly on a large set. Precisely, there exists $(c_i)_{i\in \N_0}$ and sets $P_i\subset\Omega$ with $\mu^\Z(P_i)>1-\alpha^i$ such that, for every $\omega\in P_i$,
the interval analogs of \eqref{claim2:separation}--\eqref{claim2:H-average}, as written in the lower-density section of the proof of Proposition \ref{lem:last2}, and the additional condition \eqref{dense:new} hold. This follows from the same limiting estimates
because the relevant quantities are $\mu^\Z$-measurable functions of $\omega$.

Set
\begin{align*}
	\Omega_\alpha\coloneqq  \bigg( \bigcap_{i\geq1}P_i\bigg)\cap \set{\omega\in\Omega\colon b(\omega)\leq B_0}.
\end{align*}
Then
\begin{align*}
	\mu^\Z(\Omega_\alpha)\geq 1 - \sum_{i\geq 1} \alpha^i - \frac{\alpha}{1-\alpha} =  1-\frac{2\alpha}{1-\alpha}>1-\beta.
\end{align*}
For $m\in I_i$, set
\begin{align*}
	Q_{i,m}\coloneqq \mathcal L_m\cap\mathcal R_m\cap\mathcal H_{m,i}.
\end{align*}
By the choice of the intervals $I_i$ above, in particular \eqref{partic}, $\mu^\Z(Q_{i,m}^c)<\frac{3}{2i}$ for every $m\in I_i$.
For $\omega\in\Omega_\alpha$ and each $i\in \N$, define
\begin{align*}
	E_i(\omega)
	\coloneqq
	\set{n\in S(\omega)\colon \text{there exists }m\in I_i\text{ admissible for }n},
\end{align*}
where $m\in I_i$ is admissible for $n$ if
\begin{align*}
	n-m\in S(\omega),\qquad
	\sigma^n(\omega)\in Q_{i,m},\qquad
	\frac{\norm{(f_\omega^{n-m})^*}}
	{\norm{(f_\omega^n)^*}}
	\leq \gamma_i.
\end{align*}
As in the proof of Proposition \ref{lem:last2}, let
\begin{align*}
	A(\omega)\coloneqq \bigcup_{i\geq1}\left(E_i(\omega)\cap(c_i,c_{i+1}]\right).
\end{align*}
The condition $c_i-M_i\to\infty$ will only be used below to ensure that $n-m\to\infty$ whenever $c_i<n\leq c_{i+1}$ and $m\in I_i$.

If $c_i<n\leq c_{i+1}$, set
\begin{align*}
	U_{n,m}\coloneqq
	\set{\omega\in\Omega\colon
	\frac{\norm{(f_\omega^{n-m})^*}}
	{\norm{(f_\omega^n)^*}}
	\leq \gamma_i}.
\end{align*}
Then, recalling the definition of $V_k$ in \eqref{Vndefn}, $\set{\omega\in\Omega_\alpha\colon n\in A(\omega)}$ is the finite union over $m\in I_i$ of the sets
\begin{align*}
	\Omega_\alpha
	\cap V_n
	\cap V_{n-m}
	\cap \sigma^{-n}(Q_{i,m})
	\cap U_{n,m}.
\end{align*}
Thus this set is $\mu^\Z$-measurable. Since the union contains each of the sets displayed above, its measure is at least the average of their measures over $m\in I_i$. Moreover, as $c_i>M_i$, every $m\in I_i$ satisfies $m<n$. For each fixed $\omega\in\Omega_\alpha$, the estimate \eqref{estimate} gives at most $B_0/\gamma_i$ values of $m\in I_i$ for which $\omega\notin U_{n,m}$. Therefore
\begin{align*}
	\frac1{|I_i|}\sum_{m\in I_i}\mu^\Z(\Omega_\alpha\cap U_{n,m}^c)
	\leq
	\frac{B_0}{\gamma_i|I_i|}.
\end{align*}
Hence, using shift-invariance of $\mu^\Z$,
\begin{align*}
	& \mu^\Z(\set{\omega\in\Omega_\alpha\colon n\in A(\omega)})
	\\
	& \hspace{1.5 cm}\geq
	1-\mu^\Z(\Omega_\alpha^c)
	-\mu^\Z(V_n^c)  -\frac1{|I_i|}\sum_{m\in I_i}\mu^\Z(V_{n-m}^c)
	-\frac1{|I_i|}\sum_{m\in I_i}\mu^\Z(Q_{i,m}^c)
	-\frac{B_0}{\gamma_i|I_i|} \\
		&\hspace{1.5 cm} \geq
	1-\frac{2\alpha}{1-\alpha}
	-\mu^\Z(V_n^c) -\frac1{|I_i|}\sum_{m\in I_i}\mu^\Z(V_{n-m}^c) 
	- \frac{3}{2i}
	-\frac{B_0}{\gamma_i|I_i|}.
\end{align*}
When $i=i(n)$ is determined by $c_i<n\leq c_{i+1}$, we have $i(n)\to\infty$ as $n\to\infty$. By the choice above, $n-m\to\infty$ uniformly for $m\in I_i$. Since $\limsup_{k\to\infty}\mu^\Z(V_k^c)\leq r(\epsilon_0)$, it follows that
\begin{align}\label{uniformity:A-lower}
	\liminf_{n\to\infty}\mu^\Z\set{\omega\in\Omega_\alpha\colon n\in A(\omega)}
	&\geq 1 - 2 r(\epsilon_0) -\frac{2\alpha}{1-\alpha} =  s'(\rho)-\frac{2\alpha}{1-\alpha}
\end{align}

Choose $\delta>0$ so that
\[
	\frac{2\alpha}{1-\alpha}+\delta<s'(\rho)-s(\rho).
\]
Then choose the sets $K$, $K'$, and $D_\tau$ appearing at the end of the proof of Proposition \ref{lem:last2} so that their total measure loss is less than $\delta$ and small enough to preserve the lower-density conclusion there.
Let $\Lambda$ be the full-measure set on which $S(\omega)$ has the lower-density required above and on which $\omega$ is Birkhoff-generic for $K$, $K'$, and $D_\tau$. Finally set
\begin{align*}
	\Omega_\beta\coloneqq \Omega_\alpha\cap\Lambda.
\end{align*}
Then $\mu^\Z(\Omega_\beta)>1-\beta$. For $\omega\in\Omega_\beta$, define
\begin{align*}
	G(\omega)
	\coloneqq
	A(\omega)\cap\set{n\in\N\colon \sigma^n(\omega)\in K\cap K'\cap D_\tau}.
\end{align*}
Then, by the proof of Proposition \ref{lem:last2}, $G(\omega)$ satisfies the conclusion of Proposition \ref{lem:last2}. Moreover, for each fixed $n$, the set
\begin{align*}
	Z_n^\beta\coloneqq \set{\omega\in\Omega_\beta\colon n\in G(\omega)}
\end{align*}
is $\mu^\Z$-measurable. Indeed, if $n\leq c_1$, then $Z_n^\beta=\varnothing$. If $c_i<n\leq c_{i+1}$, then
		\begin{align*}
			Z_n^\beta
			=
			\Omega_\beta\cap\sigma^{-n}(K\cap K'\cap D_\tau) \cap
			\bigcup_{m\in I_i}
			\left(
			V_n\cap V_{n-m}\cap\sigma^{-n}(Q_{i,m})\cap U_{n,m}
			\right),
		\end{align*}
which is $\mu^\Z$-measurable. With these choices, \eqref{uniformity:A-lower} gives
\begin{align*}
	\liminf_{n\to\infty}\mu^\Z(Z_n^\beta)
		&\geq s'(\rho)-\frac{2\alpha}{1-\alpha}-\delta > s(\rho).
\end{align*}
This proves the claimed lower bound.
	
\end{proof}

\begin{proof}[Proof of Theorem \ref{thm:mixing}]
	Since $m$ is a Borel probability measure on the metrizable space $\Omega\times X$, bounded continuous functions are dense in $L^2(m)$. Moreover, bounded continuous functions whose restrictions to the fibers $\set{\omega}\times X$ are smooth are dense in $L^2(m)$. This follows by approximating on compact subsets of $\Omega$ with arbitrarily large $\mu^\Z$-measure and smoothing in the $X$-variable. Thus it suffices to show that 
	\begin{align*}
		 \int \psi \cdot (\phi \circ F^n)  \,dm - \int \psi  \, dm  \int \phi  \, dm \to 0
	\end{align*}
	for bounded, continuous functions $\psi, \phi\colon \Omega \times X \to \R$ whose fiber restrictions $\psi_\omega,\phi_\omega$ are smooth. Finally, it is enough to treat the case $\phi_\omega,\psi_\omega \geq 0$ for every $\omega \in \Omega$.

	Write 
	\begin{align*}
		\beta(\omega)\coloneqq \int_X T_\omega^+ \wedge T_\omega^-,
	\end{align*}  
	which is positive for $\mu^\Z$-almost every $\omega$ by Remark \ref{tbd}.
	Using the definition of $m$ and the identity 
	\[
	(\phi \circ  F^n)_\omega = \phi_{\sigma^n(\omega)} \circ f_\omega^n
	\] we obtain
	\begin{align*}
		\int\psi \cdot (\phi \circ F^n)  \,dm &=    \int\psi_\omega \cdot (\phi \circ F^n)_\omega  \frac{1}{\beta(\omega)} d(T_\omega^+\wedge T_\omega^-) \,d\mu^\Z(\omega) \\
		& =     \int(\phi_{\sigma^n(\omega)} \circ f_\omega^n) \frac{1}{\beta(\omega)}  d ( T_\omega^+\wedge \psi_\omega  T_\omega^-) \,d\mu^\Z(\omega) \\
		& =    \int_\Omega \frac{1}{\beta(\omega)} \int_X \phi_{\sigma^n(\omega)} d ((f_\omega^n)_* ( T_\omega^+\wedge \psi_\omega  T_\omega^-)) \,d\mu^\Z (\omega ).
	\end{align*}

	Fix $\rho>0$ and choose $\gamma>0$. Since $T_\omega^-$ is closed and depends $\mu^\Z$-measurably on $\omega$, the assignment $\omega\mapsto \psi_\omega T_\omega^-$ satisfies the hypotheses of Lemma \ref{lem:uniformity}. Applying the lemma with $\gamma$ in place of $\beta$, we obtain a set $\Omega_\gamma$ and, for each $\omega\in\Omega_\gamma$, a set
	\[
		G(\omega)\coloneqq \set{\tilde n_k(\omega)}_{k\in\N}.
	\]
	Set $G(\omega)=\varnothing$ for $\omega\notin\Omega_\gamma$, and define
	\[
		Z_n\coloneqq \set{\omega\in\Omega\colon n\in G(\omega)}.
	\]
	Then $Z_n$ is the set $Z_n^\gamma$ from Lemma \ref{lem:uniformity}. Hence, for some $s(\rho)\to1$ as $\rho\to0$,
	\begin{align}\label{mixing:uniformity}
		\liminf_{n\to\infty}\mu^\Z(Z_n)\geq s(\rho).
	\end{align}
	
	Fix $\delta>0$. By regularity, choose a compact set $K_\phi\subset\Omega$ with $\mu^\Z(K_\phi)>1-\delta$. Since $\phi$ is continuous on $\Omega\times X$, the family
	\[
		\set{\phi_\eta\colon \eta\in K_\phi}
	\]
	is equicontinuous and uniformly bounded. Define
	\[
		G_\delta(\omega)\coloneqq G(\omega)\cap\set{n\in\N\colon \sigma^n(\omega)\in K_\phi}
	\]
	and
	\[
		Z_{n,\delta}\coloneqq \set{\omega\in\Omega\colon n\in G_\delta(\omega)}
		=Z_n\cap \sigma^{-n}(K_\phi).
	\]
	By shift-invariance of $\mu^\Z$ and \eqref{mixing:uniformity},
	\begin{align*}
		\liminf_{n\to\infty}\mu^\Z(Z_{n,\delta})
		\geq s(\rho)-\delta.
	\end{align*}
	
	By Remark \ref{tbd2}, for $\mu^\Z$-almost every $\omega$ the signed measures
	\begin{align}\label{diff}
		\mathds{1}_{G(\omega)}(n)
		\bigg(
		(f_\omega^n)_*(T_\omega^+\wedge \psi_\omega T_\omega^-)
		-
		\frac{\int_X T_\omega^+\wedge \psi_\omega T_\omega^-}{\beta(\sigma^n(\omega))}
		T_{\sigma^n(\omega)}^+\wedge T_{\sigma^n(\omega)}^-
		\bigg)
		\to 0
	\end{align}
	weakly as $n\to\infty$. Since $G_\delta(\omega)\subset G(\omega)$, the same convergence holds with $G_\delta(\omega)$ in place of $G(\omega)$. For fixed $\omega$, the signed measures in \eqref{diff} have uniformly bounded mass. Along $G_\delta(\omega)$ the test functions $\phi_{\sigma^n(\omega)}$ belong to the equicontinuous family above, so \eqref{diff} gives
	\begin{align}\label{close2}
		\mathds{1}_{G_\delta(\omega)}(n)|a_n(\omega)-b_n(\omega)|\to0
	\end{align}
	for $\mu^\Z$-almost every $\omega$, where
	\begin{align*}
		a_n(\omega)
		&\coloneqq
		\frac{1}{\beta(\omega)}
		\int_X \phi_{\sigma^n(\omega)}\, d((f_\omega^n)_*(T_\omega^+\wedge \psi_\omega T_\omega^-)), \\
		b_n(\omega)
		&\coloneqq
		\frac{1}{\beta(\omega)}
		\frac{\int_X T_\omega^+\wedge \psi_\omega T_\omega^-}{\beta(\sigma^n(\omega))}
		\int_X \phi_{\sigma^n(\omega)}\, d(T_{\sigma^n(\omega)}^+\wedge T_{\sigma^n(\omega)}^-).
	\end{align*}
	Moreover,
	\begin{align*}
		|a_n(\omega)|,\ |b_n(\omega)|
		\leq \norm{\phi}_{C^0}\norm{\psi}_{C^0}.
	\end{align*}
	Thus, by dominated convergence and \eqref{close2},
	\begin{align*}
		\int_{Z_{n,\delta}} |a_n(\omega)-b_n(\omega)|\,d\mu^\Z(\omega)
		\to0.
	\end{align*}
	It follows that
	\begin{align*}
		\limsup_{n\to\infty}
		\left|
		\int_\Omega a_n(\omega)\,d\mu^\Z(\omega)
		-
		\int_\Omega b_n(\omega)\,d\mu^\Z(\omega)
		\right|
		&\leq
		2\norm{\phi}_{C^0}\norm{\psi}_{C^0}
		\limsup_{n\to\infty}\mu^\Z(Z_{n,\delta}^c) \\
		&\leq
		2\norm{\phi}_{C^0}\norm{\psi}_{C^0}(1-s(\rho)+\delta).
	\end{align*}
	By the computation at the start of the proof,
	\begin{align*}
		\int_\Omega a_n(\omega)\,d\mu^\Z(\omega)
		=
		\int\psi\cdot(\phi\circ F^n)\,dm.
	\end{align*}
	
	We now identify the limit of the $b_n$ terms. Set
	\begin{align*}
		u(\omega)&\coloneqq \frac{1}{\beta(\omega)}\int_X T_\omega^+\wedge \psi_\omega T_\omega^-, \hspace{1cm} v(\omega)\coloneqq \frac{1}{\beta(\omega)}\int_X \phi_\omega\,d(T_\omega^+\wedge T_\omega^-).
	\end{align*}
	Then $u$ and $v$ belong to $L^2(\mu^\Z)$ and
	\begin{align*}
		\int_\Omega b_n(\omega)\,d\mu^\Z(\omega)
		=
		\int_\Omega u(\omega)v(\sigma^n(\omega))\,d\mu^\Z(\omega).
	\end{align*}
	Since $(\Omega,\mu^\Z,\sigma)$ is mixing,
	\begin{align*}
		\int_\Omega b_n(\omega)\,d\mu^\Z(\omega)
		\to
		\int_\Omega u\,d\mu^\Z \int_\Omega v\,d\mu^\Z.
	\end{align*}
	By the definition of $m$,
	\begin{align*}
		\int_\Omega u\,d\mu^\Z=\int \psi\,dm
		\hspace{.5 cm}\text{and}\hspace{.5 cm}
		\int_\Omega v\,d\mu^\Z=\int \phi\,dm.
	\end{align*}
	Therefore
	\begin{align*}
		\int_\Omega b_n(\omega)\,d\mu^\Z(\omega)
		\to
		\int \psi\,dm\int \phi\,dm.
	\end{align*}
	Combining the last two estimates gives
	\begin{align*}
		\limsup_{n\to\infty}
		\left|
		\int \psi\cdot(\phi\circ F^n)\,dm
		-
		\int \psi\,dm\int \phi\,dm
		\right|
		\leq
		2\norm{\phi}_{C^0}\norm{\psi}_{C^0}(1-s(\rho)+\delta).
	\end{align*}
	Letting $\delta\to0$ and then $\rho\to0$, and using $s(\rho)\to1$, proves the theorem.
		
\end{proof}

\begin{appendices}
	\section{Random topological entropy and the variational principle}\label{app:entropy}
	
	In this section, we define random topological entropy and state the random variational principle. Fix a K\"ahler surface $X$ and a K\"ahler form $\kappa$, and let $d$ denote the induced metric on $X$. 
	\begin{defn}[Separated set]
		Let $\epsilon > 0$ and $n\in \N$. We say that a subset $S\subset X$ is $(\omega, n,\epsilon)$-separated if, for all distinct $x,y\in S$, $d(f_\omega^i (x), f_\omega^i (y)) > \epsilon$ for some $0\leq i < n$. 
	\end{defn}
	By compactness, every $(\omega,n,\epsilon)$-separated set is finite, and for each fixed $(\omega,n,\epsilon)$ there is a uniform bound on the cardinalities of such sets. Let $r^\omega(n,\epsilon)$ denote the maximum size of an $(\omega, n, \epsilon)$-separated set. 
	\begin{defn}[Random topological entropy]
		We define
		\[
			h_{top}(\mu) \coloneqq \lim_{\epsilon \to 0} \limsup_{n\to\infty} \frac{1}{n}\int_\Omega \ln r^\omega(n,\epsilon)\,d\mu^\Z(\omega).
		\]
		We call $h_{top}(\mu)$ the random topological entropy of $\mu$.
	\end{defn}
		\begin{rmk}\label{rmk:appendix}
				The topological entropy $h_{top}(\mu)$ has several equivalent definitions, and we refer to \cite{MR884892} for details. We will use that for $\mu^\N$-generic $\omega$ \cite[Theorem~II.2.2]{MR1819189},
		\[
		h_{top}(\mu)=\lim_{\epsilon\to 0}\limsup_{n\to\infty}\frac{1}{n}\ln r^\omega(n,\epsilon).
		\]
	\end{rmk}

	Just like in the deterministic case, fiber entropy and random topological entropy satisfy a variational principle. In the case that $\mu$ has compact support, the variational principle is due to Ledrappier and Walters \cite{MR476995}; the general case is due to Bogensch\"utz \cite{MR1181382} and Kifer \cite{MR1819189}.
	
	\begin{thm}[Variational Principle]
		\begin{align*}
			\sup_{\nu\in S(\mu)}h_{\nu}^\mathcal{F} =  h_{top}(\mu)
		\end{align*}
	and the supremum is achieved. 
	\end{thm}

	\section{Randomized Gromov's theorem}\label{app:gromov}
	We turn our attention to proving that $h_{top}(\mu) = \lambda_\mu$. The direction $\lambda_\mu \leq h_{top}(\mu)$ is due to Yomdin and Kifer \cite{MR970566}; they prove the analog of Yomdin's theorem in our setting. In a sentence, they show that, in the $C^\infty$ category, exponential volume growth bounds topological entropy from below. Trivially, growth in cohomology bounds volume growth from below, so the inequality follows. 
	
	The direction $h_{top}(\mu) \leq \lambda_\mu$ is due to Gromov \cite{MR2026895}, once we restate his theorem using the random notation.
	\begin{thm}[Randomized Gromov]
		Let $X$ be a K\"ahler surface and $\mu$ a probability measure supported on $\Aut(X)$ satisfying the moment condition \eqref{int:cohomology}. Then $h_{top}(\mu) \leq \lambda_\mu$.
	\end{thm}
	We recall the proof below, following closely the one presented in \cite{Filip2019_K3_Dynamics}.
	
	\begin{proof}[Proof of Randomized Gromov]
		Fix a K\"ahler form $\kappa$ and let $d$ denote the induced metric. Given $\omega\in \Omega_+$, define
		\begin{align*}
			\cF_\omega^n(x)\coloneqq (x, f_\omega^1(x), f_\omega^2(x), \dots, f_\omega^{n-1}(x)) \in X^n
		\end{align*}
	and let $\Delta_\omega^n \coloneqq \cF_\omega^n(X)$. The key idea in Gromov's theorem is to bound $r^\omega(n, \epsilon)$ above by the volume of $\Delta_\omega^n$ in $X^n$ times a constant independent of $n$. This will follow from the monotonicity formula for minimal surfaces and Federer's theorem. Then one shows that, since $\Delta_\omega^n$ is holomorphic, its volume grows exponentially at rate $\lambda_\mu$.

	Let $\pi_i\colon X^n \to X$ denote projection to the $i$th factor, for $0\leq i<n$, and set $\kappa_i\coloneqq \pi_i^* \kappa$. Then
	\[
		\widehat\kappa_n \coloneqq \sum_{i=0}^{n-1}\kappa_i
	\]
	is a K\"ahler form on $X^n$. Let $d_n$ be the distance induced by $\widehat\kappa_n$. All volumes and lengths in $X^n$ will be with respect to $d_n$.

	Let $S$ be an $(\omega, n, \epsilon)$-separated set. Then, for all distinct $x,y\in \cF_\omega^n(S)$, $d_n(x, y) > \epsilon$. Therefore,
		\begin{align}\label{vol}
		\Vol(\Delta_\omega^n) \geq \sum_{x\in \cF_\omega^n(S)}\Vol(\Delta_\omega^n\cap B_{\epsilon/2}(x)).
	\end{align}

		Given a compact complex manifold $Y$ with K\"ahler form $\alpha$ and a $q$-dimensional complex submanifold $W$, for $\epsilon > 0$ let 
	\begin{align*}
			\Dens_\epsilon(W, x) \coloneqq \frac{1}{\epsilon^{2q}}\int_{W\cap B_\epsilon(x)} \alpha^q = \frac{1}{\epsilon^{2q}} \Vol (W\cap B_\epsilon(x)),
	\end{align*}
	where by volume we mean the $2q$-dimensional volume induced by $\alpha$, and balls are with respect to the K\"ahler metric.
		\begin{thm}[Monotonicity Formula and Federer's Theorem]
			Suppose the sectional curvature of $Y$ has absolute value uniformly bounded above by $K$. There exists $\epsilon(K) > 0$ and $C>0$, depending only on $q$, $K$, and a lower injectivity-radius bound for $Y$, such that, for all $0<\epsilon<\epsilon(K)$, $\Dens_\epsilon (W, x) \geq C$ for all $x\in W$. 
		\end{thm}
			\noindent In the Euclidean case this follows from the Monotonicity Formula and Federer's theorem, and the bound $K$ limits distortion in charts. 
		
			Now notice that if $X$ has sectional curvature bounded by $K$, then $X^n$ has sectional curvature bounded by $K$ and injectivity radius bounded below uniformly in $n$. Therefore, there exists $\epsilon_0 > 0$ such that for all $0 < \epsilon < \epsilon_0$ and $n\geq 1$, $\Dens_{\epsilon/2} (\Delta_\omega^n, x) > C$ for all $x\in \Delta_\omega^n$. Then, for any $\epsilon < \epsilon_0$ and $n$, applying \eqref{vol} we get
	\begin{align*}
			\Vol(\Delta_\omega^n) \geq \sum_{x\in \cF_\omega^n(S)}\Vol(\Delta_\omega^n\cap B_{\epsilon/2}(x)) \geq \left(\frac{\epsilon}{2}\right)^4 C |S|.
		\end{align*}
		Letting $S$ be a maximum-cardinality $(\omega, n, \epsilon)$-separated set, by Remark \ref{rmk:appendix} we find
	\begin{align*}
			h_{top}(\mu) \leq \limsup_{n\to\infty} \frac 1 n \ln \Vol(\Delta_\omega^n).
	\end{align*}
	Now,
	\begin{align*}
		\Vol(\Delta_\omega^n) = \int_{\Delta_\omega^n} \widehat\kappa_n^2 = \int_{X} ((\cF_\omega^n)^*\widehat\kappa_n)^2 = \int_{X} \left( \sum_{i=0}^{n-1}(f_\omega^i)^* \kappa \right)^2.
	\end{align*}

	Let $\omega$ be Furstenberg-generic. On one hand, 
	\begin{align*}
			 \limsup_{n\to\infty} \frac 1 n \ln\int_{X} \left( \sum_{i=0}^{n-1}(f_\omega^i)^* \kappa \right)^2  \geq  \limsup_{n\to\infty} \frac 1 n \ln \int_X \kappa \wedge (f_\omega^{n-1})^* \kappa = \lambda_\mu.
	\end{align*} On the other,
	\begin{align*}
		\limsup_{n\to\infty} \frac 1 n \ln\int_{X} \left( \sum_{i=0}^{n-1}(f_\omega^i)^* \kappa \right)^2 &= \limsup_{n\to\infty} \frac 1 n \ln \sum_{0\leq i,j < n}\int_X (f_{\omega}^i)^* \kappa \wedge (f_{\omega}^j)^* \kappa \\
		&\leq \limsup_{n\to\infty} \frac 1 n \ln \left(2C_\kappa\sum_{0\leq i \leq j < n} \norm{(f_{\sigma^i(\omega)}^{j-i})^*}\right) \\
			&\leq \limsup_{n\to\infty} \frac 1 n \ln \left(2C_\kappa n^2\max_{0\leq i \leq j < n} \norm{(f_{\sigma^i(\omega)}^{j-i})^*}\right) \\
			&= \limsup_{n\to\infty} \frac 1 n \ln \max_{0\leq i \leq j < n} \norm{(f_{\sigma^i(\omega)}^{j-i})^*},
		\end{align*}
where $C_\kappa>0$ is independent of $i$, $j$, and $n$. Let
\[
E_L \coloneqq \frac 1 L \int_\Omega \ln \norm{(f_\eta^L)^*}\,d\mu^\Z(\eta).
\]
By Furstenberg's theorem and Fekete's lemma,
	\begin{align*}
		\lim_{L\to \infty}E_L= \lambda_\mu;
	\end{align*} fix $\epsilon > 0$ and choose $L$ large enough that $E_L< \lambda_\mu + \epsilon$.
Now, for each $i,j$ with $0\leq i \leq j < n$, choose integers $m,q\geq 0$ and $0\leq r,s<L$ satisfying
\begin{align*}
	j-i=Lm+r  \hspace{.5 cm } \text{and} \hspace{.5 cm } i=Lq+s. 
\end{align*}
Subadditivity gives
\begin{align*}
	\ln \norm{(f_{\sigma^i(\omega)}^{j-i})^*}
	&\leq \ln \norm{(f_{\sigma^i(\omega)}^{Lm})^*}
	+ \ln \norm{(f_{\sigma^{i+Lm}(\omega)}^r)^*},
\end{align*}
and
\begin{align*}
	\ln \norm{(f_{\sigma^i(\omega)}^{Lm})^*}
	&\leq \sum_{k=0}^{m-1}\ln\norm{(f_{\sigma^{i+Lk}(\omega)}^L)^*} \leq\sum_{k=0}^{q+m-1}\ln\norm{(f_{\sigma^{s+Lk}(\omega)}^L)^*}.
\end{align*}
By the Birkhoff ergodic theorem applied to $\sigma^L$, for generic $\omega$ there exists a nondecreasing function $\alpha\colon\N_0 \to\R_{\geq0}$ satisfying $\lim \frac 1 N \alpha(N)=0$ and
\begin{align*}
	\sum_{k=0}^{N-1}\ln\norm{(f_{\sigma^{s+Lk}(\omega)}^L)^*}
	\leq NLE_L+\alpha(N)
\end{align*}
for every $N\in\N$ and $0\leq s<L$. Since
\begin{align*}
	L(q+m)=i-s+Lm=j-s-r<n,
\end{align*}
we obtain, uniformly in $0\leq i\leq j<n$,
\begin{align*}
	\frac1n\ln \norm{(f_{\sigma^i(\omega)}^{Lm})^*}
	\leq E_L+\frac{\alpha(n)}{n}.
\end{align*}

Choose $\delta>0$. By Borel--Cantelli and the integrability condition \eqref{int:cohomology}, there exists $C>0$ such that
\begin{align*}
	\ln\norm{(f_{\sigma^k(\omega)}^1)^*}\leq k\delta+C
\end{align*}
for every $k\geq 0$. Hence, since $r<L$ and $i+Lm+r=j<n$,
\begin{align*}
	\ln \norm{(f_{\sigma^{i+Lm}(\omega)}^r)^*}
	&\leq\sum_{\ell=0}^{r-1}\ln\norm{(f_{\sigma^{i+Lm+\ell}(\omega)}^1)^*} \\
	&\leq L(n\delta+C).
\end{align*}
It follows that
\begin{align*}
	\limsup_{n\to\infty}\frac1n\ln\max_{0\leq i\leq j<n}
	\norm{(f_{\sigma^i(\omega)}^{j-i})^*}
	\leq E_L+L\delta
	<\lambda_\mu+\epsilon+L\delta.
\end{align*}
Letting $\delta\to0$ and then $\epsilon\to0$ proves the upper bound.
	\end{proof}
	\section{Product structure of $m$ in Pesin boxes}\label{app:Product}
	In this section, we outline a proof that $m$ has product structure in Pesin boxes. For us, a Pesin box has the following definition.
	\begin{defn}[Fibered Pesin box]\label{defn:PesinBox}
		A \emph{fibered Pesin box} (see \cite{MR1207478}, \cite{MR2219241}) over $\omega\in \Omega$ for a hyperbolic measure $\nu\in S(\mu)$ consists of a compact subset $Q_\omega \subset X$ with positive $\nu_\omega$-measure, an open neighborhood $N(Q_\omega)$ of $Q_\omega$ that is biholomorphic to $\mathbb{D}\times\mathbb{D}$, and transverse compact laminations $\mathcal L_\omega^u$ and $\mathcal L_\omega^s$ of $N(Q_\omega)$ whose leaves are horizontal and vertical graphs, respectively, under the identification of $N(Q_\omega)$ with $\mathbb{D}\times\mathbb{D}$. The defining property of a Pesin box is that $\mathcal L_\omega^s, \mathcal L_\omega^u$ are laminations by local stable/unstable manifolds, and $Q_\omega$ has product structure with respect to these laminations. This will be made precise in a moment.
		
		Let $K^-_\omega$ denote the intersection of $\mathcal L^u_\omega$ with the vertical fiber $\set{0}\times \mathbb{D}$ in the coordinates above, and let $K^+_\omega$ denote the intersection of $\mathcal L^s_\omega$ with the horizontal fiber $\mathbb{D}\times \set{0}$. Let $L_\omega^u(x,y)$ (respectively $L_\omega^s(x,y)$) denote the leaf of $\mathcal L_\omega^u$ (respectively $\mathcal L_\omega^s$) through $(x,y)\in\mathbb{D}\times\mathbb{D}$. By definition, a Pesin box satisfies the following:
		\begin{enumerate}
			\item For every $x\in K^+_\omega$ and $y\in K^-_\omega$, the leaf $L_\omega^u(0,y)$ intersects the leaf $L_\omega^s(x,0)$ in a unique point $h_\omega(x,y)\in K_\omega$. In this way, one obtains a homeomorphism
			\[
			h_\omega\colon K^+_\omega\times K^-_\omega\to K_\omega,
			\]
			where $K_\omega\subset N(Q_\omega)$ denotes the intersection of the supports of $\mathcal L_\omega^u$ and $\mathcal L_\omega^s$. 
			\item For $\nu_\omega$-almost every $(x,y)\in Q_\omega$, the leaf $L_\omega^u(x,y)$ (resp. $L_\omega^s(x,y)$) is contained in $W_\omega^{u}(x,y)$ (resp. $W_\omega^{s}(x,y)$).
		\end{enumerate}
	\end{defn}

By the theory of non-uniform hyperbolicity in our setting, in particular Lyapunov charts and the local stable/unstable manifold theorem \cite[Proposition~3.3]{MR2032494} or \cite[Theorem~6.4]{MR3671937}, for $\mu^\Z$-almost every $\omega$ there exists a countable collection of Pesin boxes over $\omega$ whose union has full $\nu_\omega$-measure. This holds in particular for our measure of maximal entropy $m$. 

	For ease of notation, let $\omega$ be as above and let $T^\pm \coloneqq T_\omega^\pm$. Let $\set{(Q_i,\mathcal L^s_i, \mathcal L^u_i)}_{i\in \N}$ be a collection of Pesin boxes over $\omega$ such that $\cup_i Q_i$ has full $m_\omega$-measure. Furthermore, take the stable (resp. unstable) Pesin laminations in this collection to be pairwise disjoint. We now outline an argument that $m_\omega$ has local product structure in the $Q_i$, meaning that $m_\omega$ is a countable sum of product measures on the $Q_i$'s. The key inputs are the ideas of laminar currents and geometric intersections \cite{MR1207478,MR2132264,MR2219241}.

	In \cite{MR4635340}, Cantat and Dujardin use their Ahlfors-Nevanlinna characterizations of $T^\pm$ combined with Lemma 8.8 to prove that $T^\pm$ are so-called \emph{strongly approximable} currents. To us, this means we may choose sequences of uniformly laminar currents $T^+_r, T^-_r$ such that, as $r\to 0$, $T^+_r \nearrow T^+$, $T^-_r \nearrow T^-$ and $T^+_r \wedge T^-_r \nearrow T^+ \wedge T^-$.

	Let $i\in \N$. For $r$ small, $(T^+_r \wedge T^-_r)(Q_i) > 0$. Now $\mathcal L^{s/u}_i$ are closed laminations on $N(Q_i)$; the analytic continuation theorem of Dujardin \cite[Theorem~1.1]{MR2132264} gives that the restrictions of $T^{+}_r$ to $\mathcal L^{s}_i$ (resp. $T^{-}_r$ to $\mathcal L^{u}_i$) are uniformly laminar, although a priori they may be zero.

	By the argument in Appendix A of \cite{MR2219241}, the restriction $\restr{(T^+_r \wedge T^-_r)}{Q_i} = \restr{T^+_r}{\mathcal L^s_i}\wedge \restr{T^-_r}{\mathcal L^u_i}$. Here, Proposition \ref{lem:last2} may be used as input in place of the convergence result $f_*^n(\psi C)/\lambda^n \to T_f^{-}\int_X(T_f^+ \wedge \psi C)$. 

Taking $r\to 0$, we have
\begin{align*}
	\restr{T_r^+}{\mathcal L^s_i}\wedge \restr{T_r^-}{\mathcal L^u_i}
	\nearrow \restr{T^+}{\mathcal L^s_i}\wedge \restr{T^-}{\mathcal L^u_i} \hspace{.5 cm}\text{and}\hspace{.5 cm}	\restr{(T_r^+ \wedge T_r^-)}{Q_i} \nearrow \restr{(T^+ \wedge T^-)}{Q_i}.
\end{align*}
Thus
\begin{align*}
	\restr{(T^+ \wedge T^-)}{Q_i}
	=
	\restr{T^+}{\mathcal L^s_i}\wedge \restr{T^-}{\mathcal L^u_i}.
\end{align*}
Now each $\restr{T^+}{\mathcal L^s_i}\wedge \restr{T^-}{\mathcal L^u_i}$ is a geometric intersection of uniformly laminar currents on $N(Q_i)$ and hence is locally a product measure in the coordinates of the Pesin box. Hence $T^+ \wedge T^-$ is a countable sum of local product measures; it has local product structure.

\section{Questions of rigidity}
	A natural question raised by this work is when, if ever, the measure $m$ arises from a stationary measure $\nu$ on $X$ (in the sense of \cite{MR884892}, \cite{MR1369243}). One checks that this occurs if and only if
	\begin{align*}
		\omega_+ \mapsto \int_{\omega_-}m_{(\omega_-, \omega_+)}\,d\mu^\N(\omega_-)
	\end{align*}
	 is almost surely equal to $\nu$. In the case that $X$ is a K3 surface, assuming $\mu$ is finitely supported and there are no $\Gamma_\mu$-invariant algebraic curves, by work of Roda \cite{roda2024classifying} every hyperbolic stationary measure is invariant. Thus $\nu$ as above would be invariant and equal to $m_\omega$ almost surely. 
	 
	The measure $m$ arises from a stationary measure in at least one case, namely when $(X,\Gamma_\mu)$ is what we call a \emph{simultaneous Kummer example}. Up to birational modification, a Kummer example is a pair $(X,f)$ in which $X$ is a complex torus and $f$ lifts to an affine transformation of the universal cover \cite[Definition~1.3]{MR4071328}. Given a subgroup $\Gamma\subset \Aut(X)$, we call the pair $(X, \Gamma)$ a \emph{simultaneous Kummer example} if $(X, f)$ is a Kummer example for each $f\in \Gamma$ and $\pi, \epsilon$ as in Definition 1.3 of \cite{MR4071328} can be chosen independently of $f$. It is enough to check this condition on a generating set for $\Gamma$.  

	For every Kummer example, the image of the Haar measure from the torus provides a nonzero $\Aut(X)$-invariant measure on $X$ which is absolutely continuous with respect to volume. We call this normalized measure $Leb$. If $(X,\Gamma_\mu)$ is a simultaneous Kummer example, $m_\omega = Leb$ almost surely, and $m = \mu^\Z \otimes Leb$. Hence, $m$ is induced by $Leb$ when $Leb$ is viewed as a stationary measure on $X$. In fact, this is stronger since $Leb$ is actually invariant.
	\begin{ques}\label{ques1}
	If $m$ arises from a stationary measure on $X$, is $(X, \Gamma_\mu)$ a simultaneous Kummer example?
\end{ques}
In the spirit of Cantat--Dupont \cite{MR4071328} and Filip--Tosatti \cite{MR4629008}, we have the following weaker question, which already seems approachable using the techniques of Cantat--Dupont.
\begin{ques}\label{ques2}
	Suppose $m$ arises from a stationary measure $\nu$ on $X$ and $\nu$ is in the Lebesgue class. Is $(X, \Gamma_\mu)$ a simultaneous Kummer example?
\end{ques}
\end{appendices}

	\bibliographystyle{amsplain}
	\bibliography{bibliography}
	
\end{document}